\pdfoutput = 1
\documentclass[11pt]{article}
\usepackage[utf8]{inputenc}
\usepackage[margin=1in]{geometry}
\pdfoutput=1
\usepackage{mathrsfs}
\usepackage{enumitem}
\usepackage{fullpage}
\usepackage{xcolor}
\usepackage{dsfont}
\usepackage{graphicx}
\usepackage[font=small,skip=0pt]{caption}
\usepackage{amsmath, amssymb,amsthm}
\theoremstyle{plain}
\usepackage[labelfont=bf,font=sf]{caption}
\usepackage[font={sf,small},labelfont=md]{subfig}
\usepackage{hyperref}
\hypersetup{colorlinks=true,linkcolor=blue,citecolor=blue,urlcolor=blue} 

\newcommand{\E}{\mathbb{E}}

\newcommand{\N}{\mathbb{N}}
\renewcommand{\P}{\mathbb{P}}

\newcommand{\R}{\mathbb{R}}

\newcommand{\Z}{\mathbb{Z}}

\renewcommand{\AA}{\mathcal{A}}

\newcommand{\CC}{\mathcal{C}}
\newcommand{\DD}{\mathcal{D}}

\newcommand{\FF}{\mathcal{F}}
\newcommand{\GG}{\mathcal{G}}

\newcommand{\LL}{\mathcal{L}}

\newcommand{\RR}{\mathcal{R}}
\renewcommand{\SS}{\mathcal{S}}
\newcommand{\TT}{\mathcal{T}}

\newcommand{\WW}{\mathcal{W}}

\newcommand{\AAA}{\mathscr{A}}

\newcommand{\GGG}{\mathscr{G}}

\newcommand{\PPP}{\mathscr{P}}

\newcommand{\XXX}{\mathscr{X}}

\newcommand{\given}{\,|\,}
\newcommand{\bgiven}{\,\big|\,}
\newcommand{\Bgiven}{\,\Big|\,}
\newcommand{\bbgiven}{\,\,\bigg|\,\,}
\newcommand{\BBgiven}{\,\,\Bigg|\,\,}
\newcommand{\supp}{\textnormal{supp}}

\newcommand{\eps}{\varepsilon}
\newcommand{\one}{\mathds{1}}
\newcommand{\bone}{\mathbf{1}}
\newcommand{\la}{\lambda}
\newcommand{\normal}{\mathsf{N}}

\def\sT{{\sf T}}

\newcommand{\pto}{\stackrel{p}{\longrightarrow}}
\newcommand{\dto}{\stackrel{d}{\longrightarrow}}

\newcommand{\ga}{\gamma}

\newcommand{\beq}{\begin{equation}}
\newcommand{\eeq}{\end{equation}}
\newcommand{\beqn}{\begin{equation*}}
\newcommand{\eeqn}{\end{equation*}}
\newcommand{\bpsi}{\boldsymbol{\psi}}
\newcommand{\bsigma}{\boldsymbol{\sigma}}
\newcommand{\bsigmas}{\boldsymbol{\sigma}^\star}
\newcommand{\ux}{\underline{x}}
\newcommand{\ui}{\underline{i}}
\newcommand{\uj}{\underline{j}}
\newcommand{\bsig}{\boldsymbol{\underline{\sigma}}}

\newcommand{\sig}{\underline{\sigma}}
\newcommand{\utau}{\underline{\tau}}
\newcommand{\iid}{\stackrel{i.i.d.}{\sim}}
\newcommand{\bG}{\boldsymbol{G}}
\newcommand{\wbG}{\boldsymbol{\widetilde{G}}}
\newcommand{\bGH}{\boldsymbol{G}^\star_{\sf HSBM}}

\newcommand{\bom}{\boldsymbol{\omega}}
\newcommand{\bm}{\boldsymbol{m}}
\newcommand{\wN}{\widetilde{N}}
\newcommand{\bN}{\Breve{N}}
\newcommand{\wnu}{\widetilde{\nu}}
\newcommand{\bnu}{\Breve{\nu}}
\newcommand{\bX}{\boldsymbol{X}}
\newcommand{\bZ}{\boldsymbol{Z}}
\newcommand{\bL}{\boldsymbol{L}}

\newcommand{\bR}{\boldsymbol{R}}
\newcommand{\ks}{{\tiny\textsf{\textup{KS}}}}

\newcommand{\DKL}{D_{\textup{KL}}}
\def\sT{{\mathsf T}}
\def\HSBM{{\sf HSBM}}
\def\ER{{\sf ER}}
\def\Unif{{\sf Unif}}
\def\bal{{\sf bal}}
\def\Poi{{\sf Poi}}
\def\Ber{{\sf Ber}}
\def\ind{{\sf ind}}
\def\co{{\sf cor}}
\def\de{{\rm d}}

\newcommand{\ot}{\otimes}
\newcommand{\wt}[1]{\widetilde{#1}}
\newcommand{\wh}[1]{\widehat{#1}}
\newcommand{\wb}[1]{\overline{#1}}

\newcommand{\pr}{M_0}

\DeclareMathOperator{\Var}{Var}

\DeclareMathOperator{\id}{Id}

\DeclareMathOperator{\tr}{tr}
\DeclareMathOperator{\diag}{diag}
\DeclareMathOperator{\Eig}{Eig}
\DeclareMathOperator{\proj}{proj}

\DeclareMathOperator*{\argmax}{arg\,max}

\DeclareMathOperator{\tv}{TV} 
\DeclareMathOperator{\Fr}{F}

\newtheorem{thm}{Theorem}[section]
\newtheorem*{thm*}{Theorem}
\newtheorem{prop}[thm]{Proposition}
\newtheorem*{prop*}{Proposition}
\newtheorem{cor}[thm]{Corollary}
\newtheorem{lemma}[thm]{Lemma}
\newtheorem*{lemma*}{Lemma}

\newtheorem{fact}[thm]{Fact}

\theoremstyle{definition}
\newtheorem{defn}[thm]{Definition}

\newtheorem{remark}[thm]{Remark}

\title{Weak recovery, hypothesis testing, and mutual information in stochastic block models and planted factor graphs}
\date{}

\author{
Elchanan Mossel\thanks{Department of Mathematics, Massachusetts Institute of Technology. Email: \textup{\tt elmos@mit.edu}} \and 
Allan Sly\thanks{Department of Mathematics, Princeton University. Email: \textup{\tt asly@math.princeton.edu}}\and 
Youngtak Sohn \thanks{Division of Applied Mathematics, Brown University. Email: \textup{\tt youngtak\_sohn@brown.edu}}}
\begin{document}

\maketitle

\begin{abstract}
The stochastic block model is a canonical model of communities in random graphs.
It was introduced in the social sciences and statistics as a model of communities, and in theoretical computer science as an average case model for graph partitioning problems under the name of the 
``planted partition model.'' Given a sparse stochastic block model, the two standard inference tasks are: (i) Weak recovery: can we estimate the communities with non-trivial overlap with the true communities? (ii) Detection/Hypothesis testing: can we distinguish if the sample was drawn from the block model or from a random graph with no community structure with probability tending to $1$ as the graph size tends to infinity?

In this work, we show that for sparse stochastic block models, the two inference tasks are equivalent except at a critical point. That is, weak recovery is information theoretically possible if and only if detection is possible.  We thus find a strong connection between these two notions of inference for the model. We further prove that when detection is impossible, an explicit hypothesis test based on low-degree polynomials in the adjacency matrix of the observed graph achieves the optimal statistical power. This low-degree test is efficient as opposed to the likelihood ratio test, which is not known to be efficient. Moreover, we prove that the asymptotic
mutual information between the observed network and the community structure exhibits a phase transition at the weak recovery threshold.

Our results are proven in much broader settings including the hypergraph stochastic block models and general planted factor graphs.
In these settings, we prove that the impossibility of weak recovery implies contiguity and provide a condition that guarantees the equivalence of weak recovery and detection.

\end{abstract}
\maketitle

\section{Introduction}\label{sec:intro}

The stochastic block model, or simply a block model, is a random graph model generalizing the famous Erdos-Renyi random graph~\cite{ErdosRenyi:60} and is a special case of 
inhomegnuous random graphs~\cite{BoJaRi:07}. 
It has been studied extensively in statistics and the social sciences as a model of communities
\cite{HoLaLe:83, SnijdersNowicki:97,BickelChen:09,RoChYu:11} 
 and in computer science as a model to study the average case behavior of clustering algorithms~\cite{DyerFrieze:89,JerrumSorkin:98,CondonKarp:01,McSherry:01,CojaOghlan:10}.  
 In the last decade starting with the work of Decelle et al. \cite{DKMZ:11}, the sparse block model has been extensively studied with fascinating connections to belief propagation and non-backtracking random walks~\cite{Krzakala_etal:13, BoLeMa:15}. 
Furthermore, the block model is one of the canonical models in high-dimensional statistics which is believed to exhibit a statistical-computational gap, meaning a gap between what is achievable information theoretically and what is achievable with known computationally efficient algorithms.

The two central inference tasks related to the sparse block models are {\em weak recovery} and {\em detection}. Weak recovery means estimating communities with non trivial overlap with the true communities. Detection refers to designing a statistical test that, from a single sample, achieves vanishing Type I and Type II errors in determining whether the sample is drawn from the block model or from a {\em null model} without community structure -- that is, a standard random graph model with the same average degree. Numerous papers including~\cite{MoNeSl:15, BMNN:16, CKPZ:18, CEJKK:18, CKM20} have studied 
the information theoretic thresholds for both weak recovery and detection. Notably, it has been shown that, for the two-community case~\cite{MoNeSl:15} and the disassortative case~\cite{CEJKK:18}, the thresholds for weak recovery and detection coincide, albeit with completely different proofs in each case. This naturally leads to the question:
\begin{quotation}\centering{\emph{Is weak recovery information theoretically equivalent to detection in general?}}
\end{quotation}
The first objective of this paper is to answer this question for sparse block models, and in greater generality, for {\em planted factor models}~\cite{CKPZ:18} which encompass labeled/weighted stochastic block models~\cite{HLM12,MLX15, XJL20} and hypergraph stochastic block models~\cite{ACKZ15,GD17a,GD17b, PZ21, SZ22}. We prove that the information theoretic possibility of detection implies the possibility of weak recovery, and find a condition under which the converse holds.

The second objective is to study {\em hypothesis testing}, a more classical task than detection, involving tests between two simple hypotheses. Notably, the impossibility of detection merely indicates that there is no sequence of statistical tests that can distinguish the block model from the null model with vanishing Type I and Type II errors. This leads us to the question:
\begin{quotation}
    \centering{\emph{Can we characterize the asymptotic power of the optimal (likelihood ratio) test at significance level $\alpha\in (0,1)$? If so, can the optimal power achieved efficiently?}}
\end{quotation}
In this paper, we explicitly characterize the asymptotic power of the likelihood ratio test. Moreover, we prove that it can be achieved efficiently by a test based on low-degree polynomials of the adjacency matrix when the weak recovery, or equivalently detection, is impossible. We also investigate the relationship between these statistical tasks and the mutual information between the observed graph and the community structure.

\subsection{Our results for symmetric block models}

To illustrate our general results, we first consider the well-studied setup of sparse symmetric block models with $q\geq 2$ communities. Let $V$ be a vertex set of size $n\geq 1$. For parameters $a,b>0$, let $\bG^\star \sim \GG(n,q,\frac{a}{n},\frac{b}{n})$ denote the model of random graphs in which each vertex $u\in V$ is assigned a label $\bsigma^\star_u\in \{1,2,...,q\}$ uniformly at random, and then each possible edge $(u,v)$ is included with probability $\frac{a}{n}$ if $\bsigma^\star_u=\bsigma^\star_v$ and with probability $\frac{b}{n}$ if $\bsigma^\star_u \neq \bsigma^\star_v$. The average degree of $\bG^\star$ equals $d\equiv \frac{a+(q-1)b}{q}$. The corresponding random graph with no community structure is $\bG\sim \GG_{\sf ER}(n,\frac{d}{n})$, where $\GG_{\sf ER}(n,\frac{d}{n})\equiv \GG(n,q,\frac{d}{n},\frac{d}{n})$ is the sparse Erdos-Renyi graph with average degree $d$. Observe that the parameters $(a,b)$ is in one-to-one correspondence with the average degree $d$ and the parameter 
\[
\la\equiv \frac{a-b}{a+(q-1)b}\in \Big(-\frac{1}{q-1}\,,\,1\Big)\,.
\]
Here, $\la$ corresponds to the second eigenvalue of a certain stochastic matrix, namely the transition matrix from the color of a vertex to its neighbor or child (cf. \eqref{eq:transition}). This parametrization is emphasized in several works on the block model~\cite{Massoulie:14,MoNeSl:18,BoLeMa:15,AbbeSandon:15,AbbeSandon:18} and on phylogeny reconstruction problem~\cite{Mossel:04,DaMoRo:11,MoRoSl:11,RochSly:17}. We are interested in high-dimensional setting where $n\to\infty$ as $q, d,\la$ are fixed.

A first observation is that by fixing $\la\in (-\frac{1}{q-1},1)$ and increasing $d$ from $0$ to $\infty$, the task of weak recovery is monotone in $d$ since larger $d$ amounts to more information. Indeed given a sample from the model with a certain $d$ by deleting each edge independently with probability $t\in [0,1]$ we obtain a perfect sample of the model with the same $\la$ and the new degree $t\cdot d$. Thus, there exists a critical point $d_{\ast}(q,\la)$, which we call the \emph{weak recovery threshold}, at which the weak recovery becomes information theoretically possible (see Definition~\ref{def:weak:recovery}). Our result for the block model shows that the weak recovery and detection is equivalent possibly except for the critical point $d_{\ast}(q,\la)$.

\begin{thm}\label{thm:SBM} \textnormal{(Special case of Theorem~\ref{thm:HSBM:contiguity})} Consider the sparse symmetric block model $\bG^\star\sim \GG(n,q,\frac{a}{n},\frac{b}{n})$ with $q\geq 2$ communities and parameters $a,b>0$. For average degree $d\equiv \frac{a+(q-1)b}{q}$ below the weak recovery threshold $d<d_{\ast}(q,\la)$, the following holds.
  \begin{enumerate}[label=\textup{(\arabic*)}]
       \item \label{item:contiguity} $\bG^\star\sim \GG(n,q,\frac{a}{n},\frac{b}{n})$ is mutually contiguous with the Erdos-Renyi graph $\bG\sim \GG_{\sf ER}(n,\frac{d}{n})$. That is, for a sequence of events $(\AAA_n)_{n\geq 1}$, $\P(\bG^\star\in \AAA_n)\to 0$ as $n\to\infty$ if and only if $\P(\bG\in \AAA_n)\to 0$ as $n\to\infty$. Thus, detection is impossible, i.e. there is no test that distinguishes $\bG^\star$ from $\bG$ with power tending to $1$ as $n\to\infty$. Moreover,  there is no consistent estimator for $a,b>0$.

       \item \label{item:mutual:info:below} The mutual information between $\bG^\star$ and the community structure $\bsig^\star$ per-vertex  $I(\bG^\star,\bsig^\star)/n$ converges to $\frac{d}{2}I_0(q,\la)$ as $n\to\infty$, where $I_0(q,\la)$ is the mutual information of $\bsigma_{\la}\in [q]$ and $\bsigma\in[q]$ defined in Eq.~\eqref{eq:mutual:single:vertex} below.
  \end{enumerate}
  On the contrary, above the weak recovery threshold $d>d_{\ast}(q,\la)$, 
  \begin{enumerate}[label=\textup{(\arabic*)}]\setcounter{enumi}{2}
    \item \label{item:mutual:orthogonal} $\bG^\star\sim \GG(n,q,\frac{a}{n},\frac{b}{n})$ is not mutually contiguous to the Erdos-Renyi graph $\bG\sim \GG_{\sf ER}(n,\frac{d}{n})$. Moreover, there exists a sequence of events $(\AAA^{\ast}_{n})_{n\geq 1}$ such that after passing to a subsequence, $\P(\bG\in \AAA^{\ast}_n)\leq 2e^{-\eta n}$ and $\P(\bG^\star\in \AAA^{\ast}_n)\geq 1-2e^{-\eta n}$ hold for $\eta>0$ not depending on $n$. Thus, detection is possible, and along a subsequence, there exists a statistical test that distinguishes $\bG^\star$ from $\bG$ with power tending to $1$ exponentially fast as $n\to\infty$.
    \item \label{item:mutual:info:above} $\liminf_{n\to\infty} \frac{1}{n}I(\bG^\star,\bsig^\star)<\frac{d}{2}I_0(q,\la)$ holds.
  \end{enumerate}
\end{thm}

Our results significantly advance the understanding of stochastic block models by unifying and extending previous findings. While earlier studies have established mutual contiguity in specific parameter regimes such as the disassortative case where $a < b$ through combined results in \cite{CKPZ:18, CEJKK:18, CKM20}, the case of $q = 2$ communities \cite{MoNeSl:15}, and where $d$ is sufficiently small \cite{BMNN:16}. These works each have their own limitations as they rely on methods like the "interpolation technique," which fails for the assortative case $a > b$, or the "second moment method," which does not hold when $d$ is sufficiently close to $d_{\ast}(q,\lambda)$. Our theorems overcome these limitations by covering the maximal range of parameters. Furthermore, we extend our results to a broader context by applying them to hypergraph stochastic block models, accommodating interactions of order $k \geq 2$ among communities and allowing for non-symmetric connection probabilities between different communities. For the most general formulations of our results within the framework of planted factor models, we refer to Section~\ref{sec:general:results}.

The proof of Theorem~\ref{thm:SBM} hinges on a novel equivalence between the impossibility of weak recovery and the \textit{near-orthogonality} of the samples drawn from the posterior (see Theorem~\ref{thm:equiv}). Specifically, the key insight in proving Theorem~\ref{thm:SBM}-\ref{item:contiguity} is that below the weak recovery threshold $d<d_{\ast}(q,\la)$, appropriately truncating the likelihood ratio based on \textit{overlaps} of posterior samples reduces the second moment of the likelihood ratio by a multiplicative factor of $e^{\Omega(n)}$, while affecting the first moment by only $1-o(1)$. We integrate this insight with other techniques such as small subgraph conditioning from combinatorics~\cite{Janson95random, Wormald99models} and I-MMSE relation from information theory~\cite{guo2011estimation}. We believe our proof techniques illuminate the relationship between weak recovery and detection and have broad applicability in studying the equivalence of these two tasks in high-dimensional statistics. For detailed proofs, we refer to Section~\ref{sec:proof}.

Note that Theorem~\ref{thm:SBM} does not cover the critical case $d=d_{\ast}(q,\la)$. This is not a coincidence: there are cases where the equivalence of detection and weak recovery fails for $d=d_{\ast}(q,\la)$! For example, $d_{\ast}(q,\la)$ matches the \emph{Kesten-Stigum threshold} $d_{\ks}=\la^{-2}$ when $q=2$~\cite{MoNeSl:15} or when $q\in \{3,4\}$ and $d$ is large enough~\cite{mossel23exact}. In these cases, weak recovery is impossible at the Kesten-Stigum threshold~\cite{MoNeSl:15, mossel23exact}. In contrast, it is possible to distinguish the block model from the Erdos Renyi graph with probability approaching $1$ by counting short cycles as the asymptotic distribution of the number of cycles of fixed length is Poisson with known parameters~\cite{BoJaRi:07}. 

The quantity $I_0(q,\la)$ in Theorem~\ref{thm:SBM} can be characterized as follows. Let $\bsigma\in \Unif(\{1,\ldots,q\})$ be a uniform community assignment. Given $\bsigma$, let $\bsigma_{\la}$ a sample from the conditional distribution
\begin{equation}\label{eq:transition}
    \P(\bsigma_\la=j\given \bsigma=i)= 
    \begin{cases}
    \la + \frac{1-\la}{q} &\text{if $i=j$}\\
    \frac{1-\la}{q} &\text{otherwise}.
    \end{cases}
\end{equation}
Then, $I_0(q,\la)\equiv I(\bsigma_{\la},\bsigma)$ is the mutual information between $\bsigma_{\la}$ and $\bsigma$:
\beq\label{eq:mutual:single:vertex}
I_0(q,\la)\equiv I(\bsigma_{\la},\bsigma)= \frac{1+(q-1)\la}{q}\log\big(1+(q-1)\la\big)+\frac{(q-1)(1-\la)}{q}\log(1-\la)\,, 
\eeq
Thus, Theorem~\ref{thm:SBM}-\ref{item:mutual:info:below} establishes that for all $q\geq 2$ and $\la\in (-\frac{1}{q-1},1)$, below the weak recovery threshold $d<d_{\ast}(q,\la)$, the asymptotic per-vertex mutual information between $\bG^\star$ and $\bsig^\star$ admits a ``single-letter'' characterization. Informally, this means that the mutual information is fully determined by local interactions between neighboring vertices. In contrast, when $d>d_{\ast}(q,\la)$, global correlations in $\bsigma$ given $\bG^\star$ emerge, adding an extra term to the mutual information. For a more formal statement, we refer to Lemma~\ref{lem:KL:mutual:info:free:energy} and Proposition~\ref{prop:free:energy} below. In particular, in the case of $q=2$ communities with $d\la^2<1$, our result improves upon that of \cite{DAM16}, where they required the average degree diverge $d\to\infty$. We further prove in Theorem~\ref{thm:HSBM:contiguity} that such single letter characterization holds for (possibly non-symmetric) hypergraph stochastic block model below the weak recovery threshold. Moreover, we provide a general condition under which the phase transition of per-vertex mutual information occurs at the weak recovery threshold.
\subsubsection{Hypothesis testing for the existence of communities}

Theorem~\ref{thm:SBM} establishes that the weak recovery threshold $d_{\ast}(q,\la)$ marks the sharp phase transition of the following hypothesis testing from a single observation $G$
\beq\label{eq:hypothesis:test}
{\sf H_0}: G\sim \GG_{\sf ER}\Big(n,\frac{d}{n}\Big)\;\;\;\;\textnormal{vs.}\;\;\;\;{\sf H_1}: G\sim \GG\Big(n,q,\frac{a}{n},\frac{b}{n}\Big)
\eeq
In particular, if $d<d_{\ast}(q,\la)$, then there is no perfect sequence of statistical tests (i.e. tests have asymptotically vanishing Type~I and Type~II errors), whereas if $d>d_{\ast}(q,\la)$ there exists a test which achieves vanishing Type~I and Type~II error exponentially fast at least along a subsequence. 

However, for $d<d_{\ast}(q,\la)$, Theorem~\ref{thm:SBM} does not provide any information on the largest power achieved by the likelihood ratio (LR) test. We next characterize the asymptotic power obtained by LR test below the weak recovery threshold. Let us denote the likelihood ratio of the block model $ \bG^\star\sim  \GG(n,q,\frac{a}{n},\frac{b}{n})$ with respect to Erdos Renyi graph $\bG\sim \GG_{\sf ER}(n,\frac{d}{n})$ by
\begin{equation*}
    \LL_n(G)=\frac{\P(\bG^\star=G)}{\P(\bG=G)}\,.
\end{equation*}
\begin{thm}\label{thm:optimal:power} \textnormal{(Special case of Theorem~\ref{thm:optimal:power:HSBM})}
For average degree $d\equiv \frac{a+(q-1)b}{q}$ below the weak recovery threshold $d<d_{\ast}(q,\la)$, the likelihood ratio under null $\bG\sim \GG_{\sf ER}(n,\frac{d}{n})$ converges in distribution
\begin{equation}\label{eq:thm:likelihood:conv}
    \LL_n(\bG)\dto \boldsymbol{\LL}_{\infty}
    :=\prod_{\ell \geq 3}\Big\{\big(1+(q-1)\la^{\ell}\big)^{\bX_{\ell}}e^{-\frac{(q-1)(d\lambda)^{\ell}}{2\ell}}\Big\}\,,
\end{equation}
where $(\bX_{\ell})_{\ell\geq 3}$ are independent Poisson random variables with mean $\E \bX_{\ell}=\frac{1}{2\ell}d^{\ell}$. Here, the random variable $\boldsymbol{\LL}_{\infty}\equiv \boldsymbol{\LL}_{\infty}(d,\la)$ satisfies the following.
\begin{enumerate}[label=\textup{(\arabic*)}]
\item For $d\in (0,\la^{-2})$, which is a superset of $(0,d_{\ast}(q,\la))$, $\boldsymbol{\LL}_{\infty}$ is well-defined (i.e. the infinite product converges a.s.) and has finite second moment.
\item If $1\leq d<\la^{-2}$ holds, then the random variable $\boldsymbol{\LL}_{\infty}$ does not have a point mass. In particular, in the regime $1\leq d<d_{\ast}(q,\la)$, the likelihood ratio test for the hypothesis testing task~\eqref{eq:hypothesis:test} at significance level $\alpha\in (0,1)$ has the asymptotic power $\beta_{\ast}(\alpha)\equiv \beta_{\ast}(\alpha;q,d,\la)\in (0,1)$ defined by
 \beq\label{eq:def:power:1}
      \beta_{\ast}(\alpha)=\E\big[\boldsymbol{\LL}_{\infty}\one\{\boldsymbol{\LL}_{\infty}\geq C_{\alpha}\}\big]\,.
     \eeq
     Here, $C_{\alpha}>0$ is an arbitrary constant satisfying $\P(\boldsymbol{\LL}_{\infty}\geq C_{\alpha})=\alpha$,
     whose existence is guaranteed.
\end{enumerate}
\end{thm}
The threshold $d_{\ks}:=\la^{-2}$ is called the {\em Kesten-Stigum} (KS) threshold, which was first discovered in the context of the Markov processes on trees in the seminal work of Kesten and Stigum~\cite{KestenStigum:66}. In the context of sparse block model, it was first conjectured in~\cite{DKMZ:11} that above the KS threshold $d\la^2>1$, the weak recovery can be done in a computationally efficient way. This conjecture was proved for symmetric block model with $q=2$ communities by~\cite{Massoulie:14,MoNeSl:18,BoLeMa:15} and for general block model by~\cite{AbbeSandon:18}. In particular, the fact that the weak recovery threshold $d_{\ast}(q,\la)$ is at most the KS threshold $d_{\ks}$ for any $q\geq 2$ and $\la\in (-\frac{1}{q-1},1)$ follows from the result~\cite{AbbeSandon:18}.

In the subcritical regime $d<1$, the random variable $\boldsymbol{\LL}_{\infty}$ is a discrete random variable supported on the countable set 
    \beqn
    \Big\{z: z=(1-d\la)^{\frac{q-1}{2}}\prod_{\ell \geq 1}\big(1+(q-1)\la^{\ell}\big)^{x_{\ell}} \textnormal{ for some $x=(x_{\ell})_{\ell\geq 1}$ with finitely many non-zero elements.}\Big\}\,.
    \eeqn
    Thus, for $d<1$ the weak convergence of the likelihood ratio obtained in Theorem~\ref{thm:optimal:power} only characterizes the asymptotic power of the \textit{non-randomized} LR test for the hypothesis test~\eqref{eq:hypothesis:test} (see also Corollary~\ref{cor:power:factor}). This is because the LR test might require randomization since the law of $\boldsymbol{\LL}_{\infty}$ is atomic for $d<1$. We refer to \cite[Section 3.2]{lehmann2005testing} for the description of randomization in LR tests.

\subsubsection{Computational aspects of hypothesis testing}
We next discuss the computational aspects in the hypothesis testing \eqref{eq:hypothesis:test}. Although the LR test achieves largest power for hypothesis test~\eqref{eq:hypothesis:test} at any significance level by the Neyman-Pearson lemma, it is far from clear how to efficiently approximate the likelihood ratio $\LL_n(G)$. In fact, $\LL_n(G)$ can be interpreted as a certain {\em{partition function}} from statistical physics, and it is often NP-hard (or even \#P-hard) to approximate partition functions, 
see e.g.~\cite{JerrumSinclair:93, Istrail00, GoldbergJerrum:07}. 
Note that given an assignment of the communities $\sig=(\sigma_v)_{v\in V}\in [q]^V$ and $G=(V,E)$, respectively denote the number of monochromatic edges and non-edges by
\beqn
e_{0}(G,\sig):=\big|\big\{(u,v)\in E: \sigma_u=\sigma_v\big\}\big|\,,\;\;\;\;\;\; \wt{e}_{0}(G,\sig):=\big|\big\{(u,v)\notin E: \sigma_u=\sigma_v\big\}\big|\,.
\eeqn
Then, it is straightforward to calculate
\beqn
\begin{split}
\LL_n(G)
&=q^{-n}\sum_{\sig\in [q]^V}\bigg(\frac{n-a}{n-(a+(q-1)b/q}\bigg)^{\wt{e}_0(G,\sig)}\bigg(\frac{n-b}{n-(a+(q-1)b/q}\bigg)^{\binom{n}{2}-|E|-\wt{e}_0(G,\sig)}\\
&~~~~~~~~~~~~~~~~~~~~~~~~~~~~~~~~~~~~~~~~\bigg(\frac{qa}{a+(q-1)b}\bigg)^{e_0(G,\sig)}\bigg(\frac{qb}{a+(q-1)b}\bigg)^{|E|-e_0(G,\sig)}\,,
\end{split}
\eeqn
and in particular, summing over $\sig \in [q]^V$ makes the naive computation of $\LL_n(G)$ intractable.

In contrast, we construct a {\em computationally efficient} test achieving  the (asymptotically) largest power. Given a graph $G$ and an integer $\ell\geq 3$, let $X_{\ell}(G)$ be the number of cycles of length $\ell$. Then, consider the statistic
\beqn
    \TT_{n}(G) :=\prod_{3\leq \ell \leq K_n}\big(1+(q-1)\la^{\ell}\big)^{X_{\ell}(G)}\,,
\eeqn
where the truncation parameter $K_n$ is chosen to diverge slowly to infnity, say $1\ll K_n=O(\log\log n)$.

We note that if $G$ is drawn from the sparse block model or the sparse Erdos Renyi graph, then the condition $K_n=O(\log\log n)$ guarantees that the computation of $\TT_n(G)$ can be done in nearly linear time (up to log factors) with high probability. This is because when $G$ is sparse, the number of cycles up to length $\log\log n$ can be counted by looking at depth $\log\log n$ neighborhood of each nodes, each of which has at most $(Cd)^{\log\log n}$ neighbors with high probability.

In addition, note that $\TT_n(G)$ is a degree $K_n$ polynomial in the indicator of the edges of $G$. Our next theorem establishes that a test based on this low-degree polynomial $\TT_n(G)$ is asymptotically most powerful below the weak recovery threshold.
\begin{thm}\label{thm:optimal:cycle:test} \textnormal{(Special case of Theorem~\ref{thm:optimal:cycle:test:HSBM})}
Let $K_n=O(\log\log n)$ and $K_n\to\infty $ as $n\to\infty$. For a significance level $\alpha\in (0,1)$, consider the test $\phi_{n,\alpha}(\cdot)$ which rejects the null ${\sf H_0}$ in Eq.~\eqref{eq:hypothesis:test} with probability 
\beqn
\phi_{n,\alpha}(G):=
\begin{cases}
    1 & \;\;\;\;\textnormal{if}\;\;\;\;\; \TT_n(G)> C'_{n,\alpha}\,;\\
    0 & \;\;\;\;\textnormal{otherwise}\,.
\end{cases}
\eeqn
Here, the constants $C_{n,\alpha}'>0$ is chosen so that we have 
$\P(\TT_n(\bG)>C^\prime_{n,\alpha})\leq \alpha\leq \P(\TT_n(\bG)\geq C^\prime_{n,\alpha})$, where $\bG\sim \GG_{\sf ER}(n,\frac{d}{n})$. Then, for $1\leq d<\la^{-2}$, the test $\phi_{n,\alpha}$ achieves the power $\beta_{\ast}(\alpha)$ defined in Eq.~\eqref{eq:def:power:1}. That is, under the alternative $\bG^\star\sim \GG(n,q,\frac{a}{n},\frac{b}{n})$, we have
\beqn
\E\phi_{n,\alpha}(\bG^\star)\to \beta_{\ast}(\alpha)\;\;\;\;\textnormal{as}\;\;\;\; n\to\infty\,.
\eeqn
In particular, in the regime $1\leq d <d_{\ast}(q,\la)$, the test $\phi_{n,\alpha}(\cdot)$ is asymptotically most powerful for the hypothesis test~\eqref{eq:hypothesis:test}.
\end{thm}
For $q\geq 5$ communities, it is known that for small enough $\la$, there is a gap between the weak recovery threshold $d_{\ast}(q,\la)$ and the KS threshold $d_{\ks}\equiv \la^{-2}$~\cite{BMNN:16, AbbeSandon:16}. In this case, Theorems~\ref{thm:SBM} and~\ref{thm:optimal:cycle:test} uncover an important phenomenon: the weak recovery threshold $d_{\ast}(q,\la)$ marks the {\em local to global phase transition} of the block model. Specifically, $d<d_{\ast}(q,\la)$, the existence of the community structure is fully captured by the {\em local} structure of the graph, namely the number of short cycles. On the contrary, in the regime $d_{\ast}(q,\la)<d<d_{\ks}$, a {\em global} structure emerges, which cannot be described by the number of short cycles. Indeed, Theorem~\ref{thm:SBM}-\ref{item:mutual:orthogonal} shows that in the latter regime, the asymptotic power of the LR test must equal $1$ (cf. Neyman-Pearson lemma) along a subsequence. However,  Theorem~\ref{thm:optimal:cycle:test} implies that the test $\phi_{n,\alpha}$ based on a low-degree polynomial achieves limited power $\beta_{\ast}(\alpha)<1$ in the intermediate regime $d\in (d_{\ast}(q,\la), d_{\ks})$. This discrepancy thus demonstrates {\em statistical-computational gap} in hypothesis testing. A natural question would be to see if the low-degree test $\phi_{n,\alpha}$ achieves the optimal power among all computationally efficient tests in the regime $d\in (d_{\ast}(q,\la), d_{\ks})$ under a suitable complexity-theoretic conjecture, see e.g. \cite{MW23} for recent related work.

On the contrary, for $q=2$ communities and $q\in \{3,4\}$ communities with large enough average degree $d$, it is known that $d_{\ast}(q,\la)=d_{\ks}$ holds~\cite{MoNeSl:15, mossel23exact}. Thus, in these cases, our results yield that there is no statistical-computational gap in hypothesis testing at any significance level. Indeed, if $d<d_{\ast}(q,\la)=d_{\ks}$ holds, then the low-degree-polynomial test $\phi_{n,\alpha}(\cdot)$ is most powerful by Theorem~\ref{thm:SBM}. Moreover, it is known that above the KS threshold $d>d_{\ks}$, the detection is solvable by efficient algorithms (e.g. by counting short cycles~\cite{Massoulie:13,MoNeSl:18,abbe18survey} or by a semidefinite programming relaxation~\cite{BMR21}). Specifically, we obtain the following corollary by combining our results with the results of \cite{MoNeSl:15, mossel23exact}.

\begin{cor}\label{cor:q:3:4}
For $q=2$ communities or $q\in \{3,4\}$ and small enough $|\lambda|\leq \lambda_0$, where $\lambda_0>0$ is a universal constant, the results of Theorems~\ref{thm:SBM}, \ref{thm:optimal:power}, \ref{thm:optimal:cycle:test} hold with $d_{\ast}(q,\la)\equiv \la^{-2}$. Thus, in this case, for any average degree $d$ such that $d\geq 1$ and $d\neq \la^{-2}$ hold and any significance level $\alpha\in (0,1)$, there exists an efficient sequence of statistical tests $\varphi_{n,\alpha}(\cdot)$ that is asymptotically most powerful at level $\alpha$.
\end{cor}

\subsection{Proof techniques}

We now summarize the difficulties in proving one of our main results -- the convergence of the likelihood ratio $\LL_n(\bG)$ for $d<d_\ast(q,\la)$ in Theorem~\ref{thm:optimal:power} -- and explain the key insights that we use to overcome them. The implications from Theorem~\ref{thm:optimal:power} to Theorem~\ref{thm:SBM}-\ref{item:contiguity}, \ref{item:mutual:info:below} is standard. The results for $d>d_{\ast}(q,\la)$ in Theorem~\ref{thm:SBM}-\ref{item:mutual:orthogonal}, \ref{item:mutual:info:above} rely on a different technique, which involves differentiating $\E \log \LL_n(\bG^\star)$ with respect to $d>0$ (see Section~\ref{subsec:proof:mutual:info}), which we shall not focus here.

A standard approach to proving the convergence of $\LL_n(\bG)$ is through the second moment method supplemented with the small subgraph conditioning method by~\cite{Janson95random, Wormald99models}. Since $\E \LL_n(\bG)=1$ holds by change of measures, the second moment method requires that $\E\big(\LL_n(\bG)\big)^2=O(1)$. Indeed, \cite{BMNN:16} used this approach to show the convergence of $\LL_n(\bG)$. However, the second moment explodes exponentially fast, i.e.  $\E\big(\LL_n(\bG)\big)^2=e^{\Omega(n)}$, beyond the \emph{second moment threshold} $d_{\textsf{sec}}(q,\la)$ (see \cite[Proposition 1]{BMNN:16} which is strictly below $d_{\ast}(q,\la)$ except when $q=2$, where $d_{\textsf{sec}}(2,\la)=d_{\ast}(2,\la)=\la^{-2}$. Thus, the naive second moment approach fails to cover all $d<d_{\ast}(q,\la)$ when $q\geq 3$.

Our main insight is to relate the impossibility of weak recovery to the \emph{orthogonality} of two i.i.d. samples drawn from the posterior $\P(\bsig^\star=\cdot\given \bG^\star)$. To be precise,  given two samples $\sig^{\ell}\equiv (\sigma^{\ell}_v)_{v\in V}\in [q]^V, \ell=1,2$, define the \textit{overlap matrix} $R_{\sig^1,\sig^1}\equiv \big(R_{\sig^1,\sig^2}(i,j)\big)_{i,j\leq q} \in \R^{q\times q}$ as
\[
R_{\sig^1,\sig^2}(i,j)=\frac{1}{n}\sum_{v\in V}\one\big\{\sigma^1_v=i\,,\, \sigma^2_v=j\big\}\,.
\]
Concisely, the overlap matrix $R_{\sig^1,\sig^2}$ represents the empirical distribution of $(\sigma_v^1,\sigma_v^2)_{v\in V}$. In particular, the overlap matrix $R_{\sig^1,\sig^2}$ being close to $\pi\pi^{\sT}$, where $\pi=\Unif([q])$ denotes the prior of the communities, indicates that the two samples $\sig^1,\sig^2$ are {\em near-orthogonal}, i.e. they appear as if they are drawn independently from $\pi^{\otimes V}$. With this notation, we establish in Theorem~\ref{thm:equiv} that if the weak recovery is impossible ($d<d_{\ast}(q,\la)$) then as $n\to\infty$
\begin{equation}\label{eq:intro:orthogonal}
\E \big\langle \big\|R_{\sig^1,\sig^2}-\pi\pi^{\sT}\big\|_1\big\rangle_{\bG^\star} \to 0,
\end{equation}
where $\langle f(\sig^1,\sig^2) \rangle_{G}$ the expectation with respect to i.i.d. samples from the posterior $(\sigma^{\ell})_{\ell\geq 1}\iid \P(\sigma^\star=\cdot\given \bG^\star=G)$ (cf.~Eq.~\eqref{eq:def:expectation:posterior}). Intriguingly, we prove the reverse implication where \eqref{eq:intro:orthogonal} implies the impossibility of weak recovery, which may be of independent interest.

Armed with the orthogonality of posterior samples from Eq.~\eqref{eq:intro:orthogonal}, the crucial step is to truncate $\LL_n(G)$ to ensure that the two samples drawn from the posterior given $G$ are near-orthogonal. Namely, for a sequence of truncation parameters $(\eps_n)_{n\geq 1}$, we consider the truncated likelihood function $\LL^{\ast}_n(G)\equiv \LL^{\ast}_n(G;\eps_n)$ defined by
\[
\LL^{\ast}_n(G):=\LL_n(G)\one\Big\{\big\langle \big\|R_{\sig^1,\sig^2}-\pi\pi^{\sT}\big\|_1\big\rangle_{G}\leq \eps_n\Big\}\,.
\]
Observe that by a change of measure, the truncation affects the first moment by
\[
 \begin{split}
 \E\big| \LL_n^{\ast}(\bG)-\LL_n(\bG)\big|
&=\E \Big[\LL_n(\bG)\one\big\{\big\langle \big\|R_{\sig^1,\sig^2}-\pi\pi^{\sT}\big\|_1\big\rangle_{\bG}>\eps_n\big\}\Big]\\
&=\P\big(\big\langle \big\|R_{\sig^1,\sig^2}-\pi\pi^{\sT}\big\|_1\big\rangle_{\bG^\star}>\eps_n\big)\,,
 \end{split}
\]
which tends to $0$ by Eq.~\eqref{eq:intro:orthogonal} and Markov's inequality for $\eps_n$ that tends to $0$ sufficiently slowly. Thus, it follows that $\E \LL_n^\ast(\bG)=1-o(1)$.

In contrast to the first moment, the truncation drastically reduces the second moment. We prove in Proposition~\ref{prop:sec:mo:bound} that $\E \big[\LL_n^\ast(\bG)^2\big]=O(1)$ in the \textit{entire regime} $d<d_\ast(q,\la)$. Note that this contrasts with $\E \big[\LL_n(\bG)^2\big]=e^{\Omega(n)}$ for $d\in (d_{\textsf{sec}}(q,\la), d_{\ast}(q,\la))$ except for $q=2$.

We emphasize that although similar truncation techniques appear in \cite{CKPZ:18, CEJKK:18}, which study the disassortative case ($\lambda < 0$), there are two major differences in our approach. First, for the assortative case, the conditions \textbf{BAL} and \textbf{POS} used in those works fail, conditions that were crucial for characterizing $d_\ast(q,\lambda)$ via a variational formula. In particular, the condition \textbf{POS} enables the ``interpolation technique'' developed by \cite{guerra03} for spin glasses. Indeed, it was recently shown that the interpolation technique \textit{provably fails} for (a variant of) assortative stochastic block model in \cite{DM24}, and determining $d_{\ast}(q,\la)$ for $\la>0$ is a well-known open problem~\cite{DominiguezMourrat22}. Therefore, one cannot hope to modify the techniques of~\cite{CKPZ:18, CEJKK:18} in the assortative case. Second, the random variable under consideration in the works~\cite{CKPZ:18, CEJKK:18} is not the likelihood ratio $\LL_n(\bG)$, and it is rather the \textit{partition function} from statistical physics. In the notation introduced in Definition~\ref{def:planted} below, this is $Z_n(G)=\sum_{\sig\in [q]^V}\psi_G(\sig)\P(\bsig^\star=\sig)$ in the block model. While the partition function is more natural from the statistical physics point of view, for the assortative block model, it is far from clear how to relate $Z_n(G)$ with weak recovery threshold $d_{\ast}(q,\la)$. While the partition function is natural in statistical physics, for the assortative block model it is unclear how to relate $Z_n(G)$ to the weak recovery threshold $d_\ast(q,\lambda)$. One complication for $\lambda > 0$ is that $\sigma \mapsto \psi_G(\sigma)$ is maximized at configurations where all spins are equal, making it difficult to argue that the dominant contributions to quantities like the second moment come from \emph{near-orthogonal} configurations.

Instead, the equivalence of \textit{near-orthogonality} of the posterior samples and the impossibility of weak recovery we establish allows us to relate $d_{\ast}(q,\la)$ with the likelihood ratio function. From a high-level, in the expression of the likelihood ratio function $\LL_n(G)=\sum_{\sig\in [q]^V} \frac{\psi_{G}(\sig)}{\E[\psi_{\bG}(\sig)]}$ in the notations of Definition~\ref{def:planted}, the term $\E[\psi_{\bG}(\sig)]$ ``balances'' the $\psi_G(\sig)$ appropriately, enabling the second moment calculations to succeed. For this reason, likelihood ratio function $\LL_n(G)$ turns out to be more suitable than the partition function $Z_n(G)$ in second moment computations.

Finally, our proof incorporates additional ideas, such as establishing mutual orthogonality above the weak recovery threshold $d > d_\ast(q,\lambda)$ (without relying on the conditions \textbf{BAL} and \textbf{POS} from \cite{CEJKK:18}), and a resampling technique to show that contiguity implies the absence of consistent estimators for $a, b > 0$ (see Section~\ref{subsec:contiguity}).

\subsection{Further related work}
\label{sec:related}

\textbf{Stochastic block models and their hypergraph analogs:}~The stochastic block model, first introduced in \cite{HoLaLe:83}, has been studied extensively in statistics~\cite{SnijdersNowicki:97,BickelChen:09,RoChYu:11} and in computer science~\cite{DyerFrieze:89,JerrumSorkin:98,CondonKarp:01,McSherry:01,CojaOghlan:10}. The block model in the sparse regime became a major object of research due to the landmark paper~\cite{DKMZ:11}. In particular,~\cite{DKMZ:11} conjectured a different set of phase transitions within the sparse block model, where the Kesten-Stigum (KS) threshold~\cite{KestenStigum:66} plays a crucial role. Later, it was shown in a series of works~\cite{Massoulie:14,MoNeSl:18,BoLeMa:15,AbbeSandon:15,AbbeSandon:18} that above the KS threshold, there is an efficient algorithm that achieves both weak recovery and detection. The tightness/non-tightness of KS threshold for weak recovery was also studied by a number of papers~\cite{MoNeSl:15, AbbeSandon:16, BMNN:16, RiSeZd:19, mossel23exact}. We refer to the survey by Abbe~\cite{abbe18survey} for the developments of the block model for more references.

More recently, community detection in sparse hypergraph has also gained significant interest. When the average degree is at least of order $\Omega (\log^2 n)$, Ghoshdastidar and Dukkipati~\cite{GD17a} proved that a spectral algorithm achieves near-perfect recovery of the community structure. When the average degree is of order constant, \cite{ACKZ15} conjectured that a phase transition occurs at the KS threshold - above this threshold, efficient weak recovery is possible, and impossible below the threshold. Efficient recovery above the KS threshold was proved for hypergraph models for 2 communities by Pal and Zhu~\cite{PZ21} and general case by Stephan and Zhu~\cite{SZ22}. Gu and Polyanskiy~\cite{GP23} recently showed that the information theoretic threshold for weak recovery is the KS threshold within the $r$ uniform hypergraph stochastic block model for $r=3,4$ while the same is not true for $r \geq 7$.\\

\noindent\textbf{General factor models:}~The framework for factor models capture many standard  models in theoretical computer science such as k-SAT models and in statistical physics such as Potts models (see e.g. Chapter 9 in \cite{MezardMontanari:09}). The planted factor models was initially introduced to study the phase transitions of random constraint satisfaction problems~\cite{AC08, KZ09}. Later, Coja-Oghlan, Krzakala, Perkins, Zdeborov\'{a}~\cite{CKPZ:18} studied the information theoretic thresholds for planted factor models that satisfy certain convexity and balanced conditions. In particular,~\cite{CKPZ:18} obtained an explicit variational principle for the weak recovery threshold for the disassortative block models. Subsequently, Coja-Oghlan, Efthymiou, Jaafari, Kang, Kapetanopoulos~\cite{CEJKK:18} proved that in the disassortative case, the planted and the null models are mutually contiguous below the weak recovery threshold. The work~\cite{CEJKK:18} also introduced the Kesten-Stigum bound for planted factor models in the uniform prior $\pi=\Unif([q])$ case. The analysis of~\cite{CKPZ:18, CEJKK:18} was generalized to factor models with hard constraints by~\cite{CKM20}. The planted models also encompass certain high-dimensional Bayesian inference problems as studied in~\cite{BPS21, BP22}.\\

\noindent\textbf{Statistical Computational Gaps:}~In studying the computational complexity of statistical tasks such as weak recovery and 
detection, it is natural to question the average case complexity of the task. Many recent works attempt to answer this question using a number of different perspectives.
These include average case reduction to a widely-believed statistically hard problems such as the hidden-clique problem~\cite{J92,K95}, see e.g.~\cite{BB20}. 
A different approach is to study a restricted class of algorithms such as local algorithms (see e.g.~\cite{GS14} and follow up work), SQ algorithms~\cite{K98} 
or low-degree polynomials. In particular, a recent line of work (see e.g. 
\cite{hopkins2017efficient,hopkins2018statistical,kunisky2019notes, bandeira2019computational,gamarnik2020low,holmgren2020counterexamples,bresler2021algorithmic,wein2020optimal}) uses a ``low-degree heuristic'' to predict computational-statistical gaps for a variety of problems. In this context, Hopkins and Steurer~\cite{hopkins2017efficient} showed that below the KS threshold, functions that can determine if two vertices are in the same community better than random have to be of degree at least $\Omega(n)$.

\section{Main results in the general case}
\label{sec:general:results}
Our results in the most general setting are stated in the framework of planted factor models, which is a generalization of block models. In Sections~\ref{subsec:def:factor},\ref{subsec:results:factor}, we define and state our main results for the planted factor model and the associated inference tasks: weak recovery and detection. In Section~\ref{sec:application}, we apply these results to hypergraph stochastic block models (HSBM).

\subsection{Planted factor models}
\label{subsec:def:factor}
Recall that the sparse symmetric block model is defined by by first ``planting'' a community structure $\bsig^\star\equiv (\bsigma^\star_v)_{v\in V}\iid \Unif(\{1,\ldots, q\})$. Subsequently, based on this planted structure, a linear number of edges are drawn. Likewise, in the planted factor models, communities (also termed as spins) are represented as $\bsig^\star\equiv (\bsigma^\star_v)_{v\in V}\iid \pi$, where $\pi$ encodes the community's prior. Using $\bsig^\star$, we create a linear number of clauses, with each clause connecting $k$ nodes. When $k=2$, these clauses are analogous to the edges of the block model. However, allowing $k\geq 3$ enables us to consider hypergraph analogs of the block models called the hypergraph stochastic block models (HSBM) considered in~\cite{GD17a,GD17b,ACKZ15}. In addition, a notable extension in planted factor models is the consideration of clause connectivity probabilities as random variables. Such extension ensures that the planted factor models encompass planted constraint satisfaction problems studied in statistical physics and computer science~\cite{AJM05, JMS07,KMZ14reweighted,FPV15}.

We first define the necessary notations that will be used throughout the paper. Given $q\geq 2$, we denote by $[q] \equiv \{1,\ldots, q\}$ the finite set of spins or communities. We let $\pi \in \PPP([q])$ be a probability measure on $[q]$ which represents the prior of different communities.

We let $k\geq 2$ denote the number of interaction between different nodes. Further, we let $\Psi$ be a finite set of \textit{weight functions} descirbed by
\begin{equation*}
    \psi:\{1,\ldots,q\}^k \to \R_{>0}\,.
\end{equation*}
Here, we assumed the weight functions to have positive values, which corresponds to \textit{positive temperature} models in statistical physics language. Let $p\in \PPP(\Psi)$ be a prior distribution of the weight functions. Without loss of generality we let $\Psi=\supp(p)$ since otherwise we can reduce the set $\Psi$.

Throughout, we assume the following: for any permutation $\theta\in S_k$, let $\psi^{\theta}:[q]^{k}\to \R_{>0}$ defined by $\psi^{\theta}(\sigma_1,\ldots, \sigma_k)=\psi(\sigma_{\theta(1)},\ldots, \sigma_{\theta(k)})$. Then, for any $\psi \in \Psi$ and $\theta\in S_{k}$, we assume that
\beq\label{eq:psi:symm}
\psi^{\theta}\in \Psi\quad\textnormal{and}\quad p(\psi^{\theta})=p(\psi)>0\,,
\eeq
which reflects the exchangeability of the model with respect to the different variables. 

A (bipartite) factor graph $G= (V,F, E, (\psi_a)_{a\in F})$ consists
of the following.

\begin{itemize}
    \item The set of variables $V:=\{v_1,\ldots, v_n\}$ and the set of clauses (a.k.a. function nodes) $F:=\{a_1,\ldots, a_m\}$.
    \item The set of edges $E$, where $e=(a v)\in E$ connects $a\in F$ and $v\in V$. Here, each clause $a\in F$ is connected to $k$ variables, and we denote its neighborhood by $\delta a:= \{v\in V: a\sim v\}$ with the convention that $\delta a=(v_1,\ldots, v_k)$ is ordered. We let $\delta_i a\equiv v_i$ be the $i$'th variable adjacent to $a$ for $1\leq i \leq k$.
    \item The set of weight function $(\psi_a)_{a\in F}$. Here, $\psi_a\in \Psi$ is the weight function assigned to a clause $a\in F$. Since $\delta a$ is ordered, for an assignment of communities $\ux= (x_v)_{v\in V}\in [q]^V$, the expression $\psi_a(\ux_{\delta a})\equiv \psi_a(x_{v_1},\ldots, x_{v_k})$ is well-defined.
\end{itemize}  

Given the prior $p$ on $\Psi$, the random graph model with no community (a.k.a. planted) structure corresponds to the \textit{null model} defined below.
\begin{defn}\label{def:null}
(\textit{The null model}) Given $V=\{v_1,\ldots, v_n\}$, $F=\{a_1,\ldots, a_m\}$, and $p\in \mathscr{P}(\Psi)$, the null model $\bG(n,m)\equiv \bG(n,m,p)$ is a factor graph $(V,F, E, (\psi_a)_{a\in F})$ defined as follows. For each $a\in F$, its neighborhood $\delta a\in V^k$ is drawn independently and uniformly at random from $V^k$. The weight function $\psi_a$ is drawn i.i.d. from $p$, i.e. $(\psi_a)_{a\in F}\iid p$.
\end{defn}
Given a factor graph $G= (V,F, E, (\psi_a)_{a\in F})$ and $\sig \in [q]^V$, let
\beq\label{def:psi:G}
\psi_G(\sig):=\prod_{a\in F}\psi_a(\sig_{\delta a})\,.
\eeq
Subsequently, the random graph model with a community structure corresponds to the \textit{planted model}.
\begin{defn}\label{def:planted}
(\textit{The planted model}) Given $V=\{v_1,\ldots, v_n\}$, $F=\{a_1,\ldots, a_m\}$, and priors $\pi \in \PPP([q])$ and $p\in \PPP(\Psi)$, the planted model $\bG^\star(n,m,\bsig^\star)\equiv \bG^\star(n,m,\bsig^\star, p,\pi)$ is a factor graph defined as follows. First draw a planted/community structure $\bsig^\star\equiv (\bsigma^\star_1,\ldots, \bsigma^\star_n)$ by $\bsigma^\star_i\iid \pi$. Given $\bsig^\star$, draw $\bG^\star(n,m,\bsig^\star)$ from the distribution 
\[
\P\Big(\bG^\star(n,m,\bsig^\star)=G\bgiven \bsig^\star=\sig\Big)=\P(\bG(n,m)=G)\cdot\frac{\psi_{G}(\sig)}{\E[\psi_{\bG(n,m)}(\sig)]}\,,
\]
where $\E$ denotes the expectation with respect to the null model $\bG(n,m)$. Equivalently, given $\bsig^\star=(\sigma_v)_{v\in V}$, independently draw for each clause $a\in F$ the neighborhood $\delta a$ and the weight function $\psi_a$ from the distribution
\beq\label{eq:clause:law:planted}
\P\big(\delta a= (v_1,\ldots, v_k)\,,\,\psi_a=\psi\big)=\frac{p(\psi)\psi(\sigma_{v_1},\ldots, \sigma_{v_k})}{\sum_{\psi^\prime \in \Psi}p(\psi^\prime)\sum_{v^\prime_1,\ldots, v^\prime_k\in V}\psi^\prime(\sigma_{v^\prime_1},\ldots \sigma_{v^\prime_k})}=\frac{1}{n^k}\cdot \frac{p(\psi)\psi(\sigma_{v_1},\ldots, \sigma_{v_k})}{\E_{p,u}\big[\bpsi(\sig_{\bom})\big]}\,,
\eeq
where $\E_{p,u}$ denotes the expectation with respect to $\bpsi \sim p$ and $\bom \sim u:=\Unif(V^k)$.
\end{defn}
The planted model can be specialized to sparse symmetric block model by taking $k=2$, $\pi=\Unif([q])$, and $p$ to put all of its mass on a specific weight function (see Eq.~\eqref{eq:prior:weight:HSBM}). Within this context, a clause $a\in F$ corresponds to an edge connecting $\delta_1 a$ and $\delta_2 a$ in the block model. Thus $m$ corresponds to the number of edges in the block model, which is approximately Poisson  with mean $dn/2$. Therefore, in the broader planted factor model, we let the clause count $m$ to follow a Poisson distribution. Using a concise notation, we let

\begin{equation}\label{eq:def:planted}
    \bG^\star \equiv \bG^\star(n,\bm,\bsig^\star,p,\pi)\,,\;\;\;\;\; \bG\equiv \bG(n,\bm,p)\,,\;\;\;\;\;\textnormal{where}\;\;\bm\sim \Poi(dn/k)\,.
\end{equation}
Further, we respectively denote the probability distribution of $\bG^\star$ and $\bG$ by
\begin{equation}
    \GG_{\sf plant}(n,d,p,\pi):={\sf Law}(\bG^\star)\,,\;\;\;\;\GG_{\sf null}(n,d,p):={\sf Law}(\bG)\,.
\end{equation}
Given a single observation $G$, the two central tasks in statistical inference are \emph{weak recovery} and \emph{detection}.

\begin{defn}(\emph{Weak recovery})\label{def:weak:recovery}
Consider the planted model $\bG^\star\sim \GG_{\sf plant}(n,d,p,\pi)$. We say that {\em weak recovery is possible} at $d$, if
there exists an $\eps > 0$ and (sequence of) estimators $\hat{\sig}\equiv \hat{\sig}_n(\bG^\star)$ 
that takes as an input the factor graph $\bG^\star$ and returns 
$\hat{\sig}\equiv (\hat{\sigma}_v)_{v\in V}\in [q]^{V}$, such that
\beq\label{eq:overlap:nontrivial}
\limsup_{n \to \infty} \E\left[A(\bsig^\star, \hat{\sig})\right] \geq \frac{1}{q} + \eps\,,
\eeq
where the \textit{overlap} between $\sig^1\equiv (\sigma^1_v)_{v\in V}\in [q]^{V}$ and $\sig^2\equiv (\sigma^2_v)_{v\in V}\in [q]^V$ is defined by
\beq\label{eq:def:overlap}
A(\sig^1, \sig^2):=\max_{\Gamma \in S_q}\Bigg\{
\frac{1}{q}\sum_{i=1}^{q} \frac{\big|\{v\in V: \sigma^1_v = i, \sigma^2_v = \Gamma(i)\}\big|}{\big|\{v\in V: \sigma^1_v =i \}\big|}
\Bigg\}\,.
\eeq
Here, the maximum is taken with respect to permutations $\Gamma$ among the set of communities $[q]$. By fixing $p,\pi$ and varying $d>0$, the weak recovery threshold $d_{\ast}\equiv d_{\ast}(p,\pi)$ is defined by
\[
d_{\ast}\equiv d_{\ast}(p,\pi):=\inf\Big\{d>0: \textnormal{weak recovery is possible at $d$ for $\bG^\star\sim \GG_{\sf plant}(n,d,p,\pi)$}\Big\}\,.
\]
\end{defn}
We remark that Definition~\ref{def:weak:recovery} appeared previously in \cite{AbbeSandon:18, mossel23exact}. In particular, the maximum over $\Gamma\in S_q$ is taken in the definition of the overlap $A(\sig^1,\sig^2)$ to account for symmetries, if any, of the labels of the communities. For example, if $\pi=\Unif([q])$, inherent symmetries between communities imply that communities can only be recovered up to a permutation. For an equivalent, but slightly different formulation, see~\cite[Definition 2.3]{AbbeSandon:18}. 

We further remark that in Theorem~\ref{thm:equiv} below, we give equivalent descriptions of the weak recovery that involve two-point correlations and orthogonality of the samples drawn from the posterior. Furthermore, it is crucial to note that in Proposition~\ref{prop:nontrivial:threshold}, we characterize a simple condition ${\sf (SYM)}$ which ensures that the weak recovery threshold $d_{\ast}(p,\pi)$ is non-trivial, lying in the interval $d_{\ast}(p,\pi)\in [\frac{1}{k-1},\infty)$.

Before proceeding further, we make two elementary observations. First, note that the task of weak recovery is monotone in $d>0$. Specifically, if weak recovery is impossible at $d$, then weak recovery is impossible at $d'$ for all $0<d'<d$. This is validated by deleting the clauses of $\bG^\star\sim \GG_{\sf plant}(n,d,p,\pi)$ independently with probability $d'/d$ yields a sample drawn from $\GG_{\sf plant}(n,d',p,\pi)$. In particular, for all $d>d_{\ast}(p,\pi)$, weak recovery is possible at $d>0$. Second, for arbitrary $\sig^1,\sig^2\in [q]^V$, the overlap satisfies $A(\sig^1,\sig^2)\geq \frac{1}{q}$ by taking average over $\Gamma\in S_q$ instead of taking maximum as in \eqref{eq:def:overlap}. Thus, the impossibility of weak recovery is equivalent to 
\beq\label{eq:def:weak:recovery:impossible}
\lim_{n \to \infty}
\E\left[A(\bsig^\star, \hat{\sig})\right] = \frac{1}{q}\,,\quad\textnormal{for any (sequence of) estimators $\hat{\sig}\equiv \hat{\sig}_n(\bG^\star)$\,.}
\eeq

\begin{defn}(\emph{Detection})\label{def:detection}
   Consider the following hypothesis testing task given a single observation $G$:
   \beq\label{eq:hypothesis:test:facotr}
{\sf H_0}: G\sim \GG_{\sf null}(n,d,p)\;\;\;\;\textnormal{vs.}\;\;\;\;{\sf H_1}: G\sim \GG_{\sf plant}(n,d,p,\pi)
\eeq
  We say that {\em detection is possible} at $d$, if there exists a sequence of test $\big(\phi_n(G)\big)_{n\geq 1}$, which rejects the null ${\sf H_0}$ with probability $\phi_n(G)\in [0,1]$, such that it achieves vanishing Type $1$ and Type $2$ errors. That is,
  \beq\label{eq:vanishing:errors}
  \E\big[\phi_n(\bG)\big]+\E\big[1-\phi_n(\bG^\star)\big]\to 0\,,\;\;\;\;\textnormal{as}\;\;\;\; n\to\infty\,,
  \eeq
  where $\bG\sim \GG_{\sf null}(n,d,p)$ and $\bG^\star\sim \GG_{\sf plant}(n,d,p,\pi)$. If there does not exist such sequence of tests, we say {\em detection is impossible} at $d$.
\end{defn}

\subsection{Main results for the factor models}
\label{subsec:results:factor}
In our first result for the planted factor model, we prove that below the weak recovery threshold $d_{\ast}$ and the (generalized) Kesten-Stigum (KS) threshold $d_{\ks}\equiv d_{\ks}(\pi,p)$ defined in Section~\ref{subsec:contiguity}, the detection is impossible.
\begin{thm}\label{thm:factor:contiguity}
Consider the planted factor model $\bG^\star\sim \GG_{\sf plant}(n,d,p,\pi)$ in the regime $d<d_{\ast}\wedge d_{\ks}$. Then, the planted model $\bG^\star$ is mutually contiguous with the null model $\bG\sim \GG_{\sf null}(n,d,p)$. That is, for any sequence of events $(\AAA_n)_{n\geq 1}$, $\P_{\bG^\star}(\AAA_n)\to 0$ if and only if $\P_{\bG}(\AAA_n)\to 0$. In particular, the detection is impossible at $d$.
\end{thm}
In some special cases, it is known that the weak recovery threshold $d_{\ast}\equiv d_{\ast}(p,\pi)$ is at most the KS threshold $d_{\ks}\equiv d_{\ks}(p,\pi)$ in which case the condition $d<d_{\ast}\wedge d_{\ks}$ in Theorem~\ref{thm:factor:contiguity} can be simplified to $d<d_{\ast}$. This is the case for sparse (non-symmetric) block model, which corresponds to $k=2$ and $p$ being Dirac measure, by~\cite{AbbeSandon:18}. The same holds when $k\geq 3$, $p$ is a Dirac measure, and $\pi$ is the uniform measure $\Unif([q])$ by~\cite{SZ22}. It would be interesting to show that for general planted factor models $d_{\ast}\leq d_{\ks}$ holds.

In order to prove Theorem~\ref{thm:factor:contiguity}, we will prove a stronger result (cf. Theorem \ref{thm:likelihood:conv}) which characterizes the asymptotic power of the likelihood ratio test. Since such characterization requires extra technical notations regarding {\em $\zeta$-cycles}, we defer the statement to Section~\ref{sec:proof}. The proof of Theorem~\ref{thm:factor:contiguity} is in Section~\ref{subsec:contiguity}.

\subsubsection{Consequences of contiguity in point estimation}
As a consequence of the mutual contiguity in Theorem~\ref{thm:factor:contiguity}, we establish that it is impossible to consistently estimate the set of weight functions $\Psi\equiv \supp(p)$. Specifically, we consider the following notion of consistency.
\begin{defn}\label{def:local:consistent}
 Given $T\geq 1$, consider the parametric family of planted factor models $\GG_{\sf plant}(n,d,p,\pi)$ with fixed size of the support $T\equiv|\supp(p)|$. Let $\widehat{\Psi}_n(G)$ be an estimator that takes as an input a factor graph $G$ with $n$ variables and outputs a set of weight functions $\{\wh{\psi}_1,\ldots, \wh
 {\psi}_T\}$, where $\wh{\psi}_t:[q]^k\to \R_{>0}$ for $1\leq t\leq T$. For a fixed $\eps>0$, we say that the sequence of estimators $(\wh{\Psi}_n(G))_{n\geq 1}$ is {\em $\eps$-locally consistent at $(d,p,\pi)$} if the following holds. Let $\Psi=\{\psi_1,\ldots, \psi_T\}$. Then, for any $\Psi^\prime=\{\psi_1^\prime,\ldots, \psi_T^\prime\}$ such that $\|\psi^\prime_t-\psi_t\|_{\infty}\leq \eps$ holds for all $1\leq t\leq T$, we have as $n\to\infty$
 \beqn
 \wh{\psi}_t(\bG^\star)\pto \psi^\prime\,,\;\;\;\;\;\textnormal{for}\;\;\;\;\bG^\star\sim \GG_{\sf plant}(n,d,p^\prime,\pi)\;\;\;\textnormal{and}\;\;\; 1\leq t\leq T\,.
 \eeqn
Here, $p^\prime \sim \PPP(\Psi^\prime)$ is defined by letting $p^\prime(\psi^\prime_t)=p(\psi_t)$ for $1\leq t \leq T$.
\end{defn}

To clarify the definition, the non-existence of $\eps$-locally consistent estimator describes circumstances in which, despite having access to the information of the prior of the communities $\pi\in \PPP([q])$ and the mass of each weight functions $(p(\psi))_{\psi\in \Psi}\in \R^{T}$, it remains infeasable to estimate the set of weight functions $\Psi\equiv \supp(p)$ even within its $\eps$ neighborhood. Consequently, our definition of an $\eps$-locally consistent estimator is notably weaker than the usual definition of consistent estimator, which imposes no parameter restriction. Our subsequent corollary whose proof is in Section~\ref{subsec:contiguity} demonstrate that, even within this weaker notion of consistency, estimating $\Psi$ below the weak recovery threshold and the KS threshold is impossible.
\begin{cor}\label{cor:no:consistent}
    If $d<d_{\ast}\wedge d_{\ks}$ holds, then for any $\eps>0$, there does not exist $\eps$-locally consistent estimator of $\Psi$ at $(d,p,\pi)$.
\end{cor}
\subsubsection{Asymptotic mutual information}
We next consider the mutual information $I(\bG^\star,\bsig^\star)$ between the planted factor model $\bG^\star\sim \GG_{\sf plant}(n,d,p,\pi)$ and the community structure $\bsig^\star\in [q]^V$, and the \textit{Kuller-Leibler divergence} $\DKL(\bG^\star\,\|\,\bG)$ between the planted model $\bG^\star$ and the null model $\bG\sim \GG_{\sf null}(n,d,p)$. In the regime $d>d_{\ast}$, our results are stated under the following assumption:

\begin{itemize}
    \item {\sf (MIN):} For a probability vector $\pi \in \R^{q}$, define the set 
    \beq\label{eq:def:RR:pi}
    \RR_{\pi}:=\Big\{R\in [0,1]^{q\times q}: \sum_{i\in [q]}R(i,j)=\pi_j\textnormal{ for }j\in [q]\,,\textnormal{ and }\sum_{j\in [q]}R(i,j)=\pi_i\textnormal{ for }i\in [q]\Big\}\,.
    \eeq
    Let the function $\FF:\RR_{\pi}\to \R_{\geq 0}$ be defined by
    \beq\label{eq:def:FF}
    \FF(R):=\sum_{\sig, \utau\in [q]^k}\E_{p}[\bpsi(\sig)\bpsi(\utau)]\prod_{s=1}^{k}R(\sigma_s,\tau_s)\,.
    \eeq
    Then, the $R\mapsto \FF(R)$ is uniquely minimized at $R=\pi\pi^{\sT}$.
\end{itemize}
We remark that in the case $\pi=\Unif([q])$, the assumption {\sf (MIN)} was first considered in~\cite{CEJKK:18}. In addition, it is straightforward to verify that the sparse symmetric block model satisfies such assumption (see Eq.~\eqref{eq:symmetric:satisfies:MIN} below). The following result shows that under such condition, the normalized mutual information $I(\bG^\star,\bsig^\star)/n$ exhibits a phase transition at the weak recovery threshold $d_{\ast}(p,\pi)$ defined in Definition~\ref{def:weak:recovery}.
\begin{thm}\label{thm:mutual:info}
We have the following.
\begin{enumerate}[label=\textup{(\arabic*)}]
    \item Suppose that $d<d_{\ast}$ holds. Then as $n\to\infty$, we have
\beq\label{eq:KL:zero}
\frac{1}{n}\DKL(\bG^\star\,\|\,\bG)\to 0\,,\quad\textnormal{and}\quad
\frac{1}{n}I(\bG^\star,\bsig^\star)\to \frac{d}{k}\cdot\E_{p,\pi}\bigg[\frac{\bpsi(\bsig)}{\xi}\log\Big(\frac{\bpsi(\bsig)}{\xi}\Big)\bigg]\,,
\eeq
where $\E_{p,\pi}$ denotes the expectation with respect to $\bpsi\sim p$ and $\bsig \sim \pi^{\otimes k}$, and we denoted $\xi:=\E_{p,\pi}[\bpsi(\bsig)]$.
\item Conversely, suppose that $d>d_{\ast}$ and assume that the condition {\sf (MIN)} holds. Then, we have
\beqn
\limsup_{n\to\infty}\frac{1}{n}\DKL(\bG^\star\,\|\,\bG)> 0\,,\quad\textnormal{and}\quad
\liminf_{n\to\infty}\frac{1}{n}I(\bG^\star,\bsig^\star)< \frac{d}{k}\cdot\E_{p,\pi}\bigg[\frac{\bpsi(\bsig)}{\xi}\log\Big(\frac{\bpsi(\bsig)}{\xi}\Big)\bigg]\,.
\eeqn
\item Assuming the condition {\sf (MIN)}, the following holds for any $d_0>0$. For any $\eps>0$, there exists $\eta\equiv \eta(\eps)>0$ not depending on $n$ such that if $\E\left[A(\bsig^\star, \hat{\sig})\right]\geq \frac{1}{q}+\eps$ holds for some estimator $\hat{\sig}\equiv \hat{\sig}_n(\bG_0^\star)$ where $\bG_0^\star \sim \GG_{\sf plant}(n,d_0,p,\pi)$, then
for all $d > d_0$, there is an event $\AAA_n^{\ast}$ satisfying
\begin{equation}\label{eq:exp:orthogonal}
    \P(\bG\in \AAA_n^{\ast})\leq 2e^{-\eta n}\,,\quad\textnormal{and}\quad \P(\bG^\star\in \AAA_n^{\ast})\geq 1-2e^{-\eta n}\,,
\end{equation}
where $\bG\sim \GG_{\sf null}(n,d,p)$ and $\bG^\star \sim \GG_{\sf plant}(n,d,p,\pi)$. In particular, if $d>d_{\ast}$ and the condition {\sf (MIN)} holds, then there exists a sequence of events $(\AAA_n^\ast)_{n\geq 1}$, a subsequence $(n_{\ell})_{\ell \geq 1}$, and a constant $\eta>0$ such that $\P(\bG\in \AAA_{n_{\ell}}^\ast)\leq 2e^{-\eta n_{\ell}}$ while $\P(\bG^\star\in \AAA_{n_{\ell}}^{\ast})\geq 1-2e^{-\eta n_{\ell}}$. Thus, $\bG$ and $\bG^\star$ are mutually orthogonal along a subsequence.
\end{enumerate}
\end{thm}
\begin{remark}
We note that in general it is not known that if there exists a subsequence $(n_{\ell})_{\ell\geq 1}$ such that the weak recovery is possible at $d$ along $(n_{\ell})_{\ell\geq 1}$, i.e. there exists $\eps>0$ such that for all $\ell\geq 1$, $\E\left[A(\bsig^\star, \hat{\sig})\right]\geq \frac{1}{q}+\eps$ holds for $\hat{\sig}\equiv \hat{\sig}_{n_{\ell}}(\bG^\star), \bG^\star\sim \GG_{\sf plant}(n_{\ell},d,p,\pi)$, then the weak recovery is possible at $d$ (or even for $d'> d$) along every subsequence. Thus, if we replace $\limsup$ with $\liminf$ in Eq.~\eqref{eq:overlap:nontrivial} in Definition~\ref{def:weak:recovery} we may in principle obtain a different threshold for weak recovery. The fact that we do not know that the `$\limsup$' and `$\liminf$' thresholds are the same is the reason that the statements in Theorem~\ref{thm:SBM}-(3) and Theorem~\ref{thm:mutual:info}-(3) use subsequences. However, like in many other random graph and spin-glass models~\cite{CKPZ:18} we believe that these thresholds should be the same for weak recovery, and similarly for detection. Moreover, as seen in Theorem~\ref{thm:mutual:info}-(3), if we let $d_{\ast}'$ be the `$\liminf$' threshold where we replace $\limsup$ with $\liminf$ in Eq.~\eqref{eq:overlap:nontrivial}, then we have mutual orthogonality for the entire sequence above this (possibly different) threshold $d>d_{\ast}'$.
\end{remark}

\subsubsection{Equivalent notions of weak recovery}
The proof of Theorem \ref{thm:factor:contiguity} is based on a novel equivalence between the impossibility of weak recovery and near-orthogonality of the samples drawn from the posterior. To state the equivalence, we introduce more notations. Given a factor graph $G$ with $n$ variables and $m$ clauses, we denote the posterior $\mu_G\in \PPP([q]^V)$ by
\beq\label{def:posterior}
\mu_G(\sig)= \P\big(\bsig^\star=\sig \bgiven \bG^\star(n,m,\bsig^\star)=G\big)\,,\quad \sig \in [q]^V\,.
\eeq
We denote by $\langle \cdot \rangle_{G}$ the expectation with respect to samples $(\sigma^{\ell})_{\ell\geq 1}\iid \mu_G$ from from the posterior. That is, for any $L\geq 1$ and $f:([q]^V)^{L}\to \R$, we let 
\beq\label{eq:def:expectation:posterior}
\big\langle f\big(\sig^1,\ldots, \sig^L\big)\big\rangle_{G}:=\sum_{\sig^1,\ldots, \sig^{L}\in [q]^V} f(\sig^1,\ldots, \sig^{L})\prod_{\ell=1}^{L}\mu_G(\sig^{\ell})\,.
\eeq
Given two samples $\sig^{\ell}\equiv (\sigma^{\ell}_v)_{v\in V}\in [q]^V, \ell=1,2$, the \textit{overlap matrix} $R_{\sig^1,\sig^1}\equiv \big(R_{\sig^1,\sig^2}(i,j)\big)_{i,j\leq q} \in \R^{q\times q}$ is defined by
\beq\label{eq:R:def}
R_{\sig^1,\sig^2}(i,j)=\frac{1}{n}\sum_{v\in V}\one\big\{\sigma^1_v=i\,,\, \sigma^2_v=j\big\}\,.
\eeq
Concisely, the overlap matrix $R_{\sig^1,\sig^2}$ represents the empirical distribution of $(\sigma_v^1,\sigma_v^2)_{v\in V}$. In particular, the overlap matrix $R_{\sig^1,\sig^2}$ being close to $\pi\pi^{\sT}$ indicates that the two samples $\sig^1,\sig^2$ are {\em near-orthogonal}, i.e. $\sig^1$ and $\sig^2$ appear as if they are drawn independently from $\pi^{\otimes V}$. In our subsequent result, we establish that the impossibility of weak recovery is {\em equivalent} to the near-orthogonality of two samples drawn from the posterior. It plays a crucial role in deriving Theorem~\ref{thm:factor:contiguity}, but it might also be of independent interest. 
\begin{thm}\label{thm:equiv}
Given $k\geq 2,d>0, \pi, p$, the following are equivalent:
\begin{enumerate}[label=(\alph*)]
    \item Weak recovery is impossible at $d$. That is, \eqref{eq:def:weak:recovery:impossible} holds.
    \item\label{item:b} For any $2$ distinct variables $u\neq v$ and $i,j\in [q]$, we have that as $n\to\infty$,
    \beqn
    \P(\bsigma^\star_u=i, \bsigma^\star_v=j\given \bG^\star)\pto \pi_i\pi_j\,.
    \eeqn
    \item \label{item:c} For any $2$ distinct variables $u\neq v$, we have
    \beqn
    \P(\bsigma^\star_u=\bsigma^\star_v\given \bG^\star)\pto \sum_{i=1}^{q}\pi_i^2\,,\quad\textnormal{and}\quad \sum_{i=1}^{q}\pi_i \P(\bsigma^\star_u=i \given \bG^\star)\pto \sum_{i=1}^{q}\pi_i^2\,.
    \eeqn
    
    \item \label{item:d} The overlap matrix $R_{\sig^1,\sig^2}$ of the samples $\sig^1,\sig^2$ drawn from the posterior is trivial. That is, as $n\to\infty$,
    \beqn
    \E\Big\langle \big\|R_{\sig^1,\sig^2}-\pi\pi^{\sT}\big\|_1\Big\rangle_{\bG^\star} \to 0\,.
    \eeqn
\end{enumerate}
\end{thm}
Theorem~\ref{thm:equiv} is actually established within a more general framework, where $\bG^\star$ can be replaced by any random variable $\bX^\star\equiv \bX^\star(\sig^\star)$ that satisfy a certain condition {\sf (EXG)} (cf. Proposition~\ref{prop:equiv}).

\subsection{Main results for hypergraph stochastic block models}
\label{sec:application}
In this section, we apply Theorems~\ref{thm:factor:contiguity}, \ref{thm:mutual:info}, \ref{thm:likelihood:conv} in the setting of hypergraph stochastic block models (HSBM). Notably, HSBM significantly generalizes the symmetric block model considered in Section~\ref{sec:intro}. In particular, the results stated in this section immediately imply Theorems~\ref{thm:SBM}, \ref{thm:optimal:power}, \ref{thm:optimal:cycle:test} whose proofs are proved at the end of this section. To begin with, we define HSBM and the associated inference tasks.

\begin{defn}(\textit{Hypergraph Stochastic Block Model})
\label{def:HSBM}
Given $k,q\geq 2$, let $M\equiv \big(M(i_1,\ldots, i_k)\big)_{1\leq i_1,\ldots i_k\leq q}$ be a symmetric tensor of order $k\geq 2$ with positive entries. That is, for any permutation $\omega\in S_{q}$, $M(i_1,\ldots, i_k)=M(\omega(i_1),\ldots, \omega(i_k))>0$ holds. Also, let $\pi =(\pi_i)_{i\leq q} \in [0,1]^q$ be a probability vector, which encodes the prior of different communities. Given $n\geq 1$, let $\bG^\star_{\sf HSBM}\sim \GG^{\sf H}(n,M,\pi)$ be a random $k$-uniform hypergraph model with $n$ vertices defined as follows. Every vertex $v\in V$ is assigned a community $\bsigma_v^\star \in \{1,2\ldots,q\}$ independently according to $\bsigma_v^\star \iid \pi$. Given the community structure $\bsig^\star$, each possible hyperedge $(v_1,\ldots, v_k)$ consisting of distinct vertices $v_1,\ldots v_k\in V$ is included independently with probability 
$M(\bsigma^\star_{v_1},\ldots ,\bsigma^\star_{v_k})/\binom{n}{k-1}$.
\end{defn}
Note that by restricting to $k=2$, HSBM specializes to stochastic block models with prior $\pi \in \PPP([q])$ and connection probabilities encoded by an arbitrary symmetric matrix $M\in \R_{>0}^{q\times q}$. Let the average degree of HSBM denoted by
\begin{equation}\label{eq:degree:HSBM}
    d:=\sum_{i_1,\ldots,i_k\in [q]}M(i_1,\ldots,i_k)\prod_{s=1}^{k}\pi_{i_s}\,.
\end{equation}
Observe that the normalized tensor $\pr\equiv \big(\pr(i_1,\ldots, i_k)\big)_{1\leq i_1,\ldots,i_k\leq q} :=M/d$ must satisfy
\begin{equation}\label{eq:normalization}
    1=\sum_{i_1,\ldots,i_k\in [q]}\pr(i_1,\ldots,i_k)\prod_{s=1}^{k}\pi_{i_s}\,.
\end{equation}
By fixing such $M_0$ and varying $d$, the weak recovery threshold $d_{\ast}^{\sf H}\equiv d_{\ast}^{\sf H}\big(\pr, \pi \big)$ is defined by
\beqn
d_{\ast}^{\sf H}\equiv d_{\ast}^{\sf H}\big(\pr,\pi \big):=\inf\Big\{d>0: \textnormal{weak recovery is possible at $d$ for $\bG^\star_{\HSBM}\sim \GG^{\sf H}\big(n,d\pr,\pi \big)$}\Big\}\,.
\eeqn
Here, the weak recovery in HSBM is defined analogously to Definition~\ref{def:weak:recovery}. That is, we say that the weak recovery is possible at $d$ for $\GG^{\sf H}\big(n,d\pr,\pi\big)$ if there exists an $\eps>0$ and (sequence of) estimators $\hat{\sig}\equiv \hat{\sig}_n(\bG^\star_\HSBM)$ such that Eq.~\eqref{eq:overlap:nontrivial} is satisfied.

It is crucial to observe that the weak recovery threshold is trivial if the average degree around every node is not the same. That is, $d_{\ast}^{\sf H}\big(\pr,\pi \big)>0$ holds only if the following is satisfied:
\begin{equation}\label{eq:degree:condition}
    1=\sum_{i_1,\ldots,i_{k-1}\in [q]}M_0(i_1,\ldots,i_{k-1},i)\prod_{s=1}^{k-1}\pi_{i_s}\,,\quad\textnormal{for any}\quad 1\leq i \leq q\,.
\end{equation}
To see this, note that weak recovery is possible at $d$ just by assigning the communities to each vertex based on their degree. Thus, statements such as ``For $d<d_{\ast}^{\sf H}(M_0,\pi)$, \ldots'' are not vacant only if Eq.~\eqref{eq:degree:condition} is satisfied. We further note that if Eq.~\eqref{eq:degree:condition} is satisfied, Proposition~\ref{prop:nontrivial:threshold} below shows that $d_{\ast}^{\sf H}\big(\pr,\pi \big)\geq \frac{1}{k-1}$ holds. 

In the case $\pr=\bone_{k,q}$, the all-$1$-tensor, the random graph model $\GG_{\sf ER}^{\sf H} (n,d,k)\equiv \GG^{\sf H}(n,d\cdot \bone_{k,q},\pi)$ does not have community structure, and does not depend on $\pi$. This is the hypergraph analog of sparse Erdos-Renyi graphs with average degree $d$. Thus, the hypothesis test corresponding to Eq.~\eqref{eq:hypothesis:test} which determines the existence of a community structure for a k-uniform hypergraph $G$ is given by
   \beq\label{eq:hypothesis:test:HSBM}
{\sf H_0}: G\sim \GG_{\sf ER}^{\sf H} (n,d,k)\;\;\;\;\textnormal{vs.}\;\;\;\;{\sf H_1}: G\sim \GG^{\sf H} (n,dM_0,\pi)
\eeq
Analogously to Definition~\ref{def:detection}, we say that detection is possible at $d$ if there exists a sequence of tests $\big(\phi_n(G)\big)_{n\geq 1}$ which achieves vanishing Type 1 and Type 2 errors (cf. Eq.~ \eqref{eq:vanishing:errors}).

\subsubsection{Contiguity, point estimation, and mutual information in HSBM}
We now specialize Theorems~\ref{thm:factor:contiguity}, \ref{thm:mutual:info}, and Corollary~\ref{cor:no:consistent} to HSBM. To this end, we first define the KS threshold for HSBM, which is simple to state. Given $M_0$ and $\pi$, define the matrix $B\equiv \big(B(i,j)\big)_{i,j\leq q}\in \R^{q\times q}$ by
\beq\label{eq:def:B}
B(i,j)=\sum_{i_1,\ldots, i_{k-2}\in [q]}M_0(i_1,\ldots, i_{k-2},i,j)\pi_j\prod_{s=1}^{k-2}\pi_{i_s}\,.
\eeq
Note that under the degree condition in Eq.~\eqref{eq:degree:condition}, the matrix $B$ is a stochastic matrix. Thus, if we let $\lambda_1,\ldots, \lambda_q$ be the eigenvalues of $B$ ordered in decreasing absolute values, then by Perron-Frobenius theorem, we have 
\beqn
1=\la_1>|\la_2|\geq\ldots \geq |\la_q|\,.
\eeqn
Then, the KS threshold for HSBM is defined by
\beq\label{eq:def:KS:HSBM}
d_{\ks}^{\sf H} \equiv d_{\ks}^{\sf H}(M_0,\pi):=|\la_2|^{-2}\,.
\eeq
The following definition is the translation of local consistent estimators (cf. Definition~\ref{def:local:consistent}) within HSBM.
\begin{defn}
   Given $k,q, \pi, M$, consider an estimator $\widehat{M}(G)$, which takes as an input a $k$-uniform hypergraph $G$, and outputs a symmetric tensor of order $k\geq 2$ with positive entries. For $\eps>0$, we say that $\widehat{M}(\cdot)$ is \textit{$\eps$-locally-consistent for $M$ at $(d,\pi,M)$} if for all symmetric tensor $M^\prime$ of order $k$ such that $\|M^\prime-M\|_1\leq \eps$ holds, we have that as $n\to\infty$, 
   \beqn
   \widehat{M}(\bG^\star_{\HSBM})\pto M^\prime \,,\;\;\;\;\textnormal{for}\;\;\;\bG^\star_{\HSBM}\sim \GG^{\sf H}(n,M^\prime,\pi)\,.
   \eeqn
\end{defn}
We reiterate that $\eps$-locally consistency is a weaker notion than standard consistency since the parameter space is restricted to be the $\eps$ neighborhood of $M$. Our results establish that below the weak recovery threshold, even this weaker notion of parameter estimation cannot be achieved.
\begin{thm}\label{thm:HSBM:contiguity}
Let $\pr$ be a symmetric tensor of order $k\geq 2$ with positive entries, and $\pi$ be a $q$-dimensional probability vector for $q\geq 2$. For average degree $d$ below the weak recovery threshold and the KS threshold $d<d_{\ast}^{\sf H}\big(\pr,\pi \big)\wedge d_{\textsf{\textup{KS}}}^{\sf H}(M_0,\pi)$, the following holds.
  \begin{enumerate}[label=\textup{(\arabic*)}]
       \item \label{item:contiguity:HSBM} $\bG^\star_{\HSBM}\sim \GG^{\sf H}\big(n,d\pr,\pi \big)$ is mutually contiguous with the Erdos-Renyi hypergraph $\bG_{\ER}\sim \GG_{\sf ER}^{\sf H} (n,d,k)$. Thus, detection is impossible. Moreover, for any fixed $\eps>0$, there exists no $\eps$-locally-consistent estimator for $M$.
       \item \label{item:mutual:info:below:HSBM} The asymptotic per-vertex mutual information between $\bG^\star_{\HSBM}$ and the community structure $\bsig^\star$ is given by
    \beqn
    \lim_{n\to\infty}\frac{1}{n}I(\bG^\star_{\HSBM},\bsig^\star)=\frac{d}{k}\cdot \sum_{i_1,\ldots i_k=1}^{q}\pr(i_1,\ldots, i_k)\log\big(\pr(i_1,\ldots,i_k)\big) \prod_{s=1}^{k}\pi_{i_s}\,.
    \eeqn
  \end{enumerate}
  On the other hand, suppose the average degree is above the weak recovery threshold $d>d_{\ast}^{\sf H}(M_0,\pi)$ and the condition ${\sf (MIN)}$ holds. Then, we have 
  \begin{enumerate}[label=\textup{(\arabic*)}]\setcounter{enumi}{2}
    \item \label{item:mutual:orthogonal:HSBM}  There exists a sequence of events $(\AAA^{\ast}_n)_{n\geq 1}$, a subsequence $(n_{\ell})_{\ell \geq 1}$, and a positive constant $\eta>0$ not depending on $\ell$ such that for any $\ell \geq 1$,
    \beqn
    \P(\bG_{\ER}\in \AAA^{\ast}_{n_{\ell}})\leq 2e^{-\eta n_{\ell}}\,,\;\;\;\;\textnormal{and}\;\;\;\; \P(\bG^\star_{\HSBM}\in \AAA^{\ast}_{n_{\ell}})\geq 1-2e^{-\eta n_{\ell}}\,,
    \eeqn
    where $\bG_{\ER}\sim \GG_{\sf ER}^{\sf H} (n_{\ell},d,k)$ and $\bG^\star_{\HSBM}\sim \GG^{\sf H}\big(n_{\ell},d\pr,\pi \big)$. Thus, detection is possible along the subsequence $(n_{\ell})_{\ell \geq 1}$.
    \item \label{item:mutual:info:above:HSBM} We have that
    \beqn
    \liminf_{n\to\infty}\frac{1}{n}I(\bG^\star_{\HSBM},\bsig^\star)<\frac{d}{k}\cdot \sum_{i_1,\ldots i_k=1}^{q}\pr(i_1,\ldots, i_k)\log\big(\pr(i_1,\ldots,i_k)\big) \prod_{s=1}^{k}\pi_{i_s}\,.
    \eeqn
  \end{enumerate}
\end{thm}
We remark the condition~{\sf (MIN)} within HSBM is translated to the following.
\begin{itemize}
    \item  {\sf (MIN):} Recall the set $\RR_{\pi}$ in \eqref{eq:def:RR:pi}. Consider the function $\FF_{M_0}:\RR_{\pi}\to \R_{\geq 0}$ defined by 
    \beqn
    \FF_{M_0}(R):=\sum_{\sig, \utau\in [q]^k}M_{0}(\sig)M_{0}(\utau)\prod_{s=1}^{k}R(\sigma_s,\tau_s)\,.
    \eeqn
    Then, the $\FF_{M_0}(R)$ for $R\in \RR_{\pi}$ is uniquely minimized at $R=\pi\pi^{\sT}$.
\end{itemize}
Symmetric HSBM corresponds to the uniform prior $\pi=\Unif([q])$ and the tensor $M_0$ taking at most two values $a,b>0$, where $M_0(i_1,\ldots, i_k)=a$ if $i_1=\ldots =i_k$ and $M_0(i_1,\ldots, i_k)=b$, otherwise. The symmetric HSBM cleary satisfies {\sf (MIN)} since for $R\in \RR_{q^{-1}\bone}$, it is straightforward to compute
\beq\label{eq:symmetric:satisfies:MIN}
\FF_{M_0}(R)=1+(a-b)^2\sum_{i=1}^{q}R(i,i)^k\,,
\eeq
which is uniquely minimized at $R=q^{-2}\bone\bone^{\sT}$.
\subsubsection{Hypothesis testing in HSBM}
Our next result concerns the asymptotic power of the liklihood ratio of the hypothesis test~\eqref{eq:hypothesis:test:HSBM}. To this end, consider the likelihood ratio
\begin{equation*}
    \LL_n(G)=\frac{\P(\bG^\star_{\HSBM}=G)}{\P(\bG_{\ER}=G)}\,,
\end{equation*}
where $\bG_{\ER}\sim \GG^{\sf H}_{\ER}(n,d,k)$ and $\bG^\star_{\HSBM}\sim \GG^{\sf H}(n,dM_0,\pi)$. 
\begin{thm}\label{thm:optimal:power:HSBM}
Let $\pr$ be a symmetric tensor of order $k\geq 2$ with positive entries, and $\pi$ be a $q$-dimensional probability vector for $q\geq 2$. For average degree $d$ below the weak recovery threshold and the KS threshold $d<d_{\ast}^{\sf H}\big(\pr,\pi \big)\wedge d_{\ks}^{\sf H}\big(\pr,\pi \big)$, the likelihood ratio under the null $\bG_{\sf ER}\sim \GG^{\sf H}_{\ER}(n,d,k)$ converges in distribution to
\begin{equation*}
    \LL_n(\bG_{\ER})\dto \boldsymbol{\LL}_{\infty}
    :=\prod_{\ell=2+\one\{k=2\}}^{\infty}\frac{(1+\alpha_{\ell})^{\bX_{\ell}}}{\E(1+\alpha_{\ell})^{\bX_{\ell}}}\,.
\end{equation*}
Here, $(\bX_{\ell})_{\ell\geq 2}$ are independent Poisson random variables with mean $\E \bX_{\ell}=\frac{1}{2\ell}\big((k-1)d\big)^{\ell}$, and the constants  $(\alpha_{\ell})_{\ell\geq 2}$ are defined by
\beq\label{eq:def:alpha}
\alpha_{\ell}:=\tr(B^{\ell})-1=\sum_{i=2}^{q}\la_i^{\ell}\,,
\eeq
where $B\in \R^{q\times q}$ is defined in Eq.~\eqref{eq:def:B}. Moreover, the random variable $\boldsymbol{\LL}_{\infty}\equiv \boldsymbol{\LL}_{\infty}(d,M_0,\pi)$ satisfies the following.
\begin{enumerate}[label=\textup{(\arabic*)}]
\item Below the Kesten-Stigum threshold $d<d_{\ks}^{\sf H}(M_0,\pi)$, $\boldsymbol{\LL}_{\infty}$ is well-defined (i.e. the infinite product converges a.s.) and has finite second moment $\E \boldsymbol{\LL}_{\infty}^2<\infty$ .
\item For $\frac{1}{k-1}\leq d<d_{\ks}^{\sf H}(M_0,\pi)$, the random variable $\boldsymbol{\LL}_{\infty}$ does not have a point mass. In particular, in the regime $\frac{1}{k-1}\leq d<d_{\ast}^{\sf H}(M_0,\pi)\wedge d_{\ks}^{\sf H}(M_0,\pi)$, the asymptotic power of the likelihood ratio test for the hypothesis test~\eqref{eq:hypothesis:test} at significance level $\alpha\in (0,1)$ is given by $\beta_{\ast}(\alpha)\in (0,1)$, where
 \beq\label{eq:def:power:HSBM}
      \beta_{\ast}(\alpha)=\E\big[\boldsymbol{\LL}_{\infty}\one\{\boldsymbol{\LL}_{\infty}\geq C_{\alpha}\}\big]\,.
     \eeq
     Here, $C_{\alpha}>0$ is an arbitrary constant satisfying $\P(\boldsymbol{\LL}_{\infty}\geq C_{\alpha})=\alpha$,
     whose existence is guaranteed.
\end{enumerate}
\end{thm}
We next construct a computationally efficient and most powerful test based on a low-degree polynomial of the adjacency matrix. Given a $k$-uniform hypergraph $G$ with $n$ nodes, let $X_{\ell}(G)$ count the number of cycles of length $\ell$. Here, a cycle of length $\ell$ is a set of vertices $\{v_1,\ldots, v_{\ell}, v_{\ell+1}\equiv v_1\}$, where $v_i$ and $v_{i+1}$ are connected by an hyperedge for all $1\leq i \leq \ell$. Then, consider the following statistic based on the cycle counts of $G$:
\beqn
\TT_n(G):=\prod_{\ell=2+\one\{k=2\}}^{K_n}(1+\alpha_{\ell})^{X_{\ell}(G)}\,,
\eeqn
where the truncation parameter is chosen so that $1\ll K_n=O(\log\log n)$.
\begin{thm}\label{thm:optimal:cycle:test:HSBM}
 Let $\pr$ be a symmetric tensor of order $k\geq 2$ with positive entries, and $\pi$ be a $q$-dimensional probability vector for $q\geq 2$. Also, let $K_n=O(\log\log n)$ and $K_n\to\infty $ as $n\to\infty$. For a significance level $\alpha\in (0,1)$, consider the test $\phi_{n,\alpha}(\cdot)$ which rejects the null ${\sf H_0}$ in~\eqref{eq:hypothesis:test:HSBM} with probability 
\beqn
\phi_{n,\alpha}(G):=
\begin{cases}
    1 & \;\;\;\;\textnormal{if}\;\;\;\;\; \TT_n(G)> C'_{n,\alpha}\,;\\
    0 & \;\;\;\;\textnormal{otherwise}\,.
\end{cases}
\eeqn
Here, the constants $C_{n,\alpha}'>0$ is chosen so that we have 
\beqn
\P(\TT_n(\bG_{\ER})>C^\prime_{n,\alpha})\leq \alpha\leq \P(\TT_n(\bG_{\ER})\geq C^\prime_{n,\alpha})\,,
\eeqn
where $\bG_{\ER}\sim \GG^{\sf H}_{\ER}(n,d,k)$. Then, for $\frac{1}{k-1}\leq d<d_{\ks}^{\sf H}(M_0,\pi)$, the test $\phi_{n,\alpha}$ achieves the power $\beta_{\ast}(\alpha)$ in Eq.~\eqref{eq:def:power:HSBM}. That is, under the alternative $\bG^\star_{\HSBM}\sim\GG^{\sf H}(n,dM_0,\pi)$, we have
\beqn
\E\phi_{n,\alpha}(\bG^\star_{\HSBM})\to \beta_{\ast}(\alpha)\;\;\;\;\textnormal{as}\;\;\;\; n\to\infty\,.
\eeqn
In particular, in the regime $\frac{1}{k-1}\leq d <d_{\ast}^{\sf H}(M_0,\pi)\wedge d_{\ks}^{\sf H}(M_0,\pi)$, the test $\phi_{n,\alpha}(\cdot)$ is asymptotically most powerful for the hypothesis test~\eqref{eq:hypothesis:test}.
\end{thm}
For symmetric HSBM, it was shown in \cite{GP23} that the weak recovery threshold equals the KS threshold for $q=2$ and $k\in \{3,4\}$. Thus, combining this with our results give the following corollary.
\begin{cor}
Consider symmetric HSBM with $q=2$ communities and $k\in \{3,4\}$ interactions. Then, Theorems~\ref{thm:HSBM:contiguity}, \ref{thm:optimal:power:HSBM}, \ref{thm:optimal:cycle:test:HSBM} hold with $d_{\ast}^{H}(M_0,\pi)\equiv d_{\ks}^{\sf H}(M_0,\pi)$. Thus, in these cases, for any average degree $d$ such that $d\geq \frac{1}{k-1}$ and $d\neq d_{\ks}^{\sf H}(M_0,\pi)$ hold, and any significance level $\alpha\in (0,1)$, there exists an efficient sequence of statistical tests $\varphi_{n,\alpha}(\cdot)$ that is asymptotically most powerful at level $\alpha$.
\end{cor}

\section{Proofs}
\label{sec:proof}
This section provides the proof of Theorems~\ref{thm:factor:contiguity}, \ref{thm:mutual:info}, and Corollary~\ref{cor:no:consistent} while deferring some of the technical steps to the appendices. Theorems~\ref{thm:equiv}, \ref{thm:HSBM:contiguity}, \ref{thm:optimal:power:HSBM}, and \ref{thm:optimal:cycle:test:HSBM} are derived from our results for planted factor models and their proof is deferred to Appendix~\ref{sec:appendix:HSBM}.
\vspace{2mm}

\noindent \textbf{Notations:}

 The notation $\E_{\pi}$ is used for the expectation with respect to $\bsig\equiv (\bsigma_1,\ldots, \bsigma_k)\sim \pi^{\otimes k}$, while $\E_{p}$ is used for the expectation with respect to $\bpsi \sim p$. Additionally $u=\Unif(V^k)$ denotes the uniform distribution among $k$ tuples of variables, and $\E_{u}$ denotes the expectation with respect to $\bom \sim u$. These notations can also be used jointly. For instance, $\E_{p,\pi}$ denotes the expectation with respect to $\bpsi\sim \pi$ and $\bsig\sim \pi^{\otimes k}$. We reserve the greek letter $\xi$ to denote 
 \[
 \xi:=\E_{p,\pi}[\bpsi(\bsig)]\,.
 \]
 We denote by $\PPP(\Omega)$ the set of probability measures on a finite set $\Omega$.

\subsection{Proof of Theorem~\ref{thm:factor:contiguity} and Corollary~\ref{cor:no:consistent}}\label{subsec:contiguity}
In this section, we prove Theorem~\ref{thm:factor:contiguity} and its consequences in hypothesis testing (see Corollaries~\ref{cor:power:factor}, \ref{cor:cycle:test:factor}) and point estimation in Corollary~\ref{cor:no:consistent}. The proof of Theorem~\ref{thm:factor:contiguity} and Corollary~\ref{cor:no:consistent} is at the end of Section~\ref{subsubsec:SSG} and \ref{subsubsec:cor} respectively. To begin with, we consider the following property, which guarantees that the task of weak recovery is non-trivial.

\begin{itemize}
    \item ${\sf (SYM):}$ For any $\psi \in \Psi$, $1\leq s \leq k$, and $\tau\in [q]$, $\E_{\pi}[\psi(\bsig)\given \bsigma_s=\tau]=\E_{p,\pi}[\bpsi(\bsig)]$ holds.
\end{itemize}
Indeed, we prove that if ${\sf (SYM)}$ is violated, then weak recovery is possible for any $d>0$ based on estimating communities based on the degree of each variables. On the contrary, if ${\sf (SYM)}$ is satisfied, then we prove that the task of weak recovery is non-trivial. Specifically, the following result is proven in Section~\ref{sec:nontriviality}.
\begin{prop}\label{prop:nontrivial:threshold}
If ${\sf (SYM)}$ does not hold for $p,\pi$, then weak recovery is possible for any $d>0$. That is, the violation of ${\sf (SYM)}$ implies $d_{\ast}(p,\pi)=0$. On the other hand, if ${\sf (SYM)}$ is satisfied for $p,\pi$ and at least one of the weight function $\psi\in \Psi$ are not constant functions, then we have $\frac{1}{k-1}\leq d_{\ast}(p,\pi)<\infty$.
\end{prop}

\noindent Note that Theorem \ref{thm:factor:contiguity} and Corollary~\ref{cor:no:consistent} is trivial if $d_{\ast}(p,\pi)=0$. Therefore, by Proposition \ref{prop:nontrivial:threshold}, we may safely assume ${\sf (SYM)}$ throughout the proof.

\subsubsection{Generalized Kesten-Stigum threshold} 
\label{subsec:GKS}
This section defines the KS threshold for the planted factor model, which is a generalization of the uniform prior case $\pi=\Unif([q])$ considered in~\cite{CEJKK:18} and sparse (non-symmetric) block model case considered in~\cite{AbbeSandon:18}. For $\psi\in \Psi$, let the matrix $\Phi_{\psi}\equiv \Phi_{\psi}(i,j)_{i,j\in [q]}\in \R^{q\times q}$ be defined by
\beq\label{def:Phi}
\Phi_{\psi}(i,j):=\xi^{-1}\cdot \E_{\pi}\big[\psi(\bsig)\bgiven \bsigma_1=i, \bsigma_2=j\big]\cdot\pi_{j}\,.
\eeq

With a slight abuse of notation, we often identify a matrix $A\equiv \big(A(\tau_1,\tau_2)\big)_{\tau_1,\tau_2\in [q]}$ with the linear map $x\in \R^{q}\to Ax\in \R^{q}$. With this convention, let the linear operator $\Xi\equiv \Xi_{p,\pi}$ on $\R^{q}\otimes \R^{q}$ defined by
\beqn
\Xi\equiv \Xi_p:=\E_{p}\big[\Phi_{\bpsi}\otimes \Phi_{\bpsi}\big]\,,
\eeqn
where $\otimes$ denotes the tensor product. In addition, define the linear operator $\Xi_{\ast}\equiv \Xi_{\ast,p,\pi}$ on $\R^{q}\otimes \R^{q}$ by
\beqn
\Xi_{\ast}\equiv \Xi_{\ast,p,\pi}:= \E_{p}\bigg[\Big(\Phi_{\bpsi}-\bone  \pi^{\sT}\Big)\otimes \Big(\Phi_{\bpsi}-\bone  \pi^{\sT}\Big)\bigg]\,,
\eeqn 
where $\bone \in \R^{q}$ denote the all-$1$-vector. Here, we identified the matrix $\Phi_{\bpsi}-1\pi^{\sT}\in \R^{q}$ with the corresponding linear map on $\R^{q}$. Let $\langle \cdot\,,\,\cdot \rangle_{\pi}$ denote the unique inner product on $\R^{q}\otimes \R^{q}$ that satisfies
\beq\label{def:inner:product}
\langle x_1\otimes x_2 \,,\,y_1\otimes y_2\rangle_{\pi}:= x_1^{\sT}\diag(\pi) y_1\cdot x_2^{\sT}\diag(\pi) y_2\,,
\eeq
where $x_1,x_2,y_1,y_2\in \R^{q}$ and $\diag(\pi)\equiv \diag\big((\pi)_{i\leq q}\big)$. Further, let $\SS \subset \R^{q}\otimes \R^{q}$ be the subspace defined by
\beqn
\SS:=\big\{x\in \R^{q}\otimes \R^{q}: \langle x, w\otimes \bone\rangle_{\pi}=\langle x, \bone \otimes w \rangle_{\pi}=0\,,~\textnormal{for all $w\in \R^{q}$}\big\}\,.
\eeqn
We denote the projection operator onto the subspace $\SS$ by $\proj_{\SS}$. In order to define the KS threshold, we need the following lemma whose proof is deferred to Section~\ref{sec:proof:contiguity}.
\begin{lemma}\label{lem:KS}
Assuming ${\sf (SYM)}$, the following holds for the linear operator $\Xi$ on $\R^{q}\otimes \R^{q}$.
\begin{enumerate}[label=\textup{(\arabic*)}]
    \item $\Xi$ is self-adjoint on the inner product space $\big(\R^{q}\otimes \R^{q},\langle \cdot\,,\,\cdot \rangle_{\pi}\big)$.
    \item $\Xi(\SS)\subset \SS$ and $\Xi(\SS^{\perp})\subset \SS^{\perp}$ hold, where $\SS^{\perp}$ is the orthogonal subspace of $\SS$ with respect to the inner product $\langle \cdot\,,\,\cdot \rangle_{\pi}$.
    \item The composition of $\Xi$ and $\proj_{\SS}$ is given by $\Xi\circ \proj_{\SS}=\Xi_{\ast}$.
\end{enumerate}
\end{lemma}
Having Lemma~\ref{lem:KS} in hand, we denote the eigenspace of $\Xi$ restricted to the subspace $\SS$ by
\beq\label{def:eig:space}
\Eig_{\SS}(\Xi):=\big\{\lambda\in \R: \Xi x = \lambda x~\textnormal{for some $x\in \SS\setminus \{0\}$}\big\}\,.
\eeq
By Lemma~\ref{lem:KS} and the spectral theorem, $\Eig_{\SS}(\Xi)$ is non-empty whose non-zero elements equal the non-zero elements of $\Eig(\Xi_{\ast})$, the set of eigenvalues of $\Xi_{\ast}$. The KS threshold is defined as follows.
\begin{defn}\label{def:KS:factor}
\textit{(Kesten-Stigum threshold)} Let $\lambda_{\ks}\equiv \lambda_{\ks}(p,\pi):=\max_{\lambda \in \Eig_{\SS}(\Xi)}|\lambda|=\max_{\lambda\in \Eig(\Xi_{\ast})}|\la|$. 
The Kesten-Stigum (KS) threshold is defined by $d_{\ks}\equiv \big((k-1) \lambda_{\ks}\big)^{-1}$.
\end{defn}

Notably, the KS threshold for HSBM in Eq.~\eqref{eq:def:KS:HSBM} is a special case of Definition~\ref{def:KS:factor}, where the prior $p$ puts all of its mass on a single weight function (or equivalently a tensor) on $M$.
\subsubsection{Second moment of the truncated likelihood ratio} \label{subsec:truncated}
For a factor graph $G$ with $n$ variables and $m$ clauses, the likelihood ratio (LR) for the hypothesis test~\eqref{eq:hypothesis:test:facotr} evaluated at $G$ is given by
\beq\label{def:LR}
L(G)\equiv L(G;p,\pi):=\frac{\P\big(\bG^\star(n,m,\bsig^\star)=G\big)}{\P\big(\bG(n,m)=G\big)}\,,
\eeq
where $\bG(n,m)$ is the null model in Definition~\ref{def:null} and $\bG^\star(n,m,\bsig^\star)$ is the planted factor model in Definition~\ref{def:planted}. To show that $\bG\equiv \bG(n,\bm)$ and $\bG^\star\equiv \bG^\star(n,\bm)$ are mutually contiguous for $d<d_{\ast}\wedge d_{\ks}$, we apply the second moment method to a suitable truncation of the likelihood ratio supplemented with the small subgraph conditioning method~\cite{Janson95random, Wormald99models}.

A standard observation is that because the randomness of the number of clauses $\bm\sim \Poi (dn/k)$ causes the second moment $\E L(\bG)^2$ to diverge, it is necessary to first condition on $\bm=m_n$ for typical values of $(m_n)_{n\geq 1}$. In addition, note that $\E L(\bG(n,m))=1$ holds for $n,m\geq 1$ by a simple change of measure, thus the second moment method requires to prove that $\E L(\bG(n,m_n))^2=O(1)$ holds. 

However, even after conditioning on $\bm=m_n$, $\E L(\bG(n,m_n)) ^2=O(1)$ fails to hold in general for all range of $d<d_{\ast}\wedge d_{\ks}$. Indeed, by the results of~\cite{BMNN:16}, even restricted to symmetric block model with $3$ or more communities, there is a non-trivial regime below $d_{\ast}$ such that the second moment is exponentially large, i.e. $\E L(\bG(n,m_n))^2=e^{\Omega(n)}$. 

To overcome such difficulty, we truncate the likelihood ratio $L(\bG(n,m_n))$ to ensure that the two samples drawn from the posterior is near-orthogonal. Such truncation is motivated by Theorem~\ref{thm:equiv}, which is detailed as follows. For a sequence of truncation parameters $(\eps_n)_{n\geq 1}$, let the truncated LR be defined by
 \beq\label{def:truncated:LR}
L^{\ast}(G)\equiv L^{\ast}(G;\eps_n):=\frac{\P\big(\bG^\star(n,m,\bsig^\star)=G\big)}{\P\big(\bG(n,m)=G\big)}\one\Big\{\big\langle \big\|R_{\sig^1,\sig^2}-\pi\pi^{\sT}\big\|_1\big\rangle_{G}\leq \eps_n\Big\}\,,
 \eeq
 where we recall that $\langle \cdot \rangle_G$ denotes the expectation taken w.r.t. the sampled $\sig^1,\sig^2$ drawn from the posterior (see Eq.~\eqref{eq:def:expectation:posterior}). A crucial observation is that by a change of measure, the truncation affects the first moment by
 \beq\label{eq:diff:L:L:star}
 \begin{split}
 \E\big| L\big(\bG(n,m_n)\big)-L^{\ast}\big(\bG(n,m_n))\big|
&=\E L\big(\bG(n,m_n)\big)\one\big\{\big\langle \big\|R_{\sig^1,\sig^2}-\pi\pi^{\sT}\big\|_1\big\rangle_{\bG(n,m_n)}>\eps_n\big\}\\
&=\P\Big(\big\langle \big\|R_{\sig^1,\sig^2}-\pi\pi^{\sT}\big\|_1\big\rangle_{\bG^\star(n,m_n)}>\eps_n\Big)\,.
 \end{split}
 \eeq
 Note that if $d<d_{\ast}(p,\pi)$ and $m_n$ is sufficiently close to $dn/k$, Theorem~\ref{thm:equiv} suggests that the right hand side tends to $0$ by choosing $\eps_n=o_n(1)$ appropriately. Indeed, the following lemma guarantees that the truncation does not affect the first moment below the weak recovery threshold for some $\eps_n=o_n(1)$. Its proof is deferred to Section~\ref{sec:proof:contiguity}.
\begin{lemma}\label{lem:m:n:overlap:trivial}
Let $d<d_{\ast}$ and $|m_n-dn/k|\leq n^{2/3}$ for $n\geq 1$. Then, there exists a sequence $(\eps_n)_{n\geq 1}$ such that as $n\to\infty$, $\eps_n\to 0$ and the following convergence is satisfied.
\beqn
\E \left| L\big(\bG(n,m_n)\big)-L^{\ast}\big(\bG(n,m_n))\right|\to 0\,.
\eeqn
\end{lemma}
On the other hand, the truncation in $L^\ast(G)$ reduces the second moment tremendously. In particular, the following second moment estimate is at the heart of the proof of Theorem~\ref{thm:factor:contiguity}. Its proof is the most technical piece of the paper and it is deferred to Section~\ref{sec:proof:contiguity}.
\begin{prop}\label{prop:sec:moment:LR}
    Let $d<d_{\ks}$ and assume the condition {\sf (SYM)}. Then, for any sequences $(m_n)_{n\geq 1}$ and $(\eps_n)_{n\geq 1}$ such that $|m_n-dn/k|\leq n^{2/3}$ and $\eps_n\to 0$ as $n\to\infty$, we have
    \beq
    \E L^{\ast}\big(\bG(n,m_n)\big)^2\leq \big(1+o_n(1)\big) \prod_{\lambda\in \Eig_{\SS}(\Xi)}\frac{1}{\sqrt{1-(k-1)d\lambda}}\,.
    \eeq
\end{prop}

\subsubsection{Cycles and small subgraph conditioning}
\label{subsubsec:SSG}
Proposition \ref{prop:sec:moment:LR} guarantees that $L^{\ast}\big(\bG(n,m_n)\big)$ is bounded away from zero with uniformly positive probability by the Payley-Zygmund inequality (a.k.a. second moment method). To boost this probability to close to one, we use the second moment method conditioned on the number of small cycles. This is formalized by the {\em small subgraph conditioning method} developed by~\cite{Janson95random, Wormald99models}. In particular, we have the following theorem from \cite{Wormald99models} (see also \cite[Theorem 1]{Janson95random} and \cite[Theorem 9.12]{Janson00random}).
\begin{thm}\label{thm:SSG}
\textup{\cite[Theorems 4.3]{Wormald99models}, \cite[Theorem 9.12]{Janson00random}}
Let $(X_{in})_{i\geq 1}$ be a set of non-negative integer valued random variables indexed by $n\geq 1$, and let $Y_n$ be another non-negative random variable on the same probability space as $(X_{in})_{i\geq 1}$. Suppose that $\E Y_n>0$ for large enough $n$, and there exist $\lambda_i>0$ and $\delta_i\geq -1$ such that the following hold.
\begin{enumerate}[label=\textup{(A\arabic*)}]
    \item \label{item:A1} $X_{in}\dto X_{i\infty}$ as $n\to\infty$, jointly for all $i$, where $X_{i\infty}$ are independent $\Poi(\lambda_i)$ distributed random variables.
    \item \label{item:A2} For any finite sequence $x_1,\ldots, x_{L}$ of non-negative integers, as $n\to\infty$,
    \begin{equation*}
    \frac{\E[Y_n\one\{X_{1n}=x_1,\ldots, X_{L n}=x_L\}]}{\E Y_n}\longrightarrow \prod_{i=1}^{L}\frac{\big((1+\delta_i)\lambda_i\big)^{x_i}}{x_i!}e^{-(1+\delta_i)\lambda_i}\,.
    \end{equation*}
    \item \label{item:A3} $\sum_{i\geq 1}\lambda_i\delta_i^2<\infty$.
    \item \label{item:A4} $\frac{\E Y_n^2}{(\E Y_n)^2}\leq \exp\big(\sum_{i\geq 1}\lambda_i\delta_i^2\big)+o_n(1)$ as $n\to\infty$.
\end{enumerate}
Then, we have as $n\to\infty$ that
\begin{equation*}
\frac{Y_n}{\E Y_n}\dto W\equiv \prod_{i=1}^{\infty}(1+\delta_i)^{X_{i\infty}}e^{-\lambda_i \delta_i}\,.
\end{equation*}
Moreover, this and the convergence in \textup{(A1)} hold jointly.
\end{thm}
\begin{remark}\label{rmk:L:infty:well:defined}
    An implicit conclusion of Theorem~\ref{thm:SSG} is that $W$ is well-defined as the almost-sure-limit of $W_m\equiv \prod_{i\leq m}(1+\delta_i)^{X_{i\infty}}e^{-\la_i\delta_i}$ as $m\to\infty$. This can be easily seen by martingale convergence theorem since $(W_m)_{m\geq 1}$ is a martingale with mean $1$ and bounded variance by \ref{item:A3}.
\end{remark}
In random graph theory, it is common to take $(X_{in})_{i\geq 1}$ in Theorem \ref{thm:SSG} as the number of cycles in the random graph with $n$ nodes. Indeed, the works~\cite{CO18number, CEJKK:18, nss2} applied Theorem~\ref{thm:SSG} for various factor graphs by setting $(X_{in})_{i\geq 1}$ as the number of cycles with specific signatures. We consider the following specific notion of cycles with signature $\zeta$, which is essentially the same as the one considered in \cite{CEJKK:18}.

\begin{defn}($\zeta$-cycle) A signature of order $\ell\geq 1$ is defined by a family 
\beqn
\zeta=(\psi_1,\ldots, \psi_{\ell}, s_1,t_1,\ldots, s_{\ell}, t_{\ell})\,,
\eeqn
where $\psi_1,\ldots, \psi_{\ell} \in \Psi$, and $s_1,t_1,\ldots, s_{\ell}, t_{\ell}\in [k]$ that satisfy $s_i\neq t_i$ for $i\leq \ell$. In a factor graph $G=(V,F, E, (\psi_a)_{a\in F})$, for variables $v_{i_1},\ldots, v_{i_{\ell}}\in V$ and clauses $a_{j_1},\ldots, a_{j_{\ell}}\in F$, we call $\{v_{i_1},a_{j_1},\ldots, v_{i_{\ell}}, a_{j_{\ell}}\}$ a $\zeta$-cycle if it satisfies the following conditions.
\begin{itemize}
    \item The indices $i_1,\ldots, i_{\ell}\in [n]$ are distinct such that $i_1=\min\{i_1,\ldots, i_{\ell}\}$. Similarly, $j_1,\ldots, j_{\ell}\in [m]$ are distinct and $j_1=\min \{j_1,\ldots, j_{\ell}\}$.
    \item For $1\leq h\leq \ell$, the weight function $\psi_{a_{j_h}}$ assigned to the clause $a_{j_h}$ in $G$ is given by $\psi_{a_{j_h}}=\psi_h$. 
    \item For $1\leq h \leq \ell$, the $s_h$'th variable adjacent to the clause $a_{j_h}$ in $G$ is $v_{i_h}$, i.e. $\delta_{s_h}a_{j_h}=v_{i_h}$. Similarly, for $1\leq h\leq \ell-1$, $\delta_{t_h}a_{j_h}=v_{i_{h+1}}$ holds, and $\delta_{t_{\ell}}a_{j_{\ell}}=v_{i_1}$.
\end{itemize}
We let $X_{\zeta}(G)$ be the number of $\zeta$-cycles in a factor graph $G$. Also, we let $S_{\ell}$ denote the set of signatures $\zeta$ of order $\ell$, and let $S:=\cup_{\ell\geq 1} S_{\ell}$ be the set of signatures.
\end{defn}
The first condition $i_1=\min\{i_1,\ldots, i_{\ell}\}, j_1=\min \{j_1,\ldots, j_{\ell}\}$ is to avoid overcounting. Also, note that the set $S_{\ell}$ is finite for each $\ell\geq 1$ because we assumed that $\Psi$ is finite. Thus, the set $S$ is countable.

We will take $(X_{in})_{i\geq 1}=(X_{\zeta}\big(\bG(n,m_n)\big))_{\zeta\in S}$ and $Y_n=L\big(\bG(n,m_n)\big)$ in Theorem~\ref{thm:SSG}. To verify the assumptions \ref{item:A1} and \ref{item:A2}, we describe the asymptotic distribution of $X_{\zeta}(\bG(n,m_n))\big)$ and $X_{\zeta}(\bG^\star(n,m_n))\big)$. To this end, we introduce additional notations. Let the matrix $\Phi_{\psi,s,t}\equiv \big(\Phi_{\psi,s,t}(i,j)\big)_{i,j\in [q]}\in \R^{q\times q}$ for $\psi\in \Psi$ and $s,t\in[k]$ be defined by
\beq\label{eq:Phi:s:t}
\Phi_{\psi,s,t}(i,j):=\xi^{-1}\cdot \E_{\pi}\big[\psi(\bsig)\bgiven \bsigma_s=i, \bsigma_t=j\big]\cdot\pi_j\,.
\eeq
Notably, $\Phi_{\psi,1,2}=\Phi_{\psi}$ holds by definition. Moreover, for a signature $\zeta=(\psi_1,\ldots, \psi_{\ell}, s_1,t_1,\ldots, s_{\ell}, t_{\ell})\in S_{\ell}$, we let $\Phi_{\zeta}:=\prod_{i=1}^{\ell}\Phi_{\psi_i,s_i,t_i}\in \R^{q\times q}$ and define the constants $\la_{\zeta}, \la^\star_{\zeta}$ by
\beq\label{def:lambda:zeta}
\lambda_{\zeta}\equiv \lambda_{\zeta}(k,d,p):=\frac{1}{2\ell}\left(\frac{d}{k}\right)^{\ell}\prod_{i=1}^{\ell}p(\psi_i)\,,\quad\quad \lambda^\star_{\zeta}\equiv \lambda^\star_{\zeta}(k,d,\pi,p):=\lambda_{\zeta}\cdot\tr\Big(\Phi_{\zeta}\Big)\,.
\eeq
In order to obtain our results for hypergraph stochastic block models, we condition the planted factor model on a specific event $\GGG_n$ defined as follows. Let the event $\GGG_n$ consist of factor graphs $G=(V,F,E,(\psi_a)_{a\in F})$ with $|V|=n$ variables such that
\begin{enumerate}[label=(H\arabic*)]
    \item\label{item:GGG:a} Every clauses is connected to $k$ different variables, i.e. for any $a\in F$, $|\delta a|=k$.
    \item\label{item:GGG:b} For any $a,a'\in F$ with $a\neq a'$, the set of neighbors of $a$ is distinct from $a'$, i.e. $\delta a$ does not equal any permutation of $\delta a'$.
\end{enumerate}
By using the moment method, \cite[Proposition 3.12]{CEJKK:18} deduced the following result.
\begin{fact}\label{fact:cycle}
\textup{\cite[Proposition 3.12]{CEJKK:18}} Assume the condition ${\sf (SYM)}$. Then, for $|m_n-dn/k|\leq n^{2/3}$, we have the following.
\begin{enumerate}[label=\textup{(\arabic*)}]
    \item $X_{\zeta}(\bG(n,m_n))\dto X_{\zeta, \infty}$ as $n\to\infty$, jointly for all $\zeta \in \cup_{\ell\geq 1}S_{\ell}$, where $X_{\zeta,\infty}$ are independent $\Poi(\la_{\zeta})$ distributed random variables. 
    \item $X_{\zeta}(\bG^\star(n,m_n))\dto X^\star_{\zeta, \infty}$ as $n\to\infty$, jointly for all $\zeta \in \cup_{\ell\geq 1}S_{\ell}$, where $X^\star_{\zeta,\infty}$ are independent $\Poi(\la^\star_{\zeta})$ distributed random variables. 
    \item Let $C(G):=\sum_{\zeta\in S_1}X_{\zeta}(G)+\one\{k=2\}\sum_{\zeta\in S_2}X_{\zeta}(G)$. Then, note that for $k=2$, the event $\GGG_n$ equals $\{C(G)=0\}$, and in general $\GGG_n\subseteq \{C(G)=0\}$ holds. For $k\geq3$, the probability of $\bG(n,m_n)$ and $\bG^\star(n,m_n)$ satisfying \ref{item:GGG:b} is $O(1/n)$, and we have
    \begin{equation*}
    \begin{split}
        \P\big(C(\bG(n,m_n))=0\big)
        &=\P\big(\bG(n,m_n)\in \GGG_n\big)+O(1/n)\,,\\
\P\big(C(\bG^\star(n,m_n))=0\big)&=\P\big(\bG^\star(n,m_n)\in \GGG_n\big)+O(1/n)\,.
    \end{split}
    \end{equation*}
\end{enumerate}

\end{fact}
\begin{remark}
We remark that the results of \cite[Proposition 3.12]{CEJKK:18} are only stated for $\wh{\bG}(n,m_n)$ instead of $\bG^\star(n,m_n)$, $\pi=\Unif([q])$, and when $p$ satisfies an extra condition called ${\sf (BAL)}$. Here, $\wh{\bG}(n,m_n)$ is a certain `tweaked version'
of $\bG^\star(n,m_n)$ defined in Section 3.2 therein. In particular, the condition ${\sf (BAL)}$ therein guarantees that $\wh{\bG}(n,m)$ is mutually contiguous with $\bG^\star(n,m)$ (see \cite[Lemma 3.2]{CEJKK:18}). However, a cursory examination of \cite[Proposition 3.12 and Lemma 8.2]{CEJKK:18} reveals that the claim was proven in two stages: first, they calculate the joint moments of $\big(X_{\zeta} (\bG^\star(n,m_n))\big)_{\zeta\in \cup_{\ell\geq 1}S_{\ell}}$ up to $o(1)$ error (see Eq.~(8.10)) to determine its asymptotic distribution and then transalte the result to $\wh{\bG}(n,m_n)$ using the condition ${\sf (BAL)}$. Since the argument of \cite[Proposition 3.12]{CEJKK:18} works verbatim to show Fact~\ref{fact:cycle}, we omit its proof.
\end{remark}
To verify the assumptions \ref{item:A3} and \ref{item:A4}, we need the following lemma which relates the quantities $\lambda_{\zeta},\lambda^\star_{\zeta}$ in \eqref{def:lambda:zeta} and $\lambda\in \Eig_{\SS}(\Xi\circ \Lambda^{\otimes 2})$ in \eqref{def:eig:space}. Its proof is deferred to Section~\ref{sec:proof:contiguity}.
\begin{lemma}\label{lem:lambda:equality}
 For a signature $\zeta$, define the constant
\beqn
\delta_{\zeta}:=\lambda^{\star}_{\zeta}/\lambda_{\zeta}-1=\tr(\Phi_{\zeta})-1\,.
\eeqn
Assuming ${\sf (SYM)}$ and $d<d_{\ks}$, we have that
 \beqn
 \exp\bigg(\sum_{\ell\geq 1}\sum_{\zeta\in S_{\ell}}\lambda_{\zeta}\delta_{\zeta}^2\bigg)=\prod_{\lambda\in \Eig_{\SS}(\Xi)}\frac{1}{\sqrt{1-(k-1)d\lambda}}\,.
 \eeqn
\end{lemma}
By combining our second moment estimate in Proposition \ref{prop:sec:moment:LR} and small subgraph conditioning method in Theorem~\ref{thm:SSG}, we obtain the following result.

\begin{thm}\label{thm:likelihood:conv}
    For average degree $d$ below the weak recovery threshold and the Kesten-Stigum threshold $d<d_{\ast}\wedge d_{\ks}$, the likelihood ratio evaluated at the null model $\bG\equiv\bG(n,\bm)$ converges in distribution to 
    \beq\label{eq:likelihood:conv}
    L(\bG)\dto \boldsymbol{L}_{\infty}:= \prod_{\ell \geq 1}\prod_{\zeta\in S_{\ell}}\Big\{(1+\delta_{\zeta})^{X_{\zeta, \infty}}e^{-\lambda_{\zeta}\delta_{\zeta}}\Big\}\,,
\eeq
where $X_{\zeta,\infty}\sim \Poi(\la_{\zeta})$ are independent Poisson random variables and $\delta_{\zeta}\equiv \la_{\zeta}^{\star}/\la_{\zeta}-1$. Moreover, the convergence in \eqref{eq:likelihood:conv} holds jointly with the convergence in Fact~\ref{fact:cycle}-(1).
\end{thm}
\begin{proof}
By Proposition~\ref{prop:nontrivial:threshold}, we assume ${\sf (SYM)}$ without loss of generality. Fix any sequence $(m_n)_{n\geq 1}$ such that $|m_n-dn/k|\leq n^{2/3}$ holds. We use Theorem~\ref{thm:SSG} by setting $Y_n\equiv L^{\ast}\big(\bG(n,m_n)\big)$ and $(X_{in})_{i\geq 1}\equiv \big(X_{\zeta}\big(\bG(n,m_n)\big)\big)_{\zeta\in S}$. To this end, we check the assumptions \ref{item:A1}-\ref{item:A4}. Firs, Fact~\ref{fact:cycle}-(1) implies the first assumption \ref{item:A1}. In addition, Fact~\ref{fact:cycle}-(2) and a simple change of measure guarantee that for a finite sequence of non-negative integers $x_1,\ldots x_{L}$, we have
\begin{equation*}
\E\Big[L\big(\bG(n,m_n)\big)\one\big\{X_{\zeta_1}=x_1,\ldots, X_{\zeta_L}=x_L\big\}\Big]\longrightarrow \prod_{i=1}^{L}\P\big(X^\star_{\zeta_i \infty}=x_{i}\big)\,.
\end{equation*}
Meanwhile, Lemma~\ref{lem:m:n:overlap:trivial} shows that the LHS equals $\E\big[L^{\ast}\big(\bG(n,m_n)\big)\one\big\{X_{\zeta_1}=x_1,\ldots, X_{\zeta_L}=x_L\big\}\big]$ up to $o_n(1)$ error. Moreover, $\E L^{\ast}\big(\bG(n,m_n)\big)=1-o_n(1)$ holds by Lemma~\ref{lem:m:n:overlap:trivial} since $\E L\big(\bG(n,m_n)\big)=1$. Thus, it follows that
\begin{equation*}
\frac{\E\Big[L^{\ast}\big(\bG(n,m_n)\big)\one\big\{X_{\zeta_1}=x_1,\ldots, X_{\zeta_L}=x_L\big\}\Big]}{\E L^{\ast}\big(\bG(n,m_n)\big)}\longrightarrow \prod_{i=1}^{L}\P\big(X^\star_{\zeta_i \infty}=x_{i}\big)\,,
\end{equation*}
which verifies~\ref{item:A2}. \ref{item:A3} is immediate from Lemma~\ref{lem:lambda:equality}. Furthermore, for $d<d_{\ks}$, Proposition~\ref{prop:sec:moment:LR} shows that
\beqn
\frac{\E\Big(L^{\ast}\big(\bG(n,m_n)\big)\Big)^2}{\Big(\E L^{\ast}\big(\bG(n,m_n)\big)\Big)^2}\leq \big(1+o_n(1)\big)\E\Big(L^{\ast}\big(\bG(n,m_n)\big)\Big)^2\leq \big(1+o_n(1)\big)\prod_{\lambda\in \Eig_{\SS}(\Xi\circ \Lambda^{\otimes 2})}\frac{1}{\sqrt{1-(k-1)d\lambda}}\,,
\eeqn
where the first inequality holds since $\E L^{\ast}\big(\bG(n,m_n)\big)=1-o_n(1)$. Since Lemma~\ref{lem:lambda:equality} shows that the final product equals $(1+o_n(1))\exp\big(\sum_{\ell\geq 1}\sum_{\zeta\in S_{\ell}}\lambda_{\zeta}\delta_{\zeta}^2\big)$, \ref{item:A4} also holds. Therefore, by Theorem \ref{thm:SSG}, we have for $d<d_{\ks}$ that
\beq\label{eq:truncated:LR:convergence}
\frac{L^{\ast}\big(\bG(n,m_n)\big)}{\E L^{\ast}\big(\bG(n,m_n)\big)}\dto \boldsymbol{L}_{\infty}\equiv \prod_{\ell \geq 1}\prod_{\zeta\in S_{\ell}}\Big\{(1+\delta_{\zeta})^{X_{\zeta \infty}}e^{-\lambda_{\zeta}\delta_{\zeta}}\Big\}\,,
\eeq
and this convergence holds jointly with the convergence of $X_{\zeta}\big(\bG(n,m_n)\big)$ in Fact~\ref{fact:cycle}-(1). For $d<d_{\ast}\wedge d_{\ks}$, Lemma~\ref{lem:m:n:overlap:trivial} allows us to translate this convergence to the convergence of $L(\bG(n,m_n))$: 
\beqn
L\big(\bG(n,m_n)\big)\dto \boldsymbol{L}_{\infty}\,,
\eeqn
where the convergence is also joint with the convergence in Fact~\ref{fact:cycle}-(1). Finally, since $|\bm-dn/k|\leq n^{2/3}$ holds with probability tending to one by a Chernoff bound, this concludes the proof.
\end{proof}

\begin{proof}[Proof of Theorem \ref{thm:factor:contiguity}]
Let $d<d_{\ast}\wedge d_{\ks}$. First, observe that the likelihood ratio of $\bG^{\star}\equiv \bG^\star(n,\bm)$ and $\bG\equiv \bG(n,\bm)$ is given $L(\cdot)$ since for a factor graph $G$ with $n$ variables and $m$ clauses, we have
\beq\label{eq:proof:mutual:info:tech}
\frac{\P(\bG^\star=G)}{\P(\bG=G)}=\frac{\P(\bG^\star(n,m)=G)\P(\bm=m)}{\P(\bG(n,m)=G)\P(\bm=m)}=L(G)\,.
\eeq
Moreover, Theorem~\ref{thm:likelihood:conv} guarantees that for $d<d_{\ast}\wedge d_{\ks}$, the convergence $L(\bG)\dto \bL_{\infty}$ holds as $n\to\infty$. Note that $\E \bL_{\infty}=1$ holds since $X_{\zeta,\infty}\sim \Poi(\la_{\zeta})$ are independent Poisson random variables. In addition, $\bL_{\infty}>0$ holds  a.s.. Therefore, Le Cam's first lemma (see e.g.~\cite[Lemma 6.4]{Vaart98} or \cite[Proposition 9.49]{Janson00random}) implies that $\bG\equiv \bG(n,\bm)$ and $\bG^\star\equiv \bG^\star(n,\bm)$ are mutually contiguous.
\end{proof}

\subsubsection{Consequences in hypothesis testing}
We next state the consequences of likelihood ratio convergence in hypothesis testing. To do so, we need the the following result, which shows that the distribution of $\bL_{\infty}$ contains a point mass depending on the average degree $d>0$. Its proof is deferred to Section~\ref{sec:proof:hypothesis:estimation}.
\begin{lemma}
\label{lem:dist:L:infty}
    In the subcritical regime $d\in (0,\frac{1}{k-1})$, the support of the random variable $\bL_{\infty}$ equals the countable set
    \begin{equation}\label{eq:support.set}
          \Big\{z: z=\prod_{\ell \geq 1}\prod_{\zeta\in S_{\ell}}(1+\delta_{\zeta})^{x_\zeta}e^{-\lambda_{\zeta}\delta_{\zeta}}~\textnormal{ for some $x=(x_{\zeta})_{\ell\geq 1, \zeta\in S_{\ell}}$ with finitely many non-zero elements.}\Big\}\,.
    \end{equation}
    On the contrary, in the critical or subcritical regime $d\in [\frac{1}{k-1}, d_{\ks})$, the random variable $\bL_{\infty}$ does not have a point mass.
\end{lemma}
As an immediate consequence of Theorem~\ref{thm:likelihood:conv} and Lemma~\ref{lem:dist:L:infty}, we have the following corollary.
\begin{cor}\label{cor:power:factor}
    In the subcritical regime $d\in (0,\frac{1}{k-1})$, the asymptotic power of non-randomized likelihood ratio test for the hypothesis test~\eqref{eq:hypothesis:test:facotr} at significance level $\alpha\in (0,1)$ is given by
    \beq\label{eq:def:power:factor}
     \beta_{\ast}(\alpha,d,p,\pi):=\E\big[\bL_{\infty}\one\{\bL_{\infty}> C_{\alpha}\}\big]\,.
    \eeq
     Here, $C_{\alpha}>0$ is an arbitrary constant that satisfy
     \beqn
     \P(\bL_{\infty}>C_{\alpha})\leq \alpha \leq \P(\bL_{\infty}\geq C_{\alpha})\,.
     \eeqn
     In the critical or supercriticial regime $d\in [\frac{1}{k-1},d_{\ks})$, the asymptotic power of (possibly randomized) likelihood ratio test for the hypothesis test~\eqref{eq:hypothesis:test:facotr} at significance level $\alpha\in (0,1)$ is given by $\beta_{\ast}(\alpha,d,p,\pi)$.
\end{cor}
Next, we consider computationally efficient hypothesis using the following statistic based on $\zeta$-cycles:
\beqn
T_n(G):=\prod_{1\leq \ell \leq K_n}\prod_{\zeta\in S_{\ell}}\Big\{(1+\delta_{\zeta})^{X_{\zeta}(G)}e^{-\lambda_{\zeta}\delta_{\zeta}}\Big\}\,,
\eeqn
where the truncation parameter is taken so that $1\ll K_n=O(\log\log n)$. Then, we prove the following lemma in Section~\ref{sec:proof:hypothesis:estimation} by utilizing Fact~\ref{fact:cycle}.
\begin{lemma}\label{lem:T:conv}
    Under the condition~{\sf (SYM)}, let $d<d_{\ks}$ and $1\ll K_n=O(\log\log n)$. Then, under the null model $\bG\sim \GG_{\sf null}(n,d,p)$,
    \beqn
    T_n(\bG)\dto \bL_{\infty}\equiv \prod_{\ell \geq 1}\prod_{\zeta\in S_{\ell}}\Big\{(1+\delta_{\zeta})^{X_{\zeta,\infty}}e^{-\lambda_{\zeta}\delta_{\zeta}}\Big\}\,,
    \eeqn
   where $X_{\zeta,\infty}\sim \Poi(\la_{\zeta})$ are independent Poisson random variables. Under the planted model $\bG^\star\sim \GG_{\sf plant}(n,d,p,\pi)$,
   \beqn
   T_n(\bG^\star)\dto \bL^\star_{\infty}\equiv \prod_{\ell\geq 1} \prod_{\zeta\in S_{\ell}}\Big\{(1+\delta_{\zeta})^{X^\star_{\zeta,\infty}}e^{-\lambda_{\zeta}\delta_{\zeta}}\Big\}\,,
   \eeqn
   where $X^\star_{\zeta, \infty}\sim \Poi(\la^\star_{\zeta})$ are independent Poisson random variables. 
\end{lemma}
We note that the random variable $\bL^\star_{\infty}$ is well-defined for $d<d_{\ks}$, i.e. the infinite product over $\ell \geq 1$ converges almost surely, since $\bL_{\infty}$ is well-defined (cf. Remark~\ref{rmk:L:infty:well:defined}) with mean $1$ from which $\bL^\star_{\infty}$ can be obtained by a change of measure (see Eq.~\eqref{eq:L:infty:change} below). As a corollary, we have the following result.
\begin{cor}\label{cor:cycle:test:factor}
    Under the condition~{\sf (SYM)}, let $1\leq d<d_{\ks}$ and $1\ll K_n=O(\log\log n)$. At significance level $\alpha\in (0,1)$, consider the non-randomized test $\varphi_{n,\alpha}(\cdot)$ which rejects the null ${\sf H_0}$ in Eq.~\eqref{eq:hypothesis:test:facotr} with probability 
\beqn
\varphi_{n,\alpha}(G):=
\begin{cases}
    1 & \;\;\;\;\textnormal{if}\;\;\;\;\; T_n(G)> C_{n,\alpha}\,;\\
    0 & \;\;\;\;\textnormal{otherwise}\,.
\end{cases}
\eeqn
Here, the constants $C_{n,\alpha}>0$ is chosen so that we have
\beq\label{eq:C:n:alpha}
\P(T_n(\bG)>C_{n,\alpha})\leq \alpha\leq \P(T_n(\bG)\geq C_{n,\alpha})\,,
\eeq
where $\bG\sim  \GG_{\sf null}(n,d,p)$. Then, for $1\leq d<d_{\ks}$, the test $\varphi_{n,\alpha}$ achieves the power $\beta_{\ast}(\alpha,d,p,\pi)$ defined in Eq.~\eqref{eq:def:power:factor}. That is, under the alternative $\bG^\star\sim \GG_{\sf plant}(n,d,p,\pi)$, we have
\beqn
\E\varphi_{n,\alpha}(\bG^\star)\to \beta_{\ast}(\alpha)\;\;\;\;\textnormal{as}\;\;\;\; n\to\infty\,.
\eeqn
In particular, in the regime $1\leq d <d_{\ast}\wedge d_{\ks}$, the test $\varphi_{n,\alpha}(\cdot)$ is asymptotically most powerful for the hypothesis test~\eqref{eq:hypothesis:test}.
\end{cor}
\begin{proof}
Let $1\leq d<d_{\ks}$. By passing to a subsequence, let $\tilde{C}_{\alpha}\in [-\infty,\infty]$ be the limit of $(C_{n,\alpha})_{n\geq 1}$. We first claim that $\P(\bL_{\infty}>\tilde{C}_{\alpha})=\alpha$. To see this, note that by Lemma~\ref{lem:T:conv} and Skohorod embedding, there exists a coupling such that $T_n(\bG)$ converges to $\bL_{\infty}$ almost surely. Moreover, $\P(\bL_{\infty}= \tilde{C}_{\alpha})=0$ holds by Lemma~\ref{lem:dist:L:infty}. Thus, it follows that
\begin{equation*}
    \one\{T_n(\bG)>C_{n,\alpha}\}\to \one\{\bL_{\infty}>\tilde{C}_{\alpha}\}\;\;\;\;\textnormal{as}\;\;\;\;n\to\infty\;\;\;\textnormal{almost surely.}
\end{equation*}
Similarly, $\one\{T_n(\bG)\geq C_{n,\alpha}\}$ converges to $\one\{\bL_{\infty}>\tilde{C}_{\alpha}\}=\one\{\bL_{\infty}\geq \tilde{C}_{\alpha}\}$ a.s.. Hence, by dominated convergence theorem, sending $n\to\infty$ in \eqref{eq:C:n:alpha} yield that
\begin{equation}\label{eq:thresholds:conv}
    \alpha =\lim_{n\to\infty}\P(T_n(\bG)>C_{n,\alpha})=\lim_{n\to\infty}\P(T_n(\bG)\geq C_{n,\alpha})= \P(\bL_{\infty}>\tilde{C}_{\alpha})\,,
\end{equation}
which proves our first claim. Next, note that for any Borel measurable set $A\subset \R$, we have
\beq\label{eq:L:infty:change}
\begin{split}
\E\big[\bL_{\infty}\one\{\bL_{\infty}\in A\}\big]
&\equiv \sum_{x=(x_{\zeta})_{\zeta}}\one\Big\{\prod_{\zeta}(1+\delta_{\zeta})^{x_{\zeta}}e^{-\la_{\zeta}\delta_{\zeta}}\in A\Big\}\prod_{\zeta}\Big\{(1+\delta_{\zeta})^{x_{\zeta}}e^{-\la_{\zeta}\delta_{\zeta}}\P(X_{\zeta,\infty}=x_{\zeta})\Big\}\\
&=\sum_{x=(x_{\zeta})_{\zeta}}\one\Big\{\prod_{\zeta}(1+\delta_{\zeta})^{x_{\zeta}}e^{-\la_{\zeta}\delta_{\zeta}}\in A\Big\}\prod_{\zeta}\frac{(\la^\star_{\zeta})^{x_{\zeta}}}{x_{\zeta}!}e^{-\la^\star_{\zeta}}\\
&=\P\big(\bL^\star_{\infty}\in A\big)\,,
\end{split}
\eeq
where the second equality holds since $1+\delta_{\zeta}\equiv \la^\star_{\zeta}/\la_{\zeta}$. In particular, the distribution of $\bL^\star_{\infty}$ does not have a point mass by Lemma~\ref{lem:dist:L:infty}. Moreover, by repeating the argument in deriving \eqref{eq:thresholds:conv}, Lemma~\ref{lem:T:conv} shows that
\beqn
\lim_{n\to\infty}\E \varphi_{n,\alpha}(\bG^\star) =\lim_{n\to\infty}\P(T_n(\bG^\star)>C_{n,\alpha})= \P(\bL_{\infty}^\star >\tilde{C}_{\alpha})\,.
\eeqn
Meanwhile, the final term equals $\P(\bL_{\infty}^\star >\tilde{C}_{\alpha})=\E[\bL_{\infty}\one\{\bL_{\infty}>\tilde{C}_{\alpha}\}]$ by \eqref{eq:L:infty:change}. Since $\P(\bL_{\infty}>\tilde{C}_{\alpha})=\alpha$ holds by our first claim, the final term further equals $\beta_{\ast}(\alpha,d,p,\pi)$, which concludes the proof. 
\end{proof}
\subsubsection{Proof of Corollary~\ref{cor:no:consistent}}
\label{subsubsec:cor}
We establish the impossibility of point estimation in Corollary~\ref{cor:no:consistent} based on the mutual contiguity in Theorem~\ref{thm:factor:contiguity} and a {\em resampling argument}. Given a factor graph $G=(V,F, E, \{\psi_a\}_{a\in F})$ and $\ga\in [0,1]$, consider the $\ga$-resampling procedure as follows. For each clause $a\in F$, toss a coin independently with probability $\gamma$ and let $(Z_a)_{a\in F}\iid \Ber(\gamma)$. If $Z_a=1$, then independently resample the neighborhood $\delta a$ and the factor $\psi_a$ of the clause $a$ from the ``null distribution'':
\beq\label{eq:clause:null}
\P\big(\delta a= (v_1,\ldots, v_k)\,,\,\psi_a=\psi\big)=\frac{p(\psi)}{n^k}\,,
\eeq
and if $Z_a=0$, keep $(\delta a, \psi_a)$ unchanged. We denote by $G_{\ga}$ the resulting factor graph obtained from $G$ by the $\ga$-resampling procedure. For instance, $\bG^\star_{\ga}$ denotes the factor graph obtained from the planted model $\bG^\star\sim \GG_{\sf plant}(n,d,p,\pi)$ by the $\ga$-resampling procedure. Note that since $\ga$ fraction of clauses are resampled according to the null distribution, the probability distribution of $\bG^\star_{\ga}$ only gets closer to that of the null model $\bG$ as $\ga$ becomes large. Indeed, in the extreme case $\ga=1$, $\bG^\star_1$ is distributed the same as $\bG$. Thus, the family $(\bG^\star_{\ga})_{\ga\in [0,1]}$ is a continuous interpolation of the planted and the null model. 

Our next result, which is used for the proof of Corollary~\ref{cor:no:consistent}, shows that $\bG^\star_{\ga}$ is contiguous to another planted model described as follows.  Let $\Psi=\{\psi_1,\ldots, \psi_T\}$ and $p\in \PPP(\Psi)$. Then, let $\Psi_{\ga}$ and $p_{\ga}\sim \PPP(\Psi_{\ga})$ be their $\ga$-modifications defined by
\beq\label{eq:modified:weights}
\Psi_{\ga}:=\{\psi_{1,\ga},\ldots,\psi_{T,\ga}\}\,,\;\;\;\;\textnormal{where}\;\;\;\;\psi_{t,\ga}(\cdot ):= (1-\ga)\psi_t(\cdot )+\ga \xi\,,
\eeq
and $p_{\ga}(\psi_{t,\ga}):= p(\psi_t)$ for $1\leq t\leq T$. Here, we recall the notation $\xi\equiv \E_{p,\pi}[\bpsi(\bsig)]$. The proof of the following lemma is deferred to Section~\ref{sec:proof:hypothesis:estimation}.

\begin{lemma}\label{lem:resampling}
    Suppose that the priors $p\in \PPP(\Psi)$ and $\pi\in \PPP([q])$ satisfy ${\sf (SYM)}$, and consider the planted model $\bG^\star\sim \GG_{\sf plant}(n,d,p,\pi)$. Then, for $\ga\in [0,1]$, the $\ga$-resampled factor graph $\bG^\star_{\ga}$ is mutually contiguous to the planted distribution with $\ga$-modified weight functions $\GG_{\sf plant}(n,d,p_{\ga},\pi)$.   
\end{lemma}
\begin{proof}[Proof of Corollary~\ref{cor:no:consistent}]
Fix $\eps>0$ and $d<d_{\ast}\wedge d_{\ks}$. By Proposition~\ref{prop:nontrivial:threshold}, w.l.o.g. we assume ${\sf (SYM)}$.

We first argue that the $\ga$-resampled graph $\bG^\star_{\ga}$ of the planted model $\bG^\star\equiv \bG^\star(n,\bm,\bsig^\star)\sim \GG_{\sf plant}(n,d,p,\pi)$ is mutually contiguous to the null model $\bG\sim \GG_{\sf null}(n,d,p)$ for any fixed $\ga\in [0,1]$. To see this, recall that $\bG^\star$ has $\bm\sim \Poi(dn/k)$ clauses. By Poisson thinning, $\bG^\star_{\ga}$ has $\bm_{1}\sim \Poi((1-\ga)dn/k)$ clauses distributed as Eq.~\eqref{eq:clause:law:planted} and $\bm_2$  clauses distributed as Eq.~\eqref{eq:clause:null}. Thus, the likelihood ratio (a.k.a. Radon-Nikodym derivative) of $\bG^\star_{\ga}$ and $\bG\sim \GG_{\sf null}(n,d,p)$ is the same as the likelihood ratio of $(\bG^\star)^\prime \sim \GG_{\sf plant}(n,(1-\ga)d, p,\pi)$ and $\bG^\prime\sim \GG_{\sf null}(n,(1-\ga)d,p)$. Noting that $(1-\ga)d<d_{\ast}(p,\pi)\wedge d_{\ks}$ holds, combining this equality between likelihood ratios with Theorem~\ref{thm:factor:contiguity} proves that $\bG^\star_{\ga}$ is mutually contiguous to $\bG$.

Hence, by Lemma~\ref{lem:resampling}, the planted distribution with $\ga$-modified weight functions $\GG_{\sf plant}(n,d,p_{\ga},\pi)$ is mutually contiguous to the null model $\bG$ for any $\ga\in [0,1]$. Note that by taking $\ga>0$ small enough compared to $\eps$, the weights $\psi_{t,\ga}$ defined in \eqref{eq:modified:weights} satisfy $\|\psi_{t,\ga}-\psi_t\|_{\infty}\leq\eps$ for any $1\leq t \leq T$. Therefore, the mutual contiguity of $\GG_{\sf plant}(n,d,p_{\ga},\pi)$ and $\GG_{\sf plant}(n,d,p_{\ga'},\pi)$ for sufficiently small $\ga\neq \ga'$ imply the non-existence of $\eps$-locally consistent estimator of $\Psi$ at $(d,p,\pi)$ (cf. Definition~\ref{def:local:consistent}).
\end{proof}

\subsection{Proof of Theorem \ref{thm:mutual:info}}
\label{subsec:proof:mutual:info}
The key ingredient to prove Theorem \ref{thm:mutual:info} is to analyze the (normalized) log likelihood ratio evaluated at the planted model $\frac{1}{n}\log L\big(\bG^\star\big)$. Its relationship between the KL divergence and the mutual information is given by the following lemma.

\begin{lemma}\label{lem:KL:mutual:info:free:energy}
We have the equality $\DKL(\bG^\star\,\|\,\bG)=\E\log L(\bG^\star)$. Moreover, the normalized mutual information can be approximated by
\beqn
\frac{1}{n}I(\bG^\star,\bsig^\star)=-\frac{1}{n}\E\log L(\bG^\star)+\frac{d}{k}\cdot\E_{p,\pi}\bigg[\frac{\bpsi(\bsig)}{\xi}\log\Big(\frac{\bpsi(\bsig)}{\xi}\Big)\bigg]+o_n(1)\,.
\eeqn
\end{lemma}
Thus, to prove Theorem~\ref{thm:mutual:info}, it suffices to analyze $\frac{1}{n}\E \log L\big(\bG^\star\big)$. By borrowing terminology from statistical physics, we call $\frac{1}{n} \log L\big(\bG^\star\big)$ the \textit{free energy} since $L(G)$ can be interpreted as a partition function by the equality
 \beqn
 L(G)\equiv \frac{\P\big(\bG^\star(n,m,\bsig^\star)=G\big)}{\P\big(\bG(n,m)=G\big)}=\sum_{\sig \in [q]^V}\frac{\psi_G(\sig)}{\E[\psi_{\bG(n,m)}(\sig)]}\cdot \P(\bsig^\star=\sig)\,.
 \eeqn

Because the weight functions $\psi\in \Psi$ are strictly positive, it is well-known that the free energies concentrate tightly around their expectation  (see e.g. \cite[Lemma 3.3]{CKPZ:18}). 
\begin{fact}\label{fact:concentration}
    There exists a constant $c=c(\Psi)>0$ that only depends on the set of weight functions $\Psi$ such that the following hold. For any $\eps>0$, we have
    \beqn
    \P\Bigg(\bigg|\frac{1}{n}\log L(\bG^\star)-\frac{1}{n}\E\log L(\bG^\star)\bigg|\geq \eps\Bigg)\leq e^{-cn\eps^2}\,.
    \eeqn
\end{fact}
\begin{proof}
    Note that if we have two factor graphs $G$ and $G^\prime$, which differ only in a single clause, then $\big|n^{-1}\log L(G)-n^{-1}\log L(G^\prime)\big|\leq C$ holds, where $C\equiv C(\Psi)>0$. Since conditional on $\bsig^\star$, the clauses in $\bG^\star$ are independent, the desired inequality is immediate from Azuma-Hoeffding inequality applied to to the Doob martingale w.r.t. the clause revealing filtration.
\end{proof}
The next step is to calculate derivative of the expected free energy w.r.t. $d$. Recall that $\langle \cdot \rangle_{G}$ denotes the expectation with respect to samples $(\sig^{\ell})_{\ell\geq 1} \iid \mu_{G}$ from the posterior. If there are only one sample, i.e. $\ell=1$, we omit the superscript and write $\langle f(\sig^1)\rangle_G\equiv \langle f(\sig)\rangle_G$ for simplicity.
\begin{prop}\label{prop:derivative}
    For any $d>0$ and $n\geq 1$, the derivative of the expected free energy is given by
    \beq\label{eq:prop:derivative:first}
    \frac{1}{n}\frac{\partial}{\partial d}\E\log L(\bG^\star)=\frac{1}{k}\cdot\E_{\bG^\star,p,u}\bigg[\bigg\langle \frac{\bpsi(\sig_{\bom})}{\E_{p,u}\big[\bpsi^\prime(\sig_{\bom^\prime})\big]}\bigg\rangle_{\bG^\star}\log\bigg\langle \frac{\bpsi(\sig_{\bom})}{\E_{p,u}\big[\bpsi^\prime(\sig_{\bom^\prime})\big]}\bigg\rangle_{\bG^\star}\bigg]\,,
    \eeq
    where the expectation $\E_{\bG^\star,p,u}$ is with respect to $\bG^\star, \bpsi \sim p$, and $\bom \sim u\equiv \Unif(V^k)$. The expectation $\E_{p,u}$ in the denominator is with respect to $\bpsi^\prime \sim p$ and $\bom^\prime \sim u$.
    Furthermore, we can approximate
    \beq\label{eq:prop:derivative}
     \frac{1}{n}\frac{\partial}{\partial d}\E\log L(\bG^\star)=\frac{1}{k}\cdot\E_{\bG^\star,p,u}\Bigg[ \frac{\big\langle \bpsi(\sig_{\bom})\big\rangle_{\bG^\star}}{\xi}\cdot \log\bigg(\frac{\big\langle \bpsi(\sig_{\bom})\big\rangle_{\bG^\star}}{\xi}\bigg)\Bigg]+O_{k,q,\Psi}(n^{-1/3})\,,
    \eeq
    where $f=O_{k,q,\Psi}(g)$ if there exists a constant $C$ that only depends on $k,q,\Psi$ such that $|f|\leq Cg$.
\end{prop}
The first equality~\eqref{eq:prop:derivative} can be interpreted as the analog of I-MMSE relation from information theory~\cite{guo2011estimation} within the planted factor model. By appealing to Theorem \ref{thm:equiv}, we prove that below the weak recovery threshold $d<d_{\ast}$, the RHS of \eqref{eq:prop:derivative} tends to $0$ as $n\to\infty$, whereas for $d>d_{\ast}$ it is uniformly positive along a subsequence under the ${\sf (MIN)}$ condition. As a consequence, we prove the following Proposition, which plays a crucial role in the proof of Theorem~\ref{thm:mutual:info}.
\begin{prop}\label{prop:free:energy}
The following holds.
\begin{enumerate}[label=\textup{(\arabic*)}]
    \item For any $n\geq 1$ and $d>0$, we have $\frac{\partial}{\partial d}\E \log L(\bG^\star)\geq 0$ and $\E\log L(\bG^\star)\geq 0$. 
    \item If $d<d_{\ast}$, then we have
    \beqn
    \frac{1}{n}\E\log L(\bG^\star) \longrightarrow 0\quad\textnormal{as}\quad n\to\infty\,.
    \eeqn
    \item Assume that the condition {\sf (MIN)} holds. Then, for any $\eps>0$, there exists $\eta\equiv \eta(\eps)>0$ and $n_0\equiv n_0(\eps)$ not depending on $n$ nor $d>0$ such that if $\E\big[A(\bsig^\star, \hat{\sig})\big]\geq \frac{1}{q}+\eps$ holds for some estimator $\hat{\sig}\equiv \hat{\sig}_n(\bG^\star)$ and $n\geq n_0$, then $\frac{1}{n}\frac{\partial}{\partial d}\E\log L(\bG^\star)\geq \eta$ holds. 
    \item If $d>d_{\ast}$ and {\sf (MIN)} holds, then there exists a constant $\eta>0$ such that
    \beqn
    \limsup_{n\to\infty} \frac{1}{n}\E \log L(\bG^\star)\geq \eta>0
    \eeqn

\end{enumerate}
\end{prop}
The proofs of Lemma~\ref{lem:KL:mutual:info:free:energy} and Propositions~\ref{prop:derivative},~\ref{prop:free:energy} are deferred to Section~\ref{sec:mutual:info}. Here, we prove that they imply Theorem~\ref{thm:mutual:info}.

\begin{proof}[Proof of Theorem \ref{thm:mutual:info}]
The first statement is an immediate consequence of Lemma~\ref{lem:KL:mutual:info:free:energy} and Proposition~\ref{prop:free:energy}-(2) while the second statement follows from Lemma~\ref{lem:KL:mutual:info:free:energy} and Proposition~\ref{prop:free:energy}-(4).

To prove the third statement, assume the condition ${\sf (MIN)}$. Suppose that for some $d_0>0$ and $\eps>0$, $\E\big[A(\bsig^\star,\hat{\sig})\big]\geq \frac{1}{q}+\eps$ holds for some estimator $\hat{\sig}\equiv \hat{\sig}_n(\Breve{\bG^\star})$ where $\Breve{\bG^\star}\sim \GG_{\sf plant}(n,d_0,p,\pi)$. Then, note that for all $d>d_0$, $\E\big[A(\bsig^\star,\hat{\sig})\big]\geq \frac{1}{q}+\eps$ holds for some estimator $\hat{\sig}\equiv \hat{\sig}_n(\bG^\star)$ where $\bG^\star\sim \GG_{\sf plant}(n,d,p,\pi)$. This is because starting from $\bG^\star$, subsampling each clause independently with probability $d_0/d$ yields a sample drawn from $\GG_{\sf plant}(n,d_0,p,\pi)$ by Poisson thinning. Thus, if we let $\eta\equiv \eta(\eps)$ and $n_0\equiv n_0(\eps)$ be as in Proposition~\ref{prop:free:energy}-(3), then $\frac{1}{n}\frac{\partial}{\partial d} \E \log L(\bG^\star)\geq \eta$ holds for all $d>d_0$ and $n\geq n_0$. Since in general $\frac{1}{n}\frac{\partial}{\partial d} \E \log L(\bG^\star)\geq 0$ holds by Proposition~\ref{prop:free:energy}-(1) and $\eta$ does not depend on $d$, it follows that for all $n\geq n_0$,
\[
\frac{1}{n}\E \log L(\bG^\star)\geq \eta(d-d_0)\equiv \eta_{d}>0\,.
\]
Then, consider the event
\beqn
\AAA_n^{\ast}:=\big\{G:L(G)\geq e^{\eta_{d} n/2}\big\}\,.
\eeqn
Then by Fact~\ref{fact:concentration}, there exists a constant $c>0$ such that for $n\geq n_0$,
\beqn
\P(\bG^\star\in \AAA^{\ast}_n)\geq \P\bigg(\bigg|\frac{1}{n}\log L(\bG^\star)-\frac{1}{n}\E\log L(\bG^\star)\bigg|\leq \frac{\eta}{2}\bigg)\geq 1-e^{-cn\eta_d^2}\,.
\eeqn
Meanwhile, note that $\E L(\bG)=1$ holds since $L(\cdot)$ is Radon-Nikodym derivative of the law of $\bG^\star$ and $\bG$. Thus, we have by Markov's inequality that for any $n\geq 1$,
\begin{equation*}
\P(\bG\in \AAA^{\ast}_n)=\P\big(L(\bG)\geq e^{\eta_d n/2}\big)\leq e^{-\eta_d n/2}\,.
\end{equation*}
Therefore, by adjusting $\eta$ appropriately to account for small $n<n_0$, the desired claim~\eqref{eq:exp:orthogonal} holds for any $n\geq 1$ under the stated assumptions, concluding the proof.
\end{proof}

\section{Equivalent notions of weak recovery}
\label{sec:proof:equiv}
In this section, we prove Theorem~\ref{thm:equiv}. We will prove a generalized statement, where $\bG^\star\equiv \bG^\star(n,\bm,\bsig^\star,p,\pi)$ is replaced by an arbitrary random variable $\bX^\star\equiv \bX^\star_n$ defined on the same probability space as $\bsig^\star$ and satisfy a certain condition ${\sf (EXG)}$. This generalization is stated in Proposition~\ref{prop:equiv}, which will be useful for the proof of Lemma~\ref{lem:m:n:overlap:trivial}, but it is interesting in its own right. We first specify the task of weak recovery with respect to $\bX^\star_n$, extending Definition~\ref{def:weak:recovery}. Recall the definition of $A(\sig^1,\sig^2)$ in \eqref{eq:def:overlap}.
\begin{defn}\label{def:weak:recovery:general}
    Let $(\bX^\star_n)_{n\geq 1}$ be a sequence of random variables which are defined on the same probability space as $\bsig^\star$ and take values in a probability space $\XXX$. We say that weak recovery is possible for $(\bX^\star_n)_{n\geq 1}$ if there exists $\eps>0$ and an estimator (i.e. a measurable function) $\hat{\sig}\equiv \hat{\sig}(\bX^\star_n)$ that takes as an input $\bX^\star_n$ and returns $\hat{\sig}\in [q]^V$ such that $\E[A(\bsig^\star,\hat{\sig})]\geq \frac{1}{q}+\eps$ for large enough $n\geq 1$. Otherwise, we say that the weak recovery is impossible for $(\bX^\star_n)_{n\geq 1}$.
\end{defn}

Given `data' $\bX^\star_n$, we define the `posterior' $\mu_X\in \PPP([q]^V)$ as 
\beq\label{eq:posterior:X}
\mu_X(\sig)= \P\big(\bsig^\star=\sig \bgiven \bX^\star=X\big)\,,\quad \sig \in [q]^V\,.
\eeq
We denote by $\langle \cdot \rangle_{X}$ the expectation with respect to samples $(\sigma^{\ell})_{\ell\geq 1}\iid \mu_X$ from from the posterior.

The condition we impose on the data $(\bX^\star_n)_{n\geq 1}$ is
\begin{itemize}
    \item Exchangeability: for any $n\geq 1$ and any pairs of distinct variables $(u_1,v_1)$ and $(u_2,v_2)$ with $u_i\neq v_i, i=1,2$, 
\beq\label{eq:exg}
{\sf (EXG):}~\Big(\P\big(\bsigma^\star_{u_1}=i, \bsigma^\star_{v_1}=j\bgiven \bX^\star_n\big)\Big)_{i,j\in [q]}\stackrel{d}{=}\Big(\P\big(\bsigma^\star_{u_2}=i, \bsigma^\star_{v_2}=j\bgiven \bX^\star_n\big)\Big)_{i,j\in [q]}
\eeq
\end{itemize}
By the symmetry of variables, the planted factor model satisfies the property {\sf (EXG)} as verified by the next lemma.
\begin{lemma}\label{lem:EXG:CONV}
        For any sequence of $(m_n)_{n\geq 1}$, if we let $\bX^\star_n=\bG^\star(n,m_n)$, then the sequence of random variables $(\bX^\star_n)$ satisfy {\sf (EXG)}. In particular, $\bX^\star_n=\bG^\star$ satisfies {\sf (EXG)}.
\end{lemma}
\begin{proof}
Fix two sets of distinct variables $(u_i,v_i)$, $i=1,2$. We show that for a continuous and bounded test function $f:\R^{q\times q}\to \R$,
\beq\label{eq:lem:EXG:goal}
\E f\Big(\Big\{\P\big(\bsigma^\star_{u_1}=i, \bsigma^\star_{v_1}=j\given \bG^\star(n,m)\big)\Big\}_{i,j\leq q}\Big)=\E f\Big(\Big\{\P\big(\bsigma^\star_{u_2}=i, \bsigma^\star_{v_2}=j\given \bG^\star(n,m)\big)\Big\}_{i,j\leq q}\Big)\,.
\eeq
Since $u_i\neq v_i$, note that there exists a permutation $\kappa: V\to V$ such that $\kappa(u_1)=u_2$ and $\kappa(v_1)=v_2$. With abuse of notation, denote $\kappa\circ \bsig^\star\equiv (\bsigma^\star_{\kappa(v)})_{v\in V}$ and for a factor graph $G$, let $\kappa\circ G$ be the factor graph defined by mapping $G$ by a graph isomorphism that permutes the labels of the variables according to $\kappa$. That is, $\kappa(v) \sim a$ in $\kappa\circ G$ if and only if $v\sim a$ in $G$. Then, the key observation is as follows. Recalling the definition of the planted model (cf. Definition~\ref{def:planted}) and our assumption that $p(\psi^{\theta})=p(\psi)$ holds for $\theta\in S_{\kappa}$, we have
\beqn
\big(\kappa\circ \bsig^\star, \kappa\circ \bG^\star(n,m)\big)\stackrel{d}{=} \big(\bsig^\star, \bG^\star(n,m)\big)\,.
\eeqn
Thus, $\P\big(\bsigma^\star_{u_1}=i, \bsigma^\star_{v_1}=j\bgiven \bG^\star(n,m)=G\big)=\P\big(\bsigma^\star_{u_2}=i, \bsigma^\star_{v_2}=j\bgiven \kappa\circ \bG^\star(n,m)=G\big)$ holds for any factor graph $G$ and $i,j\in [q]$. As a consequence, the LHS of \eqref{eq:lem:EXG:goal} equals
\beqn
\begin{split}
&\E f\Big(\Big\{\P\big(\bsigma^\star_{u_1}=i, \bsigma^\star_{v_1}=j\given \bG^\star(n,m)\big)\Big\}_{i,j\leq q}\Big)\\
&=\sum_{G} f\Big(\Big\{\P\big(\bsigma^\star_{u_1}=i, \bsigma^\star_{v_1}=j\given \bG^\star(n,m)=G\big)\Big\}_{i,j\leq q}\Big)\P\big(\bG^\star(n,m)=G\big)\\
&=\sum_{G} f\Big(\Big\{\P\big(\bsigma^\star_{u_2}=i, \bsigma^\star_{v_2}=j\given \kappa\circ \bG^\star(n,m)=G\big)\Big\}_{i,j\leq q}\Big)\P\big(\kappa\circ \bG^\star(n,m)= G\big)\\
&=\E f\Big(\Big\{\P\big(\bsigma^\star_{u_2}=i, \bsigma^\star_{v_2}=j\given \bG^\star(n,m)\big)\Big\}_{i,j\leq q}\Big)\,,
\end{split}
\eeqn
which concludes the proof of \eqref{eq:lem:EXG:goal}. 
\end{proof}
Under the condition ${\sf (EXG)}$, we then establish the analog of Theorem~\ref{thm:equiv}. 
\begin{prop}\label{prop:equiv}
    Assume that the sequence of random variables $(\bX^\star_n)_{n\geq 1}$ satisfies the condition ${\sf (EXG)}$. Then, the impossibility of weak recovery for $(\bX^\star_n)$ is equivalent to the (analog of) conditions in Theorem~\ref{thm:equiv}-\ref{item:b}, \ref{item:c}, \ref{item:d}, where one replaces $\bG^\star$ therein with $\bX^\star_n$.
\end{prop}

\begin{proof}[Proof of Theorem~\ref{thm:equiv}]
This is immediate from Lemma~\ref{lem:EXG:CONV} and Proposition~\ref{prop:equiv}.
\end{proof}

For the rest of this section, we prove Proposition~\ref{prop:equiv}. In Section \ref{subsec:twopoint}, we show that $(a)\Rightarrow(b) \Rightarrow(c)\Rightarrow (a)$ holds. In Section \ref{subsec:overlap}, we show that $(c)\Leftrightarrow (d)$ holds. Throughout, we abbreviate $\bX^\star\equiv \bX^\star_n$. We recall that $\langle \cdot \rangle_X$ denotes the expectation taken w.r.t. the posterior $\mu_X$.

\subsection{Triviality of two point correlation}
\label{subsec:twopoint}
We first substitute the quantity $A(\bsig^\star,\hat{\sig})$ in Definition~\ref{def:weak:recovery} by $\wt{A}(\bsig^\star,\hat{\sig})$, which will be more convenient for the proof of $(a)\Rightarrow (b)$ and $(c)\Rightarrow (a)$. For $\sig^{\ell}\equiv (\sigma^{\ell}_u)_{u\in V}\in [q]^V, \ell=1,2$, define
\beq\label{eq:def:wt:A}
\wt{A}(\sig^1,\sig^2)\equiv \wt{A}(\sig^1,\sig^2;\pi):=\max_{\Gamma\in S_q}\frac{1}{q}\sum_{i=1}^{q}\frac{1}{n\pi_i}\sum_{u\in V}\one\{\sigma^1_u=i, \sigma^2_u = \Gamma(i)\}\,.
\eeq
That is $\wt{A}(\sig^1,\sig^2)$ is defined by replacing $|\{v\in V: \sigma_v^1=i\}|$ term in \eqref{eq:def:overlap} by $n\pi_i$. If $\sig^1=\bsig^\star$, then such replacement is valid by the concentration of $|\{v\in V: \bsigma^\star_v=i\}|$ around $n\pi_i$ as shown in the lemma below.
\begin{lemma}\label{lem:equiv:weak:recovery}
There exists a constant $C\equiv C_{q,\pi}>0$ such that for any estimator $\hat{\sig}\equiv \hat{\sig}(\bX^\star)$, we have
\beqn
\big|\E[A(\bsig^\star, \hat{\sig})]-\E[\wt{A}(\bsig^\star, \hat{\sig})]\big|\leq Cn^{-1/3}\,.
\eeqn
\end{lemma}
\begin{proof}
Note that we can write $A(\bsig^\star, \hat{\sig})$ by
\beqn
A(\bsig^\star,\hat{\sig})=\max_{\Gamma \in S_q}\frac{1}{q}\sum_{i=1}^{q}\frac{\frac{1}{n\pi_i}\sum_{u\in V}\one\{\bsigma^\star_u=i, \hat{\sigma}_u=\Gamma(i)\}}{\frac{1}{n\pi_i}\sum_{u\in V}\one\{\bsigma^\star_u=i\}}\,.
\eeqn
We show that on a w.h.p. event, the denominator above is close to $1$: since $\bsig^\star\equiv (\bsigma^\star_v)_{v\in V}\iid \pi$, we have by Hoeffding's inequality and a union bound that 
\begin{equation}\label{eq:bal:event}
\P(\mathcal{A}_{\bal})\geq 1-q\exp(-2n^{1/3})\,,~~\textnormal{where}~~\mathcal{A}_{\bal}:=\bigg\{\bigg|\frac{1}{n}\sum_{u\in V}\one\big\{\bsigma^\star_u=i\big\}-\pi_i \bigg|\leq n^{-1/3},~\textnormal{for all}~i \in [q]\bigg\}\,.
\end{equation}
Thus, it follows that on the event $\AA_{\bal}$, for any estimator $\hat{\sig}$, 
\beqn
1-(\min_i \pi_i)^{-1}\cdot n^{-1/3}\leq \frac{A(\bsig^\star,\hat{\sig})}{\wt{A}(\bsig^\star,\hat{\sig})}\leq 1+(\min_i \pi_i)^{-1}\cdot n^{-1/3}\,.
\eeqn
Further, $\wt{A}(\sig^1,\sig^2)\in [0,(\min_i \pi_i)^{-1}]$ holds, so it follows that on the event $\AA_{\bal}$,
\beqn
\big|A(\bsig^\star, \hat{\sig})-\wt{A}(\bsig^\star,\hat{\sig})\big|\leq (\min_i \pi_i)^{-2}\cdot n^{-1/3}\,.
\eeqn
Combining with \eqref{eq:bal:event} concludes the proof.
\end{proof}
With Lemma \ref{lem:equiv:weak:recovery} in hand, we prove $(a)\Rightarrow (b)\Rightarrow (c)\Rightarrow (a)$. We first start with $(a)\Rightarrow (b)$.
\begin{lemma}\label{lem:two:point}
Suppose that weak recovery is impossible for $\bX^\star$ at $d$. Then, for any $u\neq v$ and $i,j\in [q]$, $\P(\bsigma^\star_u=i,\bsigma^\star_v=j\given \bX^\star)\pto \pi_i\pi_j$ holds as $n\to\infty$.
\end{lemma}
\begin{proof}
We first prove that the impossibility of weak recovery implies $\P(\bsigma^\star_u=i\given \bX^\star)\pto \pi_i$ for any $u\in V$ and $i\in [q]$. By {\sf (EXG)}, this implies that for every $u\in V$, $\P(\bsigma^\star_{u}=i_0\given \bX^\star)$  does not converge in probability to $\pi_{i_0}$. Now, note that $\max_{i\leq q}\frac{\P(\bsigmas_u=i\given \bX^\star)}{\pi_i}\geq 1$, almost surely, since $\sum_{i=1}^{q}\P(\bsigmas_u=i\given \bX^\star)=1$ holds. Hence, it follows that
\beq\label{eq:lem:two:point:crucial:1}
\limsup_{n\to\infty} \E\Bigg[\max_{i\in [q]}\bigg\{\frac{\P(\bsigmas_u=i\given \bX^\star)}{\pi_i}\bigg\}\Bigg]>1\,,
\eeq
since otherwise, $\max_{i\leq q}\frac{\P(\bsigmas_u=i)}{\pi_i}\pto 1$ holds, which implies that $\P(\bsigmas_u=i\given \bX^\star)\pto \pi_i$ for every $i\in [q]$. We now show that \eqref{eq:lem:two:point:crucial:1} implies that weak recovery is possible. To this end, for $X\in \XXX$, consider the estimator $\hat{\sig}(X)\equiv (\hat{\sigma}_u(X))_{u\in V}$ defined by
\beqn
\hat{\sigma}_u(X):=\argmax_{i\in [q]}\left\{\frac{\P(\bsigma^\star_u=i\given \bX^\star = X)}{\pi_i}\right\}\,.
\eeqn
Then, by taking $\Gamma= \id$ in the definition of $\wt{A}(\bsig^\star,\hat{\sig})$ (cf. \eqref{eq:def:wt:A}), and using ${\sf (EXG)}$ (cf. \eqref{eq:exg}), we have
\beq\label{eq:lem:two:poing:tower:1}
\E\Big[\wt{A}\big(\bsig^\star,\hat{\sig}(\bX^\star)\big)\Big]\geq \frac{1}{q}\sum_{i=1}^{q}\frac{1}{\pi_i}\P\big(\bsigmas_u=i, \hat{\sigma}_u(\bX^\star)=i\big)=\frac{1}{q}\sum_{i=1}^{q}\E\Bigg[\one(\hat{\sigma}_u(\bX^\star)=i)\cdot\frac{\P\big(\bsigmas_u=i \bgiven \bX^\star\big)}{\pi_i}\Bigg]\,,
\eeq
where the last equality is due to tower property and the fact that $\hat{\sigma}_u(\bX^\star)$ is $\bX^\star$-measurable. Note that the RHS equals $q^{-1}\E\big[\max_{i\leq q} \pi_i^{-1}\P(\bsigma^\star_u=i\given \bX^\star)\big]$. Thus, combining with \eqref{eq:lem:two:point:crucial:1}, we have
\beqn
\limsup_{n\to\infty}\E\Big[\wt{A}\big(\bsig^\star,\hat{\sig}(\bX^\star)\big)\Big]\geq \frac{1}{q}\limsup_{n\to\infty} \E\Bigg[\max_{i\in [q]}\bigg\{\frac{\P(\bsigmas_u=i\given \bX^\star)}{\pi_i}\bigg\}\Bigg]>\frac{1}{q}\,.
\eeqn
Therefore, by Lemma \ref{lem:equiv:weak:recovery}, weak recovery is possible at $d$, which contradicts our assumption. We thus conclude that $\P(\bsigma^\star_u=i\given \bX^\star)\pto \pi_i$ for any $u\in V$ and $i\in [q]$.

Next, we show that the impossibility of weak recovery implies $\P(\bsigma^\star_u=i, \bsigma^\star_v=j\given \bX^\star)\pto \pi_i\pi_j$ for any $u\neq v$ and $i,j\in [q]$ by using a similar argument as above. Assume by contradiction that there exists $u_0\neq v_0$ and $i_0, j_0\in [q]$ such that $\P(\bsigma^\star_{u_0}=i_0, \bsigma^\star_{v_0}=j_0\given \bX^\star)$ does not converge in probability to $\pi_{i_0}\pi_{j_0}$. Since we have shown that we must have $\P(\bsigma^\star_{v_0}=j_0\given \bX^\star)\pto \pi_{j_0}$, this implies that $\P(\bsigma^\star_{u_0}=i_0\given \bX^\star, \bsigma^\star_{v_0}=j_0)$ does not converge in probability to $\pi_{i_0}$. Then, by the same argument as above, the property ${\sf (EXG)}$ in \eqref{eq:exg} shows that for any two district variables $u\neq v$, 
\beq\label{eq:lem:two:point:crucial:2}
\limsup_{n\to\infty} \E\Bigg[\max_{i\in [q]}\bigg\{\frac{\P(\bsigmas_u=i\given \bX^\star, \bsigmas_v=j_0)}{\pi_i}\bigg\}\Bigg]>1\,,
\eeq
since otherwise, $\max_{i\in [q]}\frac{\P(\bsigmas_{u_0}=i\given \bX^\star, \bsigmas_{v_0}=j_0)}{\pi_i}\pto 1$ holds, which implies that $\P(\bsigma^\star_{u_0}=i\given \bX^\star, \bsigmas_{v_0}=j_0)\pto \pi_i$ holds for any $i\in [q]$. Now, for $X\in \XXX$, consider the estimator $\hat{\sig}^\prime(X)\equiv \big(\hat{\sigma}_u^\prime(X)\big)_{u\in V}$ defined by
\beqn
\hat{\sigma}_u^\prime(X):=\argmax_{i\in [q]}\left\{\frac{\P(\bsigma^\star_u=i\given \bX^\star = X, \bsigma^\star_v=j_0)}{\pi_i}\right\}\,.
\eeqn
Then, taking $\Gamma =\id$ in the definition of $\wt{A}\big(\bsig^\star, \hat{\sig}^\prime(\bX^\star)\big)$ and using a tower property as in \eqref{eq:lem:two:poing:tower:1}, we have
\beqn
\limsup_{n\to\infty}\E\Big[\wt{A}\big(\bsig^\star,\hat{\sig}^\prime(\bX^\star)\big)\Big]\geq \frac{1}{q}\limsup_{n\to\infty} \E\Bigg[\max_{i\in [q]}\bigg\{\frac{\P(\bsigmas_u=i\given \bX^\star, \bsigmas_v=j_0)}{\pi_i}\bigg\}\Bigg]>\frac{1}{q}\,,
\eeqn
where the last inequality is by \eqref{eq:lem:two:point:crucial:2}. Therefore, by Lemma \ref{lem:equiv:weak:recovery}, weak recovery is possible at $d$, which contradicts our assumption. We thus conclude that $\P(\bsigma^\star_u=i, \bsigma^\star_v=j\given \bX^\star)\pto \pi_i\pi_j$ holds for any $u\neq v$ and $i,j\in [q]$.
\end{proof}
The step $(b)\Rightarrow (c)$ is trivial, thus we next prove $(c)\Rightarrow (a)$
\begin{lemma}\label{lem:two:point:conv}
    Suppose that $\P(\bsigma^\star_u=\bsigma^\star_v \given \bX^\star)\pto \sum_{i=1}^{q}\pi_i^2$ and $\sum_{i=1}^{q}\pi_i\P(\bsigma^\star_u=i\given \bX^\star) \pto \sum_{i=1}^{q}\pi_i^2$ hold as $n\to\infty$. Then, weak recovery is impossible at $d$.
\end{lemma}
\begin{proof}
By Lemma \ref{lem:equiv:weak:recovery}, it suffices to show that $\E\big[\wt{A}(\bsig^\star,\hat{\sig}(\bX^\star)\big]= q^{-1}+o_n(1)$ holds for any estimator $\hat{\sig}(\bX^\star)$. For simplicity, denote $\hat{\sig}(\bX^\star)\equiv (\hat{\sigma}_u)_{u\in V}$. Then, note that we can express
\beqn
\E\big[\wt{A}(\bsig^\star,\hat{\sig}(\bX^\star)\big]-\frac{1}{q}=\E \max_{\Gamma\in S_q}\Bigg\{\frac{1}{q}\sum_{i=1}^{q}\frac{1}{n\pi_i}\sum_{u\in V}\bigg(\one\big(\bsigma^\star_u=i, \hat{\sigma}_u=\Gamma(i)\big)-\pi_i\one \big(\hat{\sigma}_u=\Gamma(i)\big)\bigg)\Bigg\}\,,
\eeqn
Thus, by triangle inequality, we can bound
\beq\label{eq:conv:tech:2}
\begin{split}
\left|\E\big[\wt{A}(\bsig^\star,\hat{\sig}(\bX^\star)\big]-\frac{1}{q}\right|
&\leq \big(q\min_{i\leq q} \pi_i\big)^{-1} \E \max_{\Gamma\in S_q}\Bigg\{\sum_{i=1}^{q}\frac{1}{n}\bigg|\sum_{u\in V}\big(\one(\bsigma^\star_u=i)-\pi_i\big)\cdot\one\big(\hat{\sigma}_u=\Gamma(i)\big)\bigg|\Bigg\}\\
&\leq \big(q\min_{i\leq q} \pi_i\big)^{-1} \frac{1}{n}\E \sum_{i,j=1}^{q}\bigg|\sum_{u\in V}\big(\one(\bsigma^\star_u=i)-\pi_i\big)\cdot\one\big(\hat{\sigma}_u=j\big)\bigg|\,.
\end{split}
\eeq
By Cauchy-Schwarz, we can bound the right most term by
\beq\label{eq:conv:tech}
\frac{1}{n}\sum_{i,j=1}^{q}\E\bigg|\sum_{u\in V}\big(\one(\bsigma^\star_u=i)-\pi_i\big)\cdot\one\big(\hat{\sigma}_u=j\big)\bigg|\leq \frac{q^2}{n}\Bigg(\sum_{i,j=1}^{q}\E\bigg(\sum_{u\in V}\big(\one(\bsigma^\star_u=i)-\pi_i\big)\cdot\one\big(\hat{\sigma}_u=j\big)\bigg)^2\Bigg)^{1/2}\,.
\eeq
Note that we can expand the sum inside the square root by
\beqn
\sum_{i,j=1}^{q}\E\bigg(\sum_{u\in V}\big(\one(\bsigma^\star_u=i)-\pi_i\big)\cdot\one\big(\hat{\sigma}_u=j\big)\bigg)^2=\sum_{i=1}^{q}\E\sum_{u,v\in V}\one(\hat{\sigma}_u=\hat{\sigma}_v)\big(\one(\bsigma^\star_u=i)-\pi_i\big)\big(\one(\bsigma^\star_v=i)-\pi_i\big)\,.
\eeqn
Since $\hat{\sig}(\bX^\star)\equiv (\hat{\sigma}_u)_{u\in V}$ is $\bX^\star$-measurable, tower property shows that the RHS above equals
\beqn
\begin{split}
&\E\Bigg[\sum_{u,v\in V}\one(\hat{\sigma}_u=\hat{\sigma}_v)\E\bigg[\sum_{i=1}^{q}\big(\one(\bsigma^\star_u=i)-\pi_i\big)\big(\one(\bsigma^\star_v=i)-\pi_i\big)\bbgiven\bX^\star\bigg]\Bigg]\\
&=\E\sum_{u,v\in V}\one(\hat{\sigma}_u=\hat{\sigma}_v)\Big(\P(\bsigma^\star_u=\bsigma^\star_v\given \bX^\star)-\sum_{i=1}^{q}\pi_i\P(\bsigma^\star_u=i\given\bX^\star)-\sum_{i=1}^{q}\pi_i\P(\bsigma^\star_v=i\given\bX^\star)+\sum_{i=1}^{q}\pi_i^2\Big)\,.
\end{split}
\eeqn
By dividing the sum over $u,v\in V$ into the cases $u=v$ and $u\neq v$, and using triangle inequality, we can bound
\beqn
\begin{split}
&\sum_{i,j=1}^{q}\E\bigg(\sum_{u\in V}\big(\one(\bsigma^\star_u=i)-\pi_i\big)\cdot\one\big(\hat{\sigma}_u=j\big)\bigg)^2\\
&\leq 2n+n(n-1)\cdot\E\Big|\P(\bsigma^\star_{u_1}=\bsigma^\star_{v_1}\given \bX^\star)-2\sum_{i=1}^{q}\pi_i\P(\bsigma^\star_{u_1}=i\given\bX^\star)+\sum_{i=1}^{q}\pi_i^2\Big|\,,
\end{split}
\eeqn
where we fixed two distinct variables $u_1\neq v_1$. Plugging the inequality above into \eqref{eq:conv:tech}, we have 
\beq\label{eq:conv:tech:3}
\begin{split}
&\frac{1}{n}\sum_{i,j=1}^{q}\E\bigg|\sum_{u\in V}\big(\one(\bsigma^\star_u=i)-\pi_i\big)\cdot\one\big(\hat{\sigma}_u=j\big)\bigg|\\
&\leq q^2\cdot\bigg(\E\Big|\P(\bsigma^\star_{u_1}=\bsigma^\star_{v_1}\given \bX^\star)-2\sum_{i=1}^{q}\pi_i\P(\bsigma^\star_{u_1}=i\given\bX^\star)+\sum_{i=1}^{q}\pi_i^2\Big|+\frac{2}{n}\bigg)^{1/2}\,.
\end{split}
\eeq
Since $\P\big(\bsigma^\star_{u_1}=\bsigma^\star_{v_1}\bgiven \bX^\star\big)\pto \sum_{i=1}^{q}\pi_i^2$ and $\sum_{i=1}^{q}\pi_i \P\big(\bsigma^\star_{u_1}=i\bgiven \bX^\star\big)\pto \sum_{i=1}^{q}\pi_i^2$ hold by our assumption, the RHS of \eqref{eq:conv:tech:3} converges to $0$ as $n\to\infty$ since the convergent random variables are bounded. This further implies that the RHS of \eqref{eq:conv:tech:2} converges to $0$. We thus conclude that $\E\big[\wt{A}(\bsig^\star,\hat{\sig}(\bX^\star)\big]= q^{-1}+o_n(1)$ holds, which implies that the weak recovery is impossible at $d$ by Lemma \ref{lem:equiv:weak:recovery}. 
\end{proof}
\subsection{Triviality of the overlap}
\label{subsec:overlap}
Next, we establish $(c)\Leftrightarrow (d)$. The proof of $(c)\Rightarrow (d)$ resembles the proof of Lemma \ref{lem:two:point:conv}.
\begin{lemma}\label{lem:overlap:posterior}
 Suppose that $\P(\bsigma^\star_u=\bsigma^\star_v \given \bX^\star)\pto \sum_{i=1}^{q}\pi_i^2$ and $\sum_{i=1}^{q}\pi_i\P(\bsigma^\star_u=i\given \bX^\star) \pto \sum_{i=1}^{q}\pi_i^2$ hold as $n\to\infty$. Then, $\E\big\langle \big\|R_{\sig^1,\sig^2}-\pi\pi^{\sT}\big\|_1\big\rangle_{\bX^\star} \to 0$ holds as $n\to\infty$.
\end{lemma}
\begin{proof}
By Cauchy-Schwarz inequality, we can bound
\beq\label{eq:overlap:posterior:CS}
\E\Big\langle \big\|R_{\sig^1,\sig^2}-\pi\pi^{\sT}\big\|_1\Big\rangle_{\bX^\star}\leq q\cdot\bigg(\sum_{i,j=1}^{q}\E\Big\langle \big(R_{\sig^1,\sig^2}(i,j)-\pi_i\pi_j\big)^2\Big\rangle_{\bX^\star}\bigg)^{1/2}\,.
\eeq
Recalling the definition of $R_{\sig^1,\sig^2}(i,j)$ in \eqref{eq:R:def}, we can expand
\beqn
\begin{split}
\sum_{i,j=1}^{q}\E\Big\langle \big(R_{\sig^1,\sig^2}(i,j)-\pi_i\pi_j\big)^2\Big\rangle_{\bX^\star}=
&\frac{1}{n^2}
\sum_{u,v\in V}\sum_{i,j=1}^{q}\E\Big\langle\one\big(\sigma^{1}_u=i, \sigma^1_v=j\,,\,\sigma^{2}_u=i, \sigma^{2}_v=j \big)\Big\rangle_{\bX^\star}\\
&~-\frac{2}{n}\sum_{u\in V}\sum_{i,j=1}^{q}\pi_i\pi_j\E\Big\langle\one\big(\sigma^{1}_u=i\,,\,\sigma^{2}_u=j \big)\Big\rangle_{\bX^\star}+\sum_{i,j=1}^{q}\pi_i^2\pi_j^2\,.
\end{split}
\eeqn
Recalling that $\sig^1, \sig^2$ are independent samples from the posterior $\mu_{\bX^\star}$, we can simplify the RHS above by using the property ${\sf (EXG)}$ in \eqref{eq:exg} as
\beqn
\frac{1}{n}+\frac{n-1}{n}\cdot \E\Big(\P\big(\bsigma^\star_u=\bsigma^\star_v\bgiven \bX^\star\big)\Big)^2-2\cdot\E\Big(\sum_{i=1}^{q}\pi_i \P\big(\bsigma^\star_u=i\bgiven \bX^\star\big)\Big)^2+\Big(\sum_{i=1}^{q}\pi_i^2\Big)^2\,,
\eeqn 
where $u\neq v$ are distinct variables. By our assumption, $\P(\bsigma^\star_u=\bsigma^\star_v\given \bX^\star)$ and $\sum_{i=1}^{q}\pi_i\P(\bsigma^\star_u=i\given \bX^\star)$ converge in probability to $\sum_{i=1}^{q}\pi_i^2$, which further implies $L^2$ convergence since they are bounded. Therefore, we have
$\sum_{i,j=1}^{q}\E\big\langle \big(R_{\sig^1,\sig^2}(i,j)-\pi_i\pi_j\big)^2\big\rangle_{\bX^\star}\to 0$, and combining with \eqref{eq:overlap:posterior:CS} concludes the proof.
\end{proof}
The implication $(d)\Rightarrow (c)$ is a consequence of the following lemma, which is a generalization of \cite[Lemma 4.8]{CKPZ:18} to non-uniform priors and is built on \cite[Corollary 2.2]{Bapst16harness}.
\begin{lemma}\label{lem:pinning}
For all $\eps>0$ and $q\geq 2$, there exists $\delta\equiv \delta(\eps,q)>0$ and $n_0\equiv n_0(\eps,q)$ such that if $n\geq n_0(\eps,q)$, the following holds. Suppose that a probability measure $\mu\in \PPP([q]^n)$ and a probability vector on $[q]$, $\pi \in \R^{q}, \sum_{i=1}^{q}\pi_i=1, \pi_i\geq 0$, satisfy 
\beq\label{eq:lem:pinning}
\Big\langle \big\|R_{\sig^1,\sig^2}-\pi\pi^{\sT}\big\|_{\Fr}^2\Big\rangle_{\mu}<\delta\,,
\eeq
where $\big\langle \cdot \big \rangle_{\mu}$ is the expectation with respect to independent samples $\sig^1,\sig^2\sim \mu$, $R_{\sig^1,\sig^2}$ is defined in \eqref{eq:R:def}, and $\|\,\cdot\,\|_{\Fr}$ denotes the Frobenius norm. Then, we have
\beqn
\frac{1}{n^2}\sum_{u,v\in [n]} \big\| \mu_{u,v}-\pi^{\otimes 2}\big\|_{\tv}<\eps\,,
\eeqn
where $\mu_{u,v}$ denotes the marginal distribution of $(\sigma_u,\sigma_v)$ when $\sig\equiv (\sig_u)_{u\in [n]}\sim \mu$.
\end{lemma}
\begin{proof}
The following is a restatement of \cite[Corollary 2.2]{Bapst16harness}: for all $\eps_0>0$, there exists $\eta \equiv \eta(\eps_0,q)$ and $n_0=n_0(\eps_0,q)$ such that for every $n\geq n_0$ and $\mu \in \PPP([q]^n)$, the following holds. There exists a decomposition $S_0,S_1,\ldots, S_N$ of $[q]^n$ such that it satisfies
\begin{itemize}
    \item {\sf (C1):} $\mu(S_0)\leq \eps_0$ and for any $1\leq k \leq N$, $\mu(S_k)\geq \eta$.
    \item {\sf (C2):} For $1\leq k\leq N$ and $u,v\in [n]$, denote by $\mu_{u,v}(\cdot \given S_k)$ the marginal distribution of $(\sigma_u,\sigma_v)$, where $\sig=(\sigma_u)_{u\in V}$ is drawn from the conditional distribution $\mu(\cdot \given S_k)\equiv \mu(\cdot \cap S_k)/\mu(S_k)$. Similarly, denote by $\mu_u(\cdot\given S_k)$ the marginal distribution of $\sigma_u$, where $\sig \sim \mu(\cdot \given S_k)$. Then, for any $1\leq k \leq N$,
    \beqn
    \frac{1}{n^2}\sum_{u,v\in [n]}\big\| \mu_{u,v}(\cdot \given S_k)-\mu_u(\cdot \given S_k)\otimes \mu_v(\cdot \given S_k)\big\|_{\tv}<\eps_0\,.
    \eeqn
\end{itemize}
We proceed by taking advantage of such decomposition guaranteed by \cite[Corollary 2.2]{Bapst16harness}. Note that it suffices to prove the lemma for small enough $\eps>0$, so we fix $\eps \in (0,1/100q)$. Given such $\eps>0$, we take $\eps_0$ and $\delta$ small enough depending only on $\eps$ and $q$ whose precise value is determined below. Suppose that for a probability measure $\mu\in \PPP([q]^n)$ satisfies \eqref{eq:lem:pinning}. Given $\mu$, consider the decomposition $S_0,S_1,\ldots, S_N$ of $[q]^n$, which satisfies ${\sf (C1), (C2)}$ above, and denote by $\langle\,\cdot\,\rangle_k$ the average w.r.t. $\sig^1,\sig^2\iid \mu(\cdot\given S_k)$. Then, by our assumption \eqref{eq:lem:pinning} and since $\mu(S_k)\geq \eta\equiv \eta(\eps_0,q)$ holds by {\sf (C1)}, we have for any $1\leq k \leq N$,
\beq\label{eq:RS:k}
\Big\langle \big\|R_{\sig^1,\sig^2}-\pi\pi^{\sT}\big\|_{\Fr}^2\Big\rangle_{k}<\frac{\delta}{\eta^2}\,.
\eeq
Recalling the definition of $R_{\sig^1,\sig^2}$ in \eqref{eq:R:def}, we can expand the LHS by
\beq\label{eq:RS:expand}
\Big\langle \big\|R_{\sig^1,\sig^2}-\pi\pi^{\sT}\big\|_{\Fr}^2\Big\rangle_{k}=\sum_{i,j=1}^{q}\Bigg(\frac{1}{n^2}\sum_{u,v\in [n]}\mu_{u,v}(i,i\given S_k)\mu_{u,v}(j,j\given S_k)-\frac{2\pi_i\pi_j}{n}\sum_{u\in [n]}\mu_u(i\given S_k)\mu_u(j\given S_k)+\pi_i^2\pi_j^2\Bigg)\,.
\eeq
Note that by {\sf (C2)}, we can replace $\mu_{u,v}(\cdot\given S_k)$ with $\mu_u(\cdot \given S_k)\mu_v(\cdot \given S_k)$ with error depending on $\eps_0$:
\beqn
\frac{1}{n^2}\sum_{i,j=1}^{q}\sum_{u,v\in [n]}\mu_{u,v}(i,i\given S_k)\mu_{u,v}(j,j\given S_k) \geq \frac{1}{n^2}\sum_{i,j=1}^{q}\sum_{u,v\in [n]}\mu_{u}(i\given S_k)\mu_v(i\given S_k)\mu_{u}(j\given S_k)\mu_v(j\given S_k)-4\eps_0\,,
\eeqn
thus plugging the above bound into \eqref{eq:RS:expand} and combining with \eqref{eq:RS:k} gives
\beqn
\sum_{i,j=1}^{q}\bigg(\frac{1}{n}\sum_{u\in [n]}\mu_u(i\given S_k)\mu_v(j\given S_k)-\pi_i\pi_j\bigg)^2<\frac{\delta}{\eta^2}+4\eps_0\,.
\eeqn
Thus, by Cauchy-Schwarz, we have that
\beq\label{eq:pure:decomp}
\sum_{i,j=1}^{q}\bigg|\frac{1}{n}\sum_{u\in [n]}\mu_u(i\given S_k)\mu_v(j\given S_k)-\pi_i\pi_j\bigg|\leq q^2\cdot\bigg(\frac{\delta}{\eta^2}+4\eps_0\bigg)^{1/2}\leq \eps^3\,,
\eeq
where in the last inequality we took $\eps_0\equiv \frac{\eps^6}{8q^4}$ and $\delta\equiv \frac{\eta^2 \eps^6}{2q^4}$ for $\eta\equiv \eta(\eps_0,q)$. Observe that restricting the sum over $i,j$ in \eqref{eq:pure:decomp} to $i=j$, and pushing the sum over $j$ inside the absolute value by triangle inequality shows that
\beqn
\bigg(\sum_{i=1}^{q}\bigg|\frac{1}{n}\sum_{u\in [n]}\mu_u(i\given S_k)^2-\pi_i^2\bigg|\bigg)\vee \bigg(\sum_{i=1}^{q}\bigg|\frac{1}{n}\sum_{u\in [n]}\mu_u(i\given S_k)-\pi_i\bigg|\bigg)\leq \eps^3\,.
\eeqn
Thus, by triangle inequality,
\beqn
\sum_{i=1}^{q}\frac{1}{n}\sum_{u\in [n]}\big(\mu_u(i\given S_k)-\pi_i\big)^2=\sum_{i=1}^{q}\bigg(\frac{1}{n}\sum_{u\in [n]}\mu_u(i\given S_k)^2-\pi_i^2\bigg)-2\sum_{i=1}^{q}\pi_i\bigg(\frac{1}{n}\sum_{u\in [n]}\mu_u(i\given S_k)-\pi_i\bigg)\leq 3 \eps^3\,.
\eeqn
Then, it follows from Cauchy-Schwarz that
\beq\label{eq:lem:pinning:goal}
\frac{1}{n}\sum_{i=1}^{q}\sum_{u\in [n]}\Big|\mu_u(i\given S_k)-\pi_i\Big|\leq (3q\eps^3)^{1/2}<\frac{\eps}{4}\,,
\eeq
where the last inequality holds since $\eps <(100q)^{-1}$. Consequently, combining with {\sf (C2)}, we have
\beq\label{eq:lem:pinning:goal:2}
\begin{split}
\frac{1}{n^2}\sum_{u,v\in [n]}\big\|\mu_{u,v}(\cdot \given S_k)-\pi\pi^{\sT}\big\|_{\tv}
&\leq \eps_0+\frac{1}{n^2}\sum_{u,v\in [n]}\big\|\mu_u(\cdot \given S_k)\otimes \mu_v(\cdot \given S_k)-\pi\pi^{\sT}\big\|_{\tv}\\
&\leq \eps_0+\frac{2}{n}\sum_{u\in [n]}\big\|\mu_u(\cdot \given S_k)-\pi\big\|_{\tv}<\frac{\eps}{2}\,,
\end{split}
\eeq
where the last inequality is due to \eqref{eq:lem:pinning:goal}. Therefore, by triangle inequality, we have that
\beqn
\begin{split}
\frac{1}{n^2}\sum_{u,v\in[n]}\big\| \mu_{u,v}-\pi^{\otimes 2}\big\|_{\tv}
&\leq \frac{1}{n^2}\sum_{u,v\in [n]}\sum_{k=0}^{N}\big\|\mu_{u,v}(\cdot\given S_k)-\pi^{\otimes 2}\big\|_{\tv}\cdot \mu(S_k)\\
&\stackrel{(a)}{\leq} \mu(S_0)+\frac{\eps}{2}\sum_{k=1}^{N}\mu(S_k)\stackrel{(b)}{<}\eps_0+\frac{\eps}{2}<\eps\,,
\end{split}
\eeqn
where $(a)$ is due to \eqref{eq:lem:pinning:goal:2} and $(b)$ is due to {\sf (C1)}. This concludes the proof.
\end{proof}
With Lemma \ref{lem:pinning} in hand, the step $(d)\Rightarrow (c)$ is straightforward.
\begin{lemma}\label{lem:posterior:triviality:to:two:point}
Suppose that $\E\big\langle \big\|R_{\sig^1,\sig^2}-\pi\pi^{\sT}\big\|_1\big\rangle_{\bX^\star} \to 0$ as $n\to\infty$. Then, $\P(\bsigma^\star_u=\bsigma^\star_v \given \bX^\star)\pto \sum_{i=1}^{q}\pi_i^2$ and $\sum_{i=1}^{q}\pi_i\P(\bsigma^\star_u=i\given \bX^\star) \pto \sum_{i=1}^{q}\pi_i^2$ hold as $n\to\infty$.
\end{lemma}
\begin{proof}
By Markov's inequality, $\big\|R_{\sig^1,\sig^2}(i,j)-\pi_i\pi_j\big\|\pto 0$ under $\E \mu_{\bX^\star}^{\otimes 2}$ for any $i,j\in [q]$. Since $\big\|R_{\sig^1,\sig^2}(i,j)-\pi_i\pi_j\big\|\leq 1$, convergence in probability further implies $L^2$ convergence. Summing over $i,j\in [q]$, we have $\E\big\langle\big\|R_{\sig^1,\sig^2}-\pi\pi^{\sT}\big\|_{\Fr}^2\big\rangle_{\bX^\star}\to 0$. Hence, Lemma \ref{lem:pinning} further implies that 
\beqn
\frac{1}{n^2}\sum_{u,v\in V}\Big\|\P\big(\bsigma^\star_u=\cdot\,,\, \bsigma^\star_v=\cdot\bgiven \bX^\star\big)-\pi^{\otimes 2}\Big\|_{\tv}\pto 0\,.
\eeqn
Note that the LHS above is at most $1$, thus above further implies $L^1$ convergence. Moreover,
\beq\label{eq:conseq:exg}
\begin{split}
\frac{1}{n^2}\sum_{u,v\in V}\E\Big\|\P\big(\bsigma^\star_u=\cdot\,,\, \bsigma^\star_v=\cdot\bgiven \bX^\star\big)-\pi^{\otimes 2}\Big\|_{\tv}=\sum_{i,j=1}^{q}\E\Big|\P\big(\bsigma^\star_{u_1}=i,\bsigma^\star_{v_1}=j\bgiven \bX^\star\big)-\pi_i\pi_j\Big|+O(n^{-1})\,,
\end{split}
\eeq
where we used the property ${\sf (EXG)}$ in the equality above for distinct variables $u_1\neq v_1$. Therefore, we have that $\P(\bsigma^\star_u=\bsigma^\star_v \given \bX^\star)\pto \sum_{i=1}^{q}\pi_i^2$ and $\sum_{i=1}^{q}\pi_i\P(\bsigma^\star_u=i\given \bX^\star) \pto \sum_{i=1}^{q}\pi_i^2$.
\end{proof}
Finally, we prove a non-asymptotic version of the implication $(d)\Rightarrow (a)$, which is used for the proof of Proposition~\ref{prop:free:energy} (see Lemma~\ref{lem:psi:R}).
\begin{lemma}\label{lem:nonasymptotic}
    For any $\eps>0$, there exists $\eta\equiv \eta(\eps,q,\pi)>0$ and $n_0\equiv n_0(\eps,q,\pi)$ such that the following holds. Suppose that $n\geq n_0$ and there exists an estimator $\hat{\sig}\equiv \hat{\sig}(\bX^\star)$ such that $\E\big[A(\bsig^\star, \hat{\sig})\big]\geq \frac{1}{q}+\eps$ holds. Then, we have $\E\big\langle \big\|R_{\sig^1,\sig^2}-\pi\pi^{\sT}\big\|_1\big\rangle_{\bX^\star}\geq \eta$. 
\end{lemma}
\begin{proof}
   Throughout, we write $C\equiv C(q,\pi)>0$ for a positive constant that only depends on $q,\pi$ that may differ from line to line. Also, we fix $\eps>0$ and write $n_0\equiv n_0(\eps,q,\pi)$ for an integer that only depends on $\eps,q,\pi$. Suppose that $\E\big[A(\bsig^\star, \hat{\sig})\big]\geq \frac{1}{q}+\eps$ holds for some $\hat{\sig}\equiv \hat{\sig}(\bX^\star)$. By Lemma~\ref{lem:equiv:weak:recovery}, we have $\E\big[\wt{A}(\bsig^\star, \hat{\sig})\big]\geq \frac{1}{q}+\frac{\eps}{2}$ for $n\geq n_0$. Subsequently, by the inequalities \eqref{eq:conv:tech:2} and \eqref{eq:conv:tech:3} in the proof of Lemma~\ref{lem:two:point:conv}, it follows that for $n\geq n_0$
   \beqn
C\eps^2\leq \E\Big|\P(\bsigma^\star_{u_1}=\bsigma^\star_{v_1}\given \bX^\star)-2\sum_{i=1}^{q}\pi_i\P(\bsigma^\star_{u_1}=i\given\bX^\star)+\sum_{i=1}^{q}\pi_i^2\Big|\,,
   \eeqn
   where $u_1\neq v_1$ are two distinct variables. By triangle inequality, the RHS is upper bounded by
   \beqn
   \begin{split}
&\E\Big|\sum_{i=1}^{q}\big(\P(\bsigma^\star_{u_1}=i, \bsigma^\star_{v_1}=i\given \bX^\star)-\pi_i^2\big)-2 \sum_{i=1}^{q}\pi_i\big(\P(\bsigma^\star_{u_1}=i\given\bX^\star)-\pi_i\big)\Big|\\
&\leq 3\E\big\|\P\big(\bsigma^\star_{u_1}=\cdot\,,\, \bsigma^\star_{v_1}=\cdot\bgiven \bX^\star\big)-\pi^{\otimes 2}\big\|_{\tv}\,.
   \end{split}
   \eeqn
   Thus, combining the two inequalities above, $\E\big\|\P\big(\bsigma^\star_{u_1}=\cdot\,,\, \bsigma^\star_{v_1}=\cdot\bgiven \bX^\star\big)-\pi^{\otimes 2}\big\|_{\tv}\geq C\eps^2$ holds for $n\geq n_0$. Recalling the equality \eqref{eq:conseq:exg} and using that the total variation distance is at most $1$, we have 
   \beqn
  \P\bigg(\frac{1}{n^2}\sum_{u,v\in V}\big\|\P\big(\bsigma^\star_u=\cdot\,,\, \bsigma^\star_v=\cdot\bgiven \bX^\star\big)-\pi^{\otimes 2}\big\|_{\tv}\geq C\eps^2\bigg)\geq C\eps^2\,.
   \eeqn
   Then by Lemma~\ref{lem:pinning}, there exists a constant $\delta\equiv \delta(\eps,q,\pi)>0$ such that 
   \beqn
  \P\bigg(\Big\langle\big\|R_{\sig^1,\sig^2}-\pi\pi^{\sT}\big\|_{\Fr}^2\Big\rangle_{\bX^\star}\geq \delta\bigg)\geq C\eps^2\,.
   \eeqn
   Finally, note that $\|R_{\sig^1,\sig^2}-\pi\pi^{\sT}\big\|_{\Fr}^2\leq \|R_{\sig^1,\sig^2}-\pi\pi^{\sT}\big\|_{1}$ holds since the entries of $R_{\sig^1,\sig^2}-\pi\pi^{\sT}$ have absolute values at most $1$. Therefore, setting $\eta\equiv C\delta \eps^2$ concludes the proof.
\end{proof}
\begin{proof}[Proof of Proposition~\ref{prop:equiv}]
   This is immediate from combining Lemmas \ref{lem:two:point}, \ref{lem:two:point:conv},  \ref{lem:overlap:posterior}, and \ref{lem:posterior:triviality:to:two:point}.
\end{proof}
\section{Contiguity below the weak recovery threshold}
\label{sec:proof:contiguity}
In this section, we prove Lemma~\ref{lem:KS}, Proposition \ref{prop:sec:moment:LR} and Lemma \ref{lem:lambda:equality}, which were crucial for the proof of Theorem~\ref{thm:factor:contiguity}. In Section~\ref{subsec:prop:sec:moment:LR}, we prove Proposition~\ref{prop:sec:moment:LR} by appealing to a central limit theorem. In Section~\ref{subsec:lem:lambda:equality}, we prove Lemmas \ref{lem:KS} and \ref{lem:lambda:equality}.
\subsection{Proof of Lemma~\ref{lem:KS} and Lemma~\ref{lem:lambda:equality}}
\label{subsec:lem:lambda:equality}
Throughout, we assume ${\sf (SYM)}$. We start with the first statement of Lemma~\ref{lem:KS}.
\begin{lemma}\label{lem:Xi:self}
    The linear operator $\Xi$ is self-adjoint on the inner product space $(\R^{q}\otimes \R^{q},\langle \cdot,\cdot\rangle_{\pi})$.
\end{lemma}
\begin{proof}
    By linearity, it suffices to check $\big\langle \Xi (e_i\ot e_j), e_s\ot e_t\big\rangle_{\pi}=\big\langle e_i\ot e_j, \Xi(e_s\ot e_t)\big\rangle_{\pi}$ for $i,j,s,t\in[q]$, where $(e_i)_{i\in [q]}$ are standard basis of $\R^q$. Note that by definition of $\Xi$ and $\langle \cdot, \cdot\rangle_{\pi}$, we have
    \beqn
    \begin{split}
    \big\langle \Xi (e_i\ot e_j), e_s\ot e_t\big\rangle_{\pi}
    =\sum_{\psi\in \Psi}p(\psi)\cdot \big\langle \Phi_{\psi} e_i\ot \Phi_{\psi} e_j, e_s\ot e_t \big\rangle_{\pi}=\E_{p}\Big[ e_i^{\sT}\Phi_{\bpsi}^{\sT}\diag(\pi) e_s  \cdot e_j^{\sT}\Phi_{\bpsi}^{\sT}\diag(\pi) e_t\Big]\,.
    \end{split}
    \eeqn
    Recalling the condition \eqref{eq:psi:symm} and the definition of $\Phi_{\psi}$ in \eqref{def:Phi}, note that we have 
    \beq\label{eq:diag:Phi:symmetric}
    \diag(\pi)\Phi_{\bpsi}\stackrel{d}{=}\Phi_{\bpsi}^{\sT}\diag(\pi)\quad\textnormal{for}\quad\bpsi\sim p\,.
    \eeq 
    Thus, it follows that $\big\langle \Xi (e_i\ot e_j), e_s\ot e_t\big\rangle_{\pi}=\big\langle e_i\ot e_j, \Xi(e_s\ot e_t)\big\rangle_{\pi}$ holds for $i,j,s,t\in[q]$.
\end{proof}
Consider the matrix $\Phi\equiv \Phi_p\in \R^{q\times q}$ defined by the average of $\Phi_{\bpsi},\bpsi\sim p$.
\beqn
\Phi\equiv \Phi_p:=\E_{p}\big[\Phi_{\bpsi}\big]\,.
\eeqn
As before, we often identify the matrix $\Phi$ with the linear map $x\to \Phi x$. Further, consider the inner product $\langle \cdot, \cdot \rangle_{\pi,1}$ on $\R^{q}$ defined by $\langle x,y\rangle_{\pi,1}:=x^{\sT}\diag(\pi) y$. The notation $1$ is to distinguish $\langle \cdot, \cdot \rangle_{\pi,1}$ from the inner product $\langle \cdot,\cdot\rangle_{\pi}$ in \eqref{def:inner:product}, which is the tensor product of $\langle \cdot, \cdot \rangle_{\pi,1}$ with itself. Then, we have the following lemma.
\begin{lemma}\label{lem:Phi:self}
    The linear operator $\Phi$ is self-adjoint on the inner product space $(\R^{q}, \langle \cdot, \cdot \rangle_{\pi,1})$ such that $\Phi\bone=\bone$. Further, for $w\in \R^{q}$, we have that $\Xi (\bone \otimes w)= \bone \otimes \Phi w$.
\end{lemma}
\begin{proof}
    Note that since $\diag(\pi)\Phi_{\bpsi}\stackrel{d}{=}\Phi_{\bpsi}^{\sT}\diag(\pi)$ holds for $\bpsi \sim p$ (cf.~\eqref{eq:diag:Phi:symmetric}), $\diag(\pi)\Phi=\Phi\diag(\pi)$ holds. Thus, for any $x,y\in \R^{q}$, we have that
    \beqn
    \big\langle \Phi x , y\big\rangle_{\pi,1}= x^{\sT}\Phi^{\sT} \diag(\pi) y =x^{\sT}\diag(\pi)\Phi y=\big\langle x , \Phi  y\big\rangle_{\pi,1}\,.
    \eeqn
    Further, note that by ${\sf (SYM)}$, for any $\psi \in \Psi$, we have the equality $\Phi_{\psi} \bone= \bone$. Thus, $\Phi \bone=\bone$ holds and we have for any $w\in \R^{q}$ that
    \beqn
    \Xi(\bone \otimes w)=\sum_{\psi\in \Psi}p(\psi)\cdot \Phi_{\psi} \bone \ot \Phi_{\psi}w = \sum_{\psi\in \Psi}p(\psi)\cdot  \bone \ot \Phi_{\psi} w =\bone \ot \Phi w\,,
    \eeqn
    which concludes the proof.
\end{proof}
Observe that by Lemma~\ref{lem:Phi:self}, there exists an orthonormal basis $\{w_1=1, w_2,\ldots, w_{q}\}$ of the inner product space $(\R^{q}, \langle \cdot, \cdot \rangle_{\pi,1})$ such that $w_i$'s are eigenvectors of $\Phi$. We denote by $\lambda_i$ the eigenvalue corresponding to $w_i$, where $\lambda_1=1$. Then, we have the following lemma. Below, $\proj_{\SS^{\perp}}$ and $\proj_{w_i\ot w_j}$ for $i,j\in [q]$ respectively denotes the projection operator in $(\R^{q}\otimes \R^{q},\langle \cdot,\cdot \rangle_{\pi})$ onto the subspace $\SS^{\perp}$ and the subspace spanned by the vector $w_i\ot w_j$.
\begin{lemma}\label{lem:SS:perp}
    The subspaces $\SS$ and $\SS^{\perp}$ is invariant under $\Xi$. Further, we have that
    \beqn
\Xi\circ \proj_{\SS^{\perp}}=\proj_{\bone\ot \bone}+\sum_{i=2}^{q}\lambda_i \proj_{\bone\ot w_i}+\sum_{i=2}^{q}\lambda_i \proj_{w_i \ot \bone}\,.
\eeqn
Thus, $\Xi\rvert_{S^{\perp}}$ has eigenvalues $\{1,\lambda_2,\lambda_2,\ldots, \lambda_q,\lambda_q\}$ counting multiplicities.
\end{lemma}
\begin{proof}
Since $(w_i)_{i\leq q}$ are eigenvectors of $\Phi$, Lemma~\ref{lem:Phi:self} implies that $1\otimes w_i$ and $w_i \otimes 1$ are eignevectors of $\Xi$ with $\lambda_i$ the corresponding eigenvalue. Meanwhile, since $\R^{q}\ot \R^{q}$ is finite dimensional, $\SS^{\perp}$ is given by
\beqn
\SS^{\perp}=\{x\in \R^{q}\ot \R^{q}: x=w\ot \bone~\textnormal{or}~x=\bone\ot w~\textnormal{for some}~w\in \R^{q}\}\,.
\eeqn
Thus, the set $\WW:=\{\bone\ot \bone, \bone\ot w_2,\ldots, \bone\ot w_q,w_2\ot \bone,\ldots, w_q\ot \bone\}$ is an orthonormal basis of the inner product space $(\R^{q}\otimes \R^{q},\langle \cdot,\cdot \rangle_{\pi})$, which are also eigenvectors of $\Xi$. Hence, the subspace of $\SS^{\perp}$ is invariant under $\Xi$, which further implies that $\SS$ is invariant since $\Xi$ is self-adjoint by Lemma~\ref{lem:Xi:self}. Moreover, since $\WW$ is the eigenbasis of $\Xi\rvert_{\SS^{\perp}}$, the rest of the claims follow by the spectral theorem.
\end{proof}
Having Lemmas~\ref{lem:Xi:self}, \ref{lem:Phi:self}, and \ref{lem:SS:perp} in hand, we now prove Lemma~\ref{lem:KS} and Lemma~\ref{lem:lambda:equality}.
\begin{proof}[Proof of Lemma~\ref{lem:KS}]
The first and second statements follow immediately from Lemma~\ref{lem:Xi:self} and Lemma~\ref{lem:SS:perp}, thus it remains to prove the third statement. Note that applying the spectral theorem on the self-adjoint operator $\Phi$ on $(\R^{q}, \langle \cdot, \cdot \rangle_{\pi,1})$ (cf. Lemma~\ref{lem:Phi:self}), we have
\beqn
\Phi= \sum_{i=1}^{q}\lambda_i \proj_{w_i}=\bone \pi^{\sT}+\sum_{i=2}^{q}\la_i \proj_{w_i}\,,
\eeqn
where $\proj_{w_i}$ denotes the projection operator in $(\R^{q}, \langle \cdot, \cdot \rangle_{\pi,1})$ onto the subspace spanned by the vector $w_i$. In particular, the last equality follows since $\proj_{\bone}=\bone \pi^{\sT}$. Thus, combining with Lemma~\ref{lem:SS:perp}, we have
\beqn
\Xi\circ \proj_{\SS^{\perp}}=\big(\bone \pi^{\sT}\big)\ot\Phi+\Phi\ot \big(\bone \pi^{\sT}\big)-\big(\bone \pi^{\sT}\big)\ot\big(\bone \pi^{\sT}\big)\,.
\eeqn
Since $\SS$ and $\SS^{\perp}$ are invariant subspaces of $\Xi$ by Lemma~\ref{lem:SS:perp}, $\Xi=\Xi\circ\proj_{\SS}+\Xi\circ\proj_{\SS^{\perp}}$ holds by the spectral theorem. Thus, it follows that
\beqn
\Xi\circ \proj_{\SS}=\Xi-\big(\bone \pi^{\sT}\big)\ot\Phi-\Phi\ot \big(\bone \pi^{\sT}\big)+\big(\bone \pi^{\sT}\big)\ot\big(\bone \pi^{\sT}\big)=\E_{p}\bigg[\Big(\Phi-\bone  \pi^{\sT}\Big)^{\otimes 2}\bigg]\equiv \Xi_{\ast}\,,
\eeqn
which concludes the proof.
\end{proof}
\begin{proof}[Proof of Lemma~\ref{lem:lambda:equality}]
    For each $\ell\geq 1$, we have by definition,
    \beqn
    \sum_{\zeta\in S_{\ell}}\lambda_{\zeta}\delta_{\zeta}^2=\sum_{\psi_1,\ldots, \psi_{\ell}\in \Psi}\sum_{\substack{s_1,t_1,\ldots, s_{\ell},t_{\ell}\in [k]\\ s_i\neq t_i, 1\leq i\leq \ell}}\frac{1}{2\ell}\left(\frac{d}{k}\right)^{\ell}\bigg(\tr\Big(\prod_{i=1}^{\ell}\Phi_{\psi_i,s_i,t_i}\Big)-1\bigg)^2\prod_{i=1}^{\ell}p(\psi_i)\,.
    \eeqn
   Note that given $s,t\in [k]$ with $s\neq t$, there exists a permutation $\theta$ in $[k]$ such that $\theta(1)=s, \theta(2)=t$, and for such $\theta$, we have $\Phi_{\psi,s,t}=\Phi_{\psi^{\theta}}$, where $\psi^{\theta}(\sigma_1,\ldots, \sigma_k)\equiv \psi(\sigma_{\theta(1)},\ldots, \sigma_{\theta(k)})$. Thus, by our assumption that $p(\psi^{\theta})=p(\psi)$ in \eqref{eq:psi:symm}, we can express the equation above by
   \beqn
   \sum_{\zeta\in S_{\ell}}\lambda_{\zeta}\delta_{\zeta}^2=\sum_{\psi_1,\ldots, \psi_{\ell}\in \Psi}\frac{\big((k-1)d\big)^{\ell}}{2\ell}\bigg(\tr\Big(\prod_{i=1}^{\ell}\Phi_{\psi_i}\Big)-1\bigg)^2\prod_{i=1}^{\ell}p(\psi_i)=\frac{\big((k-1)d\big)^{\ell}}{2\ell}\E\Bigg[\bigg(\tr\Big(\prod_{i=1}^{\ell}\Phi_{\bpsi_i}\Big)-1\bigg)^2\Bigg]\,,
   \eeqn
   where $(\bpsi_i)_{i\leq \ell}\iid p$. Note that we can expand the rightmost term by
   \beqn
\E\Bigg[\bigg(\tr\Big(\prod_{i=1}^{\ell}\Phi_{\bpsi_i}\Big)-1\bigg)^2\Bigg]=\E\bigg[\tr\bigg(\prod_{i=1}^{\ell}\Big(\Phi_{\bpsi_i}\otimes \Phi_{\bpsi_i}\Big)\bigg)\bigg]-2\E\bigg[\tr\Big(\prod_{i=1}^{\ell}\Phi_{\bpsi_i}\Big)\bigg]+1=\tr\big(\Xi^{\ell}\big)-2\tr\big(\Phi^{\ell}\big)+1\,,
   \eeqn
   where the last equality holds since the expectation and trace is exchangeable and $\bpsi_i$'s are independent. Since $\Xi$ and $\Phi$ are self-adjoint by Lemma~\ref{lem:Xi:self} and Lemma~\ref{lem:Phi:self}, we have
   \beqn
   \tr\big(\Xi^{\ell}\big)-2\tr\big(\Phi^{\ell}\big)+1=\sum_{\la\in \Eig(\Xi)}\la^{\ell}-\sum_{\la^\prime\in \Eig(\Phi)}(\la^\prime)^{\ell}+1=\sum_{\la\in \Eig_{\SS}(\Xi)}\la^{\ell}\,,
   \eeqn
   where the last equality follows from Lemma~\ref{lem:SS:perp}. Thus, combining the $3$ equations in the displays above and summing over $\ell\geq 1$, we have
   \beqn
   \sum_{\ell\geq 1}\sum_{\zeta\in S_{\ell}}\la_{\zeta}\delta_{\zeta}^2=\sum_{\ell\geq 1}\sum_{\la\in \Eig_{\SS}(\Xi)}\frac{\big((k-1)d\big)^{\ell}}{2\ell}\la^{\ell}=\sum_{\la\in \Eig_{\SS}(\Xi)}\frac{1}{2}\log\big(1-(k-1)d\la\big)\,,
   \eeqn
   where the last equality holds for $d<d_{\ks}$ since $\sum_{\ell\geq 1}\frac{x^{\ell}}{\ell}=\log(1-x)$ holds for for $|x|<1$. Therefore, exponentiating the equation above concludes the proof.
\end{proof}

\subsection{Proof of Lemma~\ref{lem:m:n:overlap:trivial}}
Fix $d<d_{\ast}$. We first claim that for any $(m_n)_{n\geq 1}$ such that $\big|m_n-dn/k\big|\leq n^{2/3}$, weak recovery is impossible for $\bG^\star(n,m_n)$. We refer to Definition~\ref{def:weak:recovery:general} for a general definition of weak recovery. To this end, fix such $(m_n)_{n\geq 1}$ and $\eta>0$ such that $d<d+\eta<d_{\ast}$ holds. Then, weak recovery is impossible at $d+\eta$ by definition of $d_{\ast}$. That is, for any estimator $\hat{\sig}\equiv \hat{\sig}(\bG^\star(n,\bm)\big)\in [q]^V$, where $\bm\sim \Poi\big((d+\eta)n/k\big)$, we have 
\begin{equation}\label{eq:impossibility:weak:recovery}
\lim_{n\to\infty} \E\big[A(\bsig^\star,\hat{\sig})]=\frac{1}{q}\,.
\end{equation}
Note that by definition of the planted model (cf. Definition \ref{def:planted}), conditional on the event $\bm\geq m_n$, the subgraph of $\bG^\star(n,\bm)$ formed by excluding $\bm-m_n$ clauses is distributed the same as $\bG^\star(n,m_n)$. Moreover, since $(d+\eta)n/k\geq m_n+\Omega(n)$, it follows from Chernoff bound that $\bm>m_n$ holds with high probability. As a consequence, for any estimator $\hat{\sig}\equiv \hat{\sig}\big(\bG^\star(n,m_n)\big)$ of $\bG^\star(n,m_n)$, \eqref{eq:impossibility:weak:recovery} must also hold. Therefore, for any $(m_n)_{n\geq 1}$ such that $\big|m_n-dn/k\big|\leq n^{2/3}$, weak recovery is impossible for $\bG^\star(n,m_n)$.

Note that since $m_n\to\infty$ as $n\to\infty$, Lemma~\ref{lem:EXG:CONV} implies that $\big(\bG^\star(n,m_n)\big)_{n\geq 1}$ satisfies the properties {\sf (EXG)}. Thus, it follows from Proposition~\ref{prop:equiv} that as $n\to\infty$,
\begin{equation*}
\delta_n:=\E\Big\langle \big\|R_{\sig^1,\sig^2}-\pi\pi^{\sT}\big\|_1\Big\rangle_{\bG^\star(n,m_n)}\to 0\,.
\end{equation*}
Let $\eps_n:=\delta_n^{1/2}$ and consider $L^{\ast}\big(\bG(n,m_n))\equiv L^{\ast}\big(\bG(n,m_n);\eps_n)$. Then, we have by a change of measure (cf. \eqref{eq:diff:L:L:star}) that
\beqn
\begin{split}
\E\bigg[\Big| L\big(\bG(n,m_n)\big)-L^{\ast}\big(\bG(n,m_n))\Big|\bigg]
&=\P\Big(\big\langle \big\|R_{\sig^1,\sig^2}-\pi\pi^{\sT}\big\|_1\big\rangle_{\bG^\star(n,m_n)}>\eps_n\Big)\\
&\leq \eps_n^{-1} \E\big\langle \big\|R_{\sig^1,\sig^2}-\pi\pi^{\sT}\big\|_1\big\rangle_{\bG^\star(n,m_n)}=\delta_n^{1/2}\,,
\end{split}
\eeqn
where the second equality is due to the change of measure and the inequality is by Markov's inequality. Since $\delta_n\to 0$ as $n\to\infty$, this concludes the proof.

\subsection{Proof of Proposition~\ref{prop:sec:moment:LR}}
\label{subsec:prop:sec:moment:LR}
Throughout this subsection, we fix $d< d_{\ks}$ and assume the condition ${\sf (SYM)}$. For any factor graph $G$ with $n$ variables and $m$ clauses, note that we can express the square of the truncated likelihood ratio $L^{\ast}(G)^2$ as
\beq\label{eq:decompose:L:square}
\begin{split}
&L^{\ast}(G)^2=\sum_{\sig,\utau\in [q]^V}L^2_{\sig,\utau}(G)\,,\quad\textnormal{where}\\
&L^2_{\sig,\utau}(G):=\frac{\P\big(\bsig^\star=\sig, \bG^\star(n,m)=G\big)\cdot \P\big(\bsig^\star=\utau, \bG^\star(n,m)=G\big)}{\P\big(\bG(n,m)=G\big)^2}\one\Big\{\big\langle \big\|R_{\sig^1,\sig^2}-\pi\pi^{\sT}\big\|_1\big\rangle_{G}\leq \eps_n\Big\}\,.
\end{split}
\eeq
We divide the sum above into the \textit{near-independent} regime where $\sig, \tau\in [q]^V$ satisfies $\big\|R_{\sig,\utau}-\pi\pi^{\sT}\big\|_1\leq \sqrt{\eps_n}$, and \textit{correlated} regime where $\big\|R_{\sig,\utau}-\pi\pi^{\sT}\big\|_1>\sqrt{\eps_n}$. Namely, we have $L^{\ast}(G)^2=L^2_{\ind}(G)+L^2_{\co}(G)$, where
\beq\label{eq:L:decompose}
\begin{split}
L^2_{\ind}(G)=\sum_{\sig,\utau\in [q]^V:\|R_{\sig,\utau}-\pi\pi^{\sT}\|_1\leq \sqrt{\eps_n}}L^2_{\sig,\utau}(G)\,, \quad\quad L^2_{\co}(G)=\sum_{\sig,\utau\in [q]^V:\|R_{\sig,\utau}-\pi\pi^{\sT}\|_1> \sqrt{\eps_n}}L^2_{\sig,\utau}(G)\,.
\end{split}
\eeq
An important consequence of the truncation $\one\big\{\big\langle \|R_{\sig^1,\sig^2}-\pi\pi^{\sT}\|_1\big\rangle_{G}\leq \eps_n\big\}$ in $L^{\ast}(G)$ is that the contribution to its second moment from the correlated regime is negligible compared to the contribution from the near-independence regime as seen by the following lemma.
\begin{lemma}\label{lem:cor:negligible}
    For any factor graph $G$, $L^2_{\co}(G)\leq \sqrt{\eps_n}\cdot \big
    (L^{\ast}(G)\big)^2$ holds.
\end{lemma}
\begin{proof}
Recalling the posterior $\mu_{G}(\sig)$ (cf. \eqref{def:posterior}), note that $L_{\sig,\utau}(G)=\big(L^{\ast}(G)\big)^2\mu_{G}(\sig)\mu_{G}(\utau)$ holds. Thus, it follows that
\beqn
L^2_{\co}(G)=\big(L^{\ast}(G)\big)^2\cdot \bigg\langle \one\Big\{\big\|R_{\sig^1,\sig^2}-\pi\pi^{\sT}\big\|_1>\sqrt{\eps_n}\Big\}\bigg\rangle_{G}\leq \eps_n^{-1/2}\big(L^{\ast}(G)\big)^2\Big\langle \big\|R_{\sig^1,\sig^2}-\pi\pi^{\sT}\big\|_1\Big\rangle_{G}\,.
\eeqn
where the inequality is due to Markov's inequality. Since $L^{\ast}(G)=0$ holds if $\big\langle\big\|R_{\sig^1,\sig^2}-\pi\pi^{\sT}\big\|\big\rangle_{G}>\eps_n$, the RHS can be further bounded by $L^2_{\co}(G)\leq \sqrt{\eps_n}\big(L^{\ast}(G)\big)^2$, which concludes the proof.
\end{proof}
Having Lemma~\ref{lem:cor:negligible} in hand, we now compute $\E L^2_{\ind}\big(\bG(n,m_n)\big)$. The first step is to compute $\E L^2_{\sig,\utau}\big(\bG(n,m)\big)$.
\begin{lemma}\label{lem:L:sig:utau:moment}
    For any $\sig, \utau\in [q]^V$ and $n,m\geq 1$, we have
    \beqn
    \E L^2_{\sig,\utau}\big(\bG(n,m)\big)\leq \left(\frac{\E_{p,u}\big[\bpsi(\sig_{\bom})\bpsi(\utau_{\bom})\big]}{\E_{p,u}\big[\bpsi(\sig_{\bom}\big]\cdot\E_{p,u}\big[\bpsi(\utau_{\bom}\big]}\right)^m \cdot \P(\bsig^\star=\sig)\cdot\P(\bsig^\star=\utau)\,.
    \eeqn
\end{lemma}
\begin{proof}
Recall that by definition of the planted model (cf. Definition~\ref{def:planted}), $\frac{\P(\bG^\star(n,m)=G\given \bsig^\star=\sig)}{\P(\bG(n,m)=G)}=\frac{\psi_{G}(\sig)}{\E[\psi_{\bG(n,m)}(\sig)]}$ holds for a factor graph $G$ with $n$ variables and $m$ clauses. Thus, by dropping the indicator in the definition of $L^2_{\sig,\utau}(G)$ in \eqref{eq:decompose:L:square}, we have
\beqn
\E L^2_{\sig,\utau}\big(\bG(n,m)\big)\leq \frac{\E\big[\psi_{\bG(n,m)}(\sig)\psi_{\bG(n,m)}(\tau)\big]}{\E\big[\psi_{\bG(n,m)}(\sig)\big]\cdot \E\big[\psi_{\bG(n,m)}(\tau)\big]}\cdot \P(\bsig^\star=\sig)\cdot\P(\bsig^\star=\utau)\,.
\eeqn
Meanwhile, for the null model $\bG(n,m)$, its weight functions $(\psi_a)_{a\in F}$ and the neighborhoods $(\delta a)_{a\in F}$ are drawn i.i.d. from $p$ and $u:=\Unif(V^k)$ respectively. Thus, we have
\beqn
\frac{\E\big[\psi_{\bG(n,m)}(\sig)\psi_{\bG(n,m)}(\tau)\big]}{\E\big[\psi_{\bG(n,m)}(\sig)\big]\cdot \E\big[\psi_{\bG(n,m)}(\tau)\big]}=\left(\frac{\E_{p,u}\big[\bpsi(\sig_{\bom})\bpsi(\utau_{\bom})\big]}{\E_{p,u}\big[\bpsi(\sig_{\bom}\big]\cdot\E_{p,u}\big[\bpsi(\utau_{\bom}\big]}\right)^m\,,
\eeqn
which concludes the proof.
\end{proof}
Lemma~\ref{lem:L:sig:utau:moment} shows that in order to bound $\E L^2_{\ind}\big(\bG(n,m_n)\big)$, it suffices to compute the expected value of $\Big(\frac{\E_{p,u}[\bpsi(\sig_{\bom})\bpsi(\utau_{\bom})]}{\E_{p,u}[\bpsi(\sig_{\bom}]\cdot\E_{p,u}[\bpsi(\utau_{\bom}]}\Big)^m$ under $\sig,\utau\iid \pi^{\otimes V}$ conditioned on the event $\|R_{\sig,\utau}-\pi\pi^{\sT}\|_1\leq \sqrt{\eps_n}$, which happens w.h.p. since we assumed $n^{-1}\eps_n\to\infty$. As we will see next, $\E_{p,u}[\bpsi(\sig_{\bom})\bpsi(\utau_{\bom})]$ and $\E_{p,u}[\bpsi(\sig_{\bom})]$ can be computed in terms of the normalized overlap matrix $X\equiv X_{\sig,\utau}\equiv \big(X(i,j)\big)_{i,j\leq q} \in \R^{q\times q}$, where
\beq\label{def:X}
X\equiv X_{\sig,\utau}:= \sqrt{n}\left(R_{\sig,\utau}-\pi\pi^{\sT}\right)\,.
\eeq
Note that such normalization guarantees that $\sum_{i,j=1}^{q}X(i,j)=0$ and $X$ has $O(1)$ fluctuations: by viewing $X$ as a $q^2$ dimensional vector and denoting $\bX_n\equiv X_{\sig,\utau}$ for $\sig,\utau\sim \pi^{\otimes n}$, the central limit theorem shows that 
\beq\label{eq:clt}
\bX_n \dto \normal\big(0, \diag(\pi)^{\otimes 2}-(\pi\pi^{\sT})^{\otimes 2}\big)\,.
\eeq
Here, for two matrices $A=(A_{i,j})_{i,j\leq q}$ and $B=(B_{i,j})_{i,j\leq q}$, we recall that their tensor product $A\otimes B$ is defined by the $q^2\times q^2$ matrix with $\big((i_1,j_1), (i_2,j_2)\big)$ entry $(A\otimes B)_{(i_1,j_1),(i_2,j_2)}\equiv A_{i_1,i_2}B_{j_1,j_2}$.

In addition, for $\psi \in \Psi$, we define the matrix $\wh{\Phi}_{\psi}\equiv \big(\wh{\Phi}_{\psi}(i,j)\big)_{i,j\leq q}$ by
\beqn
\wh{\Phi}_{\psi}(i,j):=\xi^{-1}\cdot\E_{\pi}\big[\psi(\bsig)\bgiven \bsigma_1=i, \bsigma_2=j\big]\,.
\eeqn
Here, we note that $\wh{\Phi}_{\psi}\diag(\pi)=\Phi_{\psi}$ holds by definition. We then define the the matrix $\wh{\Phi}\in \R^{q\times q}$ by $\wh{\Phi}:=\E_{p}\big[\wh{\Phi}_{\bpsi}\big]$ and define the matrix $\wh{\Xi}\in \R^{q^2\times q^2}$ by the averaged tensor product of the matrix $\wh{\Phi}_{\bpsi}$ with itself:
\beqn
\wh{\Xi}:=\E_{p}\Big[\wh{\Phi}_{\bpsi}\otimes \wh{\Phi}_{\bpsi}\Big]\,.
\eeqn
Note that viewing $X$ as a $q^2$ dimensional vector, quantities such as $\big\langle \wh{\Xi}X, X\big\rangle$ is well-defined, where
\beqn
\Big\langle \wh{\Xi}X, X\Big\rangle\equiv \sum_{(i_1,j_1),(i_2,j_2)\in [q]^2}\wh{\Xi}\big((i_1,j_1),(i_2,j_2)\big) X(i_1,j_1)X(i_2,j_2)\,.
\eeqn
We also recall that $\bone\in \R^{q}$ denoted all $1$-vector.
\begin{lemma}\label{lem:main:approx}
For $\sig,\utau\in [q]^V$, we have for $X\equiv X_{\sig,\utau}$ that
\beq\label{eq:lem:main:approx:first}
\E_{p,u}\big[\bpsi(\sig_{\bom})\bpsi(\utau_{\bom})\big]=\xi^2\cdot \exp\bigg(\frac{k(k-1)}{2n}\Big\langle \wh{\Xi}X\,,\,X\Big\rangle+O_{k,q,\Psi}\bigg(\frac{\|X\|_{\infty}^3}{n^{3/2}}\bigg)\bigg)\,,
\eeq
where we used the notation $f=O_{k,q,\Psi}\Big(\frac{\|X\|_{\infty}^3}{n^{3/2}}\Big)$ to indicate that there exists a constant $C\equiv C_{k,q,\Psi}$ which only depends on $k,q,\Psi$ such that $|f|\leq C\frac{\|X\|_{\infty}^3}{n^{3/2}}$ holds. Further, we have
\beq\label{eq:lem:main:approx:second}
\begin{split}
&\E_{p,u}\big[\bpsi(\sig_{\bom})\big]=\xi\cdot \exp\bigg(\frac{k(k-1)}{2n}\Big\langle \big(\wh{\Phi}\otimes\bone\bone^{\sT}\big) X\,,\,X\Big\rangle+O_{k,q,\Psi}\bigg(\frac{\|X\|_{\infty}^3}{n^{3/2}}\bigg)\bigg)\,,\\
&\E_{p,u}\big[\bpsi(\utau_{\bom})\big]=\xi\cdot \exp\bigg(\frac{k(k-1)}{2n}\Big\langle \big(\bone\bone^{\sT}\ot \wh{\Phi}\big) X\,,\,X\Big\rangle+O_{k,q,\Psi}\bigg(\frac{\|X\|_{\infty}^3}{n^{3/2}}\bigg)\bigg)\,.
\end{split}
\eeq
\end{lemma}
\begin{proof}
 We start with the proof of \eqref{eq:lem:main:approx:first}. By first calculating the expecation w.r.t. $\bom\sim u:=\Unif(V^k)$, we can express $\E_{p,u}[\bpsi(\sig_{\bom})\bpsi(\utau_{\bom})]$ by summing the following over $\ui\equiv (i_1,\ldots, i_k),\uj\equiv (j_1,\ldots, j_k)\in [q]^k$:
    \beqn
    \begin{split}
    \E_{p,u}\big[\bpsi(\sig_{\bom})\bpsi(\utau_{\bom})\big]
    &=\sum_{\ui, \uj\in [q]^k}\E_p\big[\bpsi(\ui)\bpsi(\uj)\big]\prod_{s=1}^{k}\Big(\pi_{i_s}\pi_{j_s}+n^{-1/2}X(i_s,j_s)\Big)\\
    &=\sum_{\ell=0}^{k}\sum_{S\subseteq[k]:|S|=\ell}n^{-\ell/2}\sum_{\ui, \uj\in [q]^k}\E_p\big[\bpsi(\ui)\bpsi(\uj)\big]\prod_{t\in [k]\setminus S}\pi_{i_{t}}\pi_{j_t}\prod_{s\in S}X(i_{s}, j_{s})\,.
    \end{split}
    \eeqn
    Given $\ell$ and $S$ with $|S|=\ell$, we divide the sum $\ui\in [q]^k$ into the sum over $\ui_1\equiv (i_s)_{s\in S}\in [q]^{S}$ and $\ui_2\equiv (i_t)_{t\notin S}\in [q]^{[k]\setminus S}$, and similarly for the sum $\uj\in [q]^k$. Note that given $\ui_1,\uj_1\in [q]^{S}$, we have
    \beqn
    \sum_{\ui_2,\uj_2\in [q]^{[k]\setminus S}}\E_p\big[\bpsi(\ui)\bpsi(\uj)\big]\prod_{t\in [k]\setminus S}\pi_{i_{t}}\pi_{j_t}=\E_{p}\Big[\E_{\pi}\big[\bpsi(\bsig)\bgiven \bsig_{S}=\ui_1\big]\cdot\E_{\pi}\big[\bpsi(\bsig)\bgiven \bsig_{S}=\uj_1\big]\Big]\,,
    \eeqn
    where $\E_{\pi}$ is taken w.r.t. $\bsig\sim \pi^{\otimes k}$ conditional on $\bpsi$ and $\bsig_{S}\equiv (\bsigma_s)_{s\in S}$. By our assumption that $\bpsi\stackrel{d}{=}\bpsi^{\theta}$ for $\psi\sim p$ and any permutation $\theta$ in \eqref{eq:psi:symm}, the RHS above does not depend $S$ given $\ui_1,\uj_1$. Thus, by fixing $S=[\ell]$, we can express $  \E_{p,u}\big[\bpsi(\sig_{\bom})\bpsi(\utau_{\bom})\big]$ by
    \beq\label{eq:lem:main:approx:simplify}
    \E_{p,u}\big[\bpsi(\sig_{\bom})\bpsi(\utau_{\bom})\big]=\sum_{\ell=0}^{k}\binom{k}{\ell}n^{-\ell/2}\sum_{\ui,\uj\in [q]^{\ell}}\E_{p}\Big[\E_{\pi}\big[\bpsi(\bsig)\bgiven \bsig_{[\ell]}=\ui\big]\cdot\E_{\pi}\big[\bpsi(\bsig)\bgiven \bsig_{[\ell]}=\uj\big]\Big]\cdot\prod_{s=1}^{\ell}X(i_s,j_s)\,.
    \eeq
    We now divide the summand above into the cases $\ell\in \{0,1,2\}$ and $\ell\geq 3$. For $\ell=0$, by ${\sf (SYM)}$, 
    \beq\label{eq:lem:main:approx:0}
\E_{p}\Big(\E_{\pi}\big[\bpsi(\bsig)\big]\Big)^2=\sum_{\psi\in \Psi}p(\psi)\Big(\E_{\pi}\big[\psi(\bsig)\big]\Big)^2=\xi^2\,.
    \eeq
    For $\ell=1$, again by ${\sf (SYM)}$, we have
    \beq\label{eq:lem:main:approx:1}
  k n^{-1/2}\sum_{i, j=1}^{q} X(i,j)\E_{p}\Big[\E_{\pi}\big[\bpsi(\bsig)\bgiven \bsig_1=i\big]\cdot\E_{\pi}\big[\bpsi(\bsig)\bgiven \bsig_1=j\big]\Big]= kn^{-1/2}\sum_{i, j=1}^{q} X(i,j) \xi^2=0\,,
    \eeq
    where the last equality follows since $\sum_{i,j=1}^{q}X(i,j)=0$. For $\ell=2$, we have by definition of $\wh{\Xi}$ that
    \beq\label{eq:lem:main:approx:2}
    \binom{k}{2}n^{-1}\sum_{\ui,\uj\in [q]^2}\E_{p}\Big[\E_{\pi}\big[\bpsi(\bsig)\bgiven \bsig_{\{1,2\}}=\ui\big]\cdot\E_{\pi}\big[\bpsi(\bsig)\bgiven  \bsig_{\{1,2\}}=\uj\big]\Big]\cdot\prod_{s=1}^{2}X(i_s,j_s)=\binom{k}{2}\frac{\xi^2}{n}\Big\langle \wh{\Xi}X\,,\, X\Big\rangle\,.
    \eeq
    Finally, for $\ell \geq 3$, we can crudely bound
    \beq\label{eq:lem:main:approx:3}
    \begin{split}
    &\sum_{\ell\geq 3}\binom{k}{\ell}n^{-\ell/2}\sum_{\ui,\uj\in [q]^{\ell}}\E_{p}\Big[\E_{\pi}\big[\bpsi(\bsig)\bgiven \bsig_{[\ell]}=\ui\big]\cdot\E_{\pi}\big[\bpsi(\bsig)\bgiven \bsig_{[\ell]}=\uj\big]\Big]\cdot\prod_{s=1}^{\ell}X(i_s,j_s)\\
    &\leq \max_{\psi\in \Psi}\|\psi\|_{\infty}^2 \cdot \sum_{\ell\geq 3}\binom{k}{\ell}n^{-\ell/2}q^{2\ell}\cdot\|X\|_{\infty}^{\ell}\leq  \max_{\psi\in \Psi}\|\psi\|_{\infty}^2\cdot (1+q)^{2k}\cdot n^{-3/2}\|X\|_{\infty}^3\,,
    \end{split}
    \eeq
    where the last inequality follows since $\|X\|_{\infty}\leq \sqrt{n}$ by definition of $X$ in \eqref{def:X}. Therefore, plugging in the equalities \eqref{eq:lem:main:approx:0}-\eqref{eq:lem:main:approx:2} and the bound \eqref{eq:lem:main:approx:3} into \eqref{eq:lem:main:approx:simplify}, we have
    \beq
    \E_{p,u}\big[\bpsi(\sig_{\bom})\bpsi(\utau_{\bom})\big]=\xi^2\cdot \bigg(1+\frac{k(k-1)}{2n}\Big\langle \wh{\Xi}X\,,\,X\Big\rangle+O_{k,q,\Psi}\bigg(\frac{\|X\|_{\infty}^3}{n^{3/2}}\bigg)\bigg)\,.
    \eeq
    Note that since $\psi(\cdot)>0$ for any $\psi\in \Psi$, there exists a constant $c_i\equiv c_{i,\Psi}$ for $i=1,2$, such that $\E_{p,u}\big[\bpsi(\sig_{\bom})\bpsi(\utau_{\bom})\big]\in [c_1\xi^2,c_2\xi^2]$. Moreover, there exists a constant $C\equiv C_{c_1,c_2}>0$ such that for $1+x\in [c_1,c_2]$, $e^{x-Cx^2}\leq 1+x\leq e^{x}$ holds. Applying this inequality for $x=k(k-1)\big\langle \wh{X},X\big\rangle/n +O_{k,q,\Psi}\big(\|X|\|_{\infty}^3/n^{3/2}\big)$ and using the fact that $\|X\|_{\infty}\leq \sqrt{n}$, we obtain the estimate \eqref{eq:lem:main:approx:first}.
    
    Next, we prove \eqref{eq:lem:main:approx:second} by a similar argument. Denote by $X(i,\cdot ):=\sum_{j=1}^{q}X(i,j)$. Then, proceeding in the same manner as in \eqref{eq:lem:main:approx:simplify}, we can express $\E_{p,u}\big[\bpsi(\sig_{\bom})\big]$ by
    \beq\label{eq:lem:main:approx:simplify:2}
    \E_{p,u}\big[\bpsi(\sig_{\bom})\big]=\sum_{\ell=0}^{k}\binom{k}{\ell}n^{-\ell/2}\sum_{\ui\in [q]^{\ell}}\E_{p}\Big[\E_{\pi}\big[\bpsi(\bsig)\bgiven \bsig_{[\ell]}=\ui\big]\Big]\prod_{s=1}^{\ell}X(i_s,\cdot)=:\sum_{\ell=0}^{k}G(\ell)\,.
    \eeq
    Then, by the same calculations done in \eqref{eq:lem:main:approx:0} and \eqref{eq:lem:main:approx:1}, $G(0)=\xi$ and $G(1)=0$ hold by ${\sf (SYM)}$. For $\ell=2$, we have
    \beqn
    \begin{split}
    G(2)
    &\equiv \binom{k}{2}n^{-1}\sum_{\ui\in [q]^{2}}\E_{p}\Big[\E_{\pi}\big[\bpsi(\bsig)\bgiven \bsig_{\{1,2\}}=\ui\big]\Big]\prod_{s=1}^{2}X(i_s,\cdot)\\
    &= \binom{k}{2}\frac{\xi}{n}\sum_{i_1,i_2,j_1,j_2\in [q]}\Phi(i_1,i_2)\prod_{s=1}^{2}X(i_s,j_s)=\binom{k}{2}\frac{\xi}{n}\Big\langle \big(\wh{\Phi}\otimes\bone\bone^{\sT}\big) X\,,\,X\Big\rangle\,.
    \end{split}
    \eeqn
    For $\ell \geq 3$, proceeding in the same manner as done in \eqref{eq:lem:main:approx:3}, we can bound
    \beqn
    \sum_{\ell \geq 3}G(3)\equiv \sum_{\ell\geq 3}\binom{k}{\ell}n^{-\ell/2}\sum_{\ui,\uj\in [q]^{\ell}}\E_{p}\Big[\E_{\pi}\big[\bpsi(\bsig)\bgiven \bsig_{[\ell]}=\ui\big]\Big]\prod_{s=1}^{\ell}X(i_s,j_s)\leq \max_{\psi\in \Psi}\|\psi\|_{\infty}^2\cdot (1+q)^{2k}\cdot n^{-3/2}\|X\|_{\infty}^3
    \eeqn
    By plugging in the obtained bounds into \eqref{eq:lem:main:approx:simplify:2}, and using the inequality $e^{x-O(x^2)} \leq 1+x\leq e^{x}$ for $1+x$ bounded away from $0$ and $\infty$ as before, we obtain the first inequality of \eqref{eq:lem:main:approx:second}. The second inequality of \eqref{eq:lem:main:approx:second} then follows from the first by exchanging the role of $\sig$ and $\utau$.
\end{proof}
Define the matrix $\wh{\Xi}_{\ast}\in \R^{q^2\times q^2}$ by \beqn
\wh{\Xi}_{\ast}=\wh{\Xi}-\wh{\Phi}\otimes \bone\bone^{\sT}-\bone\bone^{\sT}\ot \wh{\Phi}+\bone\bone^{\sT}\ot\bone \bone^{\sT}\,.
\eeqn
We note that the matrix $\wh{\Xi}_{\ast}$ is related with the linear operator $\Xi_{\ast}$ on $\R^{q}\otimes \R^{q}$ as follows. By identifying the vector spaces $\R^{q}\otimes \R^{q}$ and $\R^{q^2}$ by the unique isomorphism that maps $e_i\otimes e_j\to e_{(i,j)}$ for $i,j\in [q]$, where $(e_i)_{i\leq q}$ is the standard basis in $\R^{q}$ and $(e_{(i,j)})_{i,j\leq q}$ is the standard basis in $\R^{q^2}$, $\Xi$ can be identified with the $q^2\times q^2$ matrix $\Xi\big((i,j),(s,t)\big)_{i,j,s,t\leq q}\in \R^{q^2\times q^2}$, where
\beqn
\Xi\big((i,j),(s,t)\big)\equiv \E_{p}\Big[\Phi_{\bpsi}(i,s)\cdot\Phi_{\bpsi}(j,t)\Big]\,.
\eeqn
Similarly, $\Xi_{\ast}$ can be identified with a $q^2\times q^2$ matrix. With such identification, we have that $\Xi=\wh{\Xi}\diag(\pi)^{\otimes 2}$, thus
\beq\label{eq:relationship:hat:Xi}
\wh{\Xi}_{\ast}\diag(\pi)^{\otimes 2}=\Xi-\Phi\otimes \bone\pi^{\sT}-\bone\pi^{\sT}\otimes \Phi+\big(\bone\pi^{\sT}\big)\otimes \big(\bone\pi^{\sT}\big)=\Xi_{\ast}.
\eeq
As a consequence of Lemmas~\ref{lem:cor:negligible}, \ref{lem:L:sig:utau:moment}, and \ref{lem:main:approx}, we have the following proposition. We recall the random variable $\bX_n\equiv X_{\sig,\utau}\in \R^{q\times q}\cong \R^{q^2}$, where $\sig, \utau\sim \pi^{\otimes n}$.
\begin{prop}\label{prop:sec:mo:bound}
There exists a constant $C\equiv C_{k,q,\Psi}>0$ which only depends on $k,q,\Psi$ such that the following holds. For $d<d_{\ks}$ and any $(m_n)_{n\geq 1}, (\eps_n)_{n\geq 1}$ such that $|m_n-dn/k|\leq n^{2/3}$ and $\eps_n\to 0$, $n^{2/3}\eps_n\to\infty$ as $n\to\infty$, we have for large enough $n$ that
    \beqn
\E\Big(L^{\ast}\big(\bG(n,m_n)\big)\Big)^2\leq \big(1-\sqrt{\eps_n}\big)^{-1}\E\Bigg[\exp\bigg(\frac{(k-1)d}{2}\Big(\Big\langle \wh{\Xi}_{\ast}\bX_n\,,\,\bX_n\Big\rangle +C\sqrt{\eps_n}\|\bX_n\|_{\infty}^2\Big)\bigg)\one\Big\{\|\bX_n\|_1\leq \sqrt{n\eps_n}\Big\}\Bigg] 
    \eeqn
\end{prop}
\begin{proof}
    Recall the decomposition $L^\star(G)^2=L^2_{\ind}(G)+L^2_{\co}(G)$ in \eqref{eq:L:decompose}. By Lemma~\ref{lem:cor:negligible}, we have
    \beq\label{eq:lem:sec:mo:bound:first}
\E\Big(L^{\ast}\big(\bG(n,m_n)\big)\Big)^2
\leq(1-\sqrt{\eps_n})^{-1}\cdot \E L^2_{\ind}\big(\bG(n,m_n)\big)\,.
    \eeq
    Meanwhile, by Lemma~\ref{lem:L:sig:utau:moment}, $\E L^2_{\ind}\big(\bG(n,m_n)\big)$ can be upper bounded by
    \beqn
   \sum_{\sig,\utau\in [q]^V:\|R_{\sig,\utau}-\pi\pi^{\sT}\|_1\leq \sqrt{\eps_n}}\left(\frac{\E_{p,u}\big[\bpsi(\sig_{\bom})\bpsi(\utau_{\bom})\big]}{\E_{p,u}\big[\bpsi(\sig_{\bom}\big]\cdot\E_{p,u}\big[\bpsi(\utau_{\bom}\big]}\right)^{m_n} \cdot \P(\bsig^\star=\sig)\cdot\P(\bsig^\star=\utau)\,.
    \eeqn
    Observe that for $\sig,\utau\in [q]^V$, the restriction $\|R_{\sig,\utau}-\pi\pi^{\sT}\|_1\leq \sqrt{\eps_n}$ is equivalent to $\|X_{\sig,\utau}\|_1\leq \sqrt{n\eps_n}$. In particular, 
 $\|X_{\sig,\utau}\|_{\infty}^3/n^{3/2}\leq \sqrt{\eps_n}\|X_{\sig,\utau}\|_{\infty}$ holds under such restriction. Thus, for $\sig,\utau\in [q]^V$ such that $\|R_{\sig,\utau}-\pi\pi^{\sT}\|_1\leq \sqrt{\eps_n}$, we have by Lemma~\ref{lem:main:approx} that
    \beqn
\left(\frac{\E_{p,u}\big[\bpsi(\sig_{\bom})\bpsi(\utau_{\bom})\big]}{\E_{p,u}\big[\bpsi(\sig_{\bom})\big]\cdot\E_{p,u}\big[\bpsi(\utau_{\bom})\big]}\right)^{m_n}
\leq \exp\Bigg(\frac{m_n}{n}\binom{k}{2}\bigg(\bigg\langle \Big(\wh{\Xi}-\wh{\Phi}\otimes \bone\bone^{\sT}-\bone\bone^{\sT}\ot \wh{\Phi}\Big)X\,,\, X\bigg\rangle+C\sqrt{\eps_n}\|X\|_{\infty}^2\bigg)\Bigg),
    \eeqn
    where $C\equiv C_{k,q,\Psi}>0$ denotes a constant that only depend on $k,q,$ and the set $\Psi$. Note that since $\big\langle \bone\bone^{\sT},X\big\rangle=\sum_{i,j=1}^{q}X(i,j)=0$, the inner product above equals $\big\langle \wh{\Xi}_{\ast} X,X\big\rangle$. Moreover, since $m_n\leq dn/k+n^{3/2}$ and $n^{-1/3}\ll \sqrt{\eps_n}$ holds, the factor $\frac{m_n}{n}\binom{k}{2}$ in the RHS above can be replaced by $\frac{(k-1)d}{2}$ for large enough $n$ with the modification of the constant $C$. Therefore,
    \beqn
    \E L^2_{\ind}\big(\bG(n,m_n)\big)\leq \E\Bigg[\exp\bigg(\frac{(k-1)d}{2}\Big(\Big\langle \wh{\Xi}_{\ast}\bX_n\,,\,\bX_n\Big\rangle +C^\prime \sqrt{\eps_n}\|\bX_n\|_{\infty}^2\Big)\bigg)\one\Big\{\|\bX_n\|_1\leq \sqrt{n\eps_n}\Big\}\Bigg]\,,
    \eeqn
    for some constant $C^\prime \equiv C^\prime_{k,q,\Psi}$. Combining this with \eqref{eq:lem:sec:mo:bound:first} concludes the proof.
\end{proof}
Having Proposition~\ref{prop:sec:mo:bound} in hand, the final ingredient to prove Proposition~\ref{prop:sec:moment:LR} is the uniform integrability of $(\bZ_n)_{n\geq 1}\equiv \big(\bZ_n(d,k,\underline{\eps},C)\big)_{n\geq 1}$, where $\underline{\eps}\equiv (\eps_n)_{n\geq 1}$ and
\beqn
\bZ_n:=\exp\bigg(\frac{(k-1)d}{2}\Big(\Big\langle \wh{\Xi}_{\ast}\bX_n\,,\,\bX_n\Big\rangle +C\sqrt{\eps_n}\|\bX_n\|_{\infty}^2\Big)\bigg)\one\Big\{\|\bX_n\|_1\leq \sqrt{n\eps_n}\Big\}\,.
\eeqn
Here, we emphasize that the indicator $\one\big\{\|\bX_n\|_1\leq \sqrt{n\eps_n}\big\}$ due to the truncation $\one\big\{\big\langle\big\|R_{\sig^1,\sig^2}-\pi\pi^{\sT}\big\|\big\rangle_{G}\leq \eps_n\big\}$ in $L_{\ast}(G)$ is crucial for uniform integrability stated below.
\begin{lemma}\label{lem:UI}
    For $d<d_{\ks}$ and $\underline{\eps}=(\eps_n)_{n\geq 1}$, where $\eps_n\to 0$ as $n\to\infty$, the random variables $(\bZ_n)_{n\geq 1}$ are uniformly integrable. 
\end{lemma}
\begin{proof}
    Denote by $\PPP_n([q]^2)$ the set of $R\equiv \big(R(i,j)\big)_{i,j\leq q}\in \R^{q\times q}$ such that $R(i,j)\in \Z_{\geq 0}/n$ for any $i,j\leq q$ and $\sum_{i,j=1}^{q}R(i,j)=1$. Further, denote by $\bR_n\equiv R_{\sig,\utau}$ the overlap matrix under $\sig,\utau\iid \pi^{\otimes n}$. Then, by Sanov's theorem  (a.k.a. Stirling's approximation), we have for any $R\in \PPP_n([q]^2)$ that
    \beq\label{eq:sanov:R}
    \P\big(\bR_n=R\big)=n^{O_q(1)}\exp\big(-n\cdot \DKL(R\,\|\,\pi\pi^{\sT})\big)\,,
    \eeq
    where $O_q(1)$ denotes a bounded constant that only depends on $q$ and $\DKL(\mu\,\|\,\nu):=\sum_{i,j\in [q]}\mu(i,j)\log\big(\frac{\nu(i,j)}{\mu(i,j)}\big)$ for $\mu,\nu\in \PPP([q]^2)$ is the KL divergence (a.k.a. relative entropy) between $\mu$ and $\nu$. By viewing $R\in \PPP_n([q]^2)$ as a $q^2$ dimensional vector, we first claim that for $d<d_{\ks}$,
    \beq\label{eq:lem:UI:claim}
\nabla^2_{R}\DKL(R\,\|\,\pi\pi^{\sT})\Big\rvert_{R=\pi\pi^{\sT}} \succ (k-1)d\cdot\wh{\Xi}_{\ast}\,,
    \eeq
    where $\nabla^2_{R}$ denotes the Hessian with respect to $R\in \R^{q^2}$. Indeed, note that a direct computation gives $ \nabla^2_{R}\DKL(R\,\|\,\pi\pi^{\sT})\big\rvert_{R=\pi\pi^{\sT}}=\diag\big((\pi_i^{-1})_{i\leq q}\big)^{\otimes 2}$, thus \eqref{eq:lem:UI:claim} is equivalent to $I_{q^2\times q^2}\succ (k-1)d\cdot \wh{\Xi}_{\ast}\diag(\pi)^{\otimes 2}$, which follows for $d<d_{\ks}$ by Lemma~\ref{lem:KS} since the maximum eigenvalue of $ \wh{\Xi}_{\ast}\diag(\pi)^{\otimes 2}$ equals $\max_{\la\in\Eig(\Xi_{\ast})}|\lambda|$ by \eqref{eq:relationship:hat:Xi}. Hence, \eqref{eq:lem:UI:claim} holds for $d<d_{\ks}$.

    Now, observe that since $R\to \DKL(R\,\|\,\pi\pi^{\sT})$ is convex with $\DKL(\pi\pi^{\sT}\,\|\,\pi\pi^{\sT})=0$, \eqref{eq:lem:UI:claim} implies that there exists small enough $\eps,\delta>0$ such that if $\|R-\pi\pi^{\sT}\|_1\leq \eps$, then
    \beq\label{eq:kl:lower}
    \DKL(R\,\|\,\pi\pi^{\sT})\geq(1+\delta)\frac{(k-1)d}{2}\Big\langle \wh{\Xi}_{\ast}(R-\pi\pi^{\sT})\,,\,R-\pi\pi^{\sT}\Big\rangle+\delta\big\|R-\pi\pi^{\sT}\big\|_2^2\,.
    \eeq
    For such $\delta>0$, we now claim that $\big(\E \bZ_n^{1+\delta}\big)_{n\geq 1}$ is bounded, which is sufficient for uniform integrability. Note that $\bX_n=\sqrt{n}(\bR_n-\pi\pi^{\sT})$ holds by definition. Thus, by \eqref{eq:sanov:R}, we can bound $\E \bZ_n^{1+\delta}$ by
    \beqn
    \sum_{\substack{R\in \PPP_n([q]^2)\\\|R-\pi\pi^{\sT}\|_1\leq \sqrt{\eps_n}}}\exp\bigg(\frac{(1+\delta)(k-1)dn}{2}\Big(\Big\langle \wh{\Xi}_{\ast}(R-\pi\pi^{\sT})\,,\,R-\pi\pi^{\sT}\Big\rangle +C\sqrt{\eps_n}\big\|R-\pi\pi^{\sT}\big\|_{\infty}^2\Big)-n\DKL\big(R\,\|\,\pi\pi^{\sT}\big)\bigg).
    \eeqn
    Observe that since $\eps_n=o_n(1)$, the above sum is restricted to $\|R-\pi\pi^{\sT}\|_1\leq \sqrt{\eps_n}<\eps$ for large enough $n$. Thus, we can use the inequality \eqref{eq:kl:lower} to further bound
    \beqn
    \begin{split}
    \E\bZ_n^{1+\delta}
    &\leq\sum_{R\in \PPP_n([q]^2)}\exp\bigg(-n\delta\big\|R-\pi\pi^{\sT}\big\|_2^2+\frac{C(1+\delta)(k-1)d}{2}n\sqrt{\eps_n}\big\|R-\pi\pi^{\sT}\big\|_{\infty}^2\bigg)\\
    &\leq \sum_{R\in \PPP_n([q]^2)}\exp\bigg(-\frac{n\delta}{2}\big\|R-\pi\pi^{\sT}\big\|_2^2\bigg)\leq C^\prime\,,
    \end{split}
    \eeqn
    where the second inequality follows for large enough $n$ since $\eps_n=o_n(1)$ and $\|R-\pi\pi^{\sT}\|_{\infty}\leq \|R-\pi\pi^{\sT}\|_2$, and the last inequality follows by Gaussian integration for some constant $C^\prime\equiv C^\prime(\delta, q)<\infty$ that only depends on $\delta, q$. Therefore, $\sup_{n\geq 1} \E \bZ_n^{1+\delta}<\infty$ holds, and uniform integrability of $(\bZ_n)_{n\geq 1}$ follows.
\end{proof}
\begin{proof}[Proof of Proposition~\ref{prop:sec:moment:LR}]
Note that $\E\big(L^{\ast}\big(\bG(n,m_n)\big)\big)^2$ increases as $\eps_n$ increases. Thus, w.l.o.g., we may assume that $n^{2/3}\eps_n\to\infty$ as $n\to\infty$. Also, recall that $\bX_n\equiv \sqrt{n}(R_{\bsig,\boldsymbol{\utau}}-\pi\pi^{\sT})$, where $\bsig,\boldsymbol{\utau}\sim \pi^{\otimes n}$. Then, $\bX_n\dto \bX_{\infty}\sim \normal\big(\diag(\pi)^{\otimes 2}-(\pi\pi^{\sT})^{\otimes 2}\big)$ holds by the central limit theorem (cf.~\eqref{eq:clt}). Recalling the random variable $\bZ_n\equiv \exp\Big(\frac{(k-1)d}{2}\Big(\Big\langle \wh{\Xi}_{\ast}\bX_n,\bX_n\Big\rangle +C\sqrt{\eps_n}\|\bX_n\|_{\infty}^2\Big)\Big)\one\Big\{\|\bX_n\|_1\leq \sqrt{n\eps_n}\Big\}$, we claim that as $n\to\infty$,
\beq\label{eq:Z:conv}
\bZ_n\dto \bZ_{\infty}:=\exp\bigg(\frac{(k-1)d}{2}\Big\langle \wh{\Xi}_{\ast}\bX_{\infty}\,,\,\bX_{\infty}\Big\rangle\bigg)\,.
\eeq
To see this, note that $\exp\Big(\frac{(k-1)d}{2}\big\langle\wh{\Xi}_{\ast}\bX_n,\bX_n\big\rangle\Big)\dto \exp\Big(\frac{(k-1)d}{2}\big\langle\wh{\Xi}_{\ast}\bX_{\infty},\bX_{\infty}\big\rangle\Big)$ holds by the continuous mapping theorem. Moreover, since $n\eps_n\to\infty$ and $\eps_n\to 0$ as $n\to\infty$, we have that
\beqn
\exp\big(C\sqrt{\eps_n}\|\bX_n\|_{\infty}^3\big)\pto 1\;\;\;\;\;\textnormal{and}\;\;\;\; \one\big\{\|\bX_n\|_1\leq \sqrt{n\eps_n}\big\}\pto 1
\eeqn
Thus, by Slutsky, \eqref{eq:Z:conv} holds. Meanwhile, $(\bZ_n)_{n\geq 1}$ is uniformly integrable by Lemma~\ref{lem:UI}, thus 
\beqn
\lim_{n\to\infty}\E\bZ_n=\E\bZ_{\infty}\,.
\eeqn
Note that $\E\bZ_{\infty}$ can be calculated explicitly as follows. Recall that if $Z\sim \normal(0,\Sigma)$ and the eigenvalues of $A\Sigma$ is bounded by $1$ in absolute value for a symmetric matrix $A$, then $\E \exp\big(Z^{\sT}A Z/2\big)=\prod_{\lambda\in \Eig(A\Sigma)}(1-\lambda)^{-1/2}$ holds, where $\Eig(\Sigma^{1/2}A\Sigma)=\Eig(A\Sigma)$ denotes the set of eignevlues of $A\Sigma$ (see e.g.~\cite[Theorem 3.2a.2]{MP92}). Also, note that
\beqn
\wh{\Xi}_{\ast}\big(\pi \pi^{\sT}\big)^{\otimes 2}=\E_p\Big[\wh{\Phi}_{\bpsi}\pi\pi^{\sT}\otimes\wh{\Phi}_{\bpsi}\pi\pi^{\sT}\Big]-\wh{\Phi}\pi\pi^{\sT}\otimes \bone\pi^{\sT}-\bone\pi^{\sT}\otimes \wh{\Phi}\pi\pi^{\sT}+\big(\bone\pi^{\sT}\big)\otimes\big(\bone\pi^{\sT}\big)=0\,,
\eeqn
where the last equality holds because ${\sf (SYM)}$ yields $\wh{\Phi}_{\psi}\pi=\bone$ for any $\psi\in \Psi$. Thus, it follows that $\wh{\Xi}_{\ast}\big(\diag(\pi)^{\otimes 2}-(\pi\pi^{\sT})^{\otimes 2}\big)=\Xi_{\ast}$ holds. Therefore, by Proposition~\ref{prop:sec:mo:bound}, we can bound
\beqn
\E\Big(L^{\ast}\big(\bG(n,m_n)\big)\Big)^2\leq \E\bZ_n=\E\bZ_{\infty}+o_n(1)=\prod_{\lambda\in \Eig(\Xi_{\ast})}\frac{1}{\sqrt{1-(k-1)d\lambda}}+o_n(1)\,,
\eeqn
which concludes the proof since non-zero elements of $\Eig(\Xi_{\ast})$ equals the non-zero elements of $\Eig_{\SS}(\Xi)$ by Lemma~\ref{lem:KS}.
\end{proof}
\section{Mutual information between the planted and the null model}
\label{sec:mutual:info}
In this section, we prove Lemma~\ref{lem:KL:mutual:info:free:energy} and Propositions~\ref{prop:derivative}, ~\ref{prop:free:energy}.  In Section~\ref{subsec:mutual:info:misc}, we prove Lemma~\ref{lem:KL:mutual:info:free:energy}. In Section~\ref{subsec:mutual:info:derivative}, we prove Proposition~\ref{prop:derivative}. In Section~\ref{subsec:mutual:info:free:energy}, we prove Proposition~\ref{prop:free:energy}.

\subsection{Proof of Lemma~\ref{lem:KL:mutual:info:free:energy}}
\label{subsec:mutual:info:misc}
By invoking the identity~\eqref{eq:proof:mutual:info:tech}, we have
\beqn
\frac{1}{n}\DKL(\bG^\star\,\|\,\bG)\equiv \frac{1}{n}\sum_{G}\P(\bG^\star=G)\log\frac{\P(\bG^\star=G)}{\P(\bG=G)}=\frac{1}{n}\E\log L(\bG^\star)\,,
\eeqn
which finishes the proof of the first equality. Turning to the mutual information, we have by definition that
\beq\label{eq:}
\begin{split}
\frac{1}{n}I(\bG^\star,\bsig^\star)
&\equiv \frac{1}{n}\sum_{G,\sig}\P(\bG^\star=G,\bsig^\star=\sig)\log \bigg(\frac{\P(\bG=G)}{\P(\bG^\star=G)}\cdot \frac{\P(\bG^\star=G\given \bsig^\star=\sig)}{\P(\bG=G)}\bigg)\\
&=-\frac{1}{n}\E\log L(\bG^\star)+\frac{1}{n}\sum_{G,\sig}\P(\bG^\star=G,\bsig^\star=\sig)\log\frac{\P(\bG^\star=G\given \bsig^\star=\sig)}{\P(\bG=G)}\,,
\end{split}
\eeq
where the last equality is due to \eqref{eq:proof:mutual:info:tech}. We now simplify the last sum above. Recall that by Definition~\ref{def:planted}, if $G=(V,F,E,(\psi_a)_{a\in F})$ is a factor graph with $m$ clauses, then 
\beq\label{eq:mutual:info:tech}
\frac{\P(\bG^\star=G\given \bsig^\star=\sig)}{\P(\bG=G)}=\frac{\psi_{G}(\sig)}{\E[\psi_{\bG(n,m)}(\sig)]}=\prod_{a\in F}\frac{\psi_a(\sig_{\delta a})}{\E_{p,u}[\bpsi(\sig_{\bom})]}\,,
\eeq
where $\bpsi\sim p$ and $\bom\sim \Unif(V^k)$. Thus, if we denote by $(\psi_{a_i},\delta a_i)_{1\leq i\leq \bm}$ the random weight functions and neighborhoods of clauses of $\bG^\star$, which are i.i.d. with distribution \eqref{eq:clause:law:planted}, then it follows that 
\beq\label{eq:mutual:info:tech:1}
\frac{1}{n}\sum_{G,\sig}\P(\bG^\star=G,\bsig^\star=\sig)\log\frac{\P(\bG^\star=G\given \bsig^\star=\sig)}{\P(\bG=G)}=\frac{1}{n}\E\Bigg[\sum_{i=1}^{\bm}\log\frac{\psi_{a_i}(\bsig^\star_{\delta a_i})}{\E_{p,u}[\bpsi(\bsig^\star_{\bom})]}\Bigg]=\frac{d}{k}\E\bigg[\log\frac{\psi_{a_1}(\bsig^\star_{\delta a_1})}{\E_{p,u}[\bpsi(\bsig^\star_{\bom})]}\bigg]\,,
\eeq
where the last equality holds since $(\psi_{a_i},\delta a_i)_{1\leq i\leq \bm}$ are i.i.d., and $\E[\bm]=dn/k$. Since $(\psi_{a_1},\delta a_1)$ is distributed according to \eqref{eq:clause:law:planted}, the RHS equals
\beq\label{eq:mutual:info:tech:2}
\frac{d}{k}\E\bigg[\log\frac{\psi_{a_1}(\bsig^\star_{\delta a_1})}{\E_{p,u}[\bpsi(\bsig^\star_{\bom})]}\bigg]=\frac{d}{k}\E \bigg[\frac{\bpsi^\prime(\bsig^\star_{\bom^\prime})}{\E_{p,u}[\bpsi(\bsig^\star_{\bom})]}\cdot
\log\frac{\bpsi^\prime(\bsig^\star_{\bom^\prime})}{\E_{p,u}[\bpsi(\bsig^\star_{\bom})]}\bigg]\,,
\eeq
where the outer expectation in the RHS is w.r.t. $\bsig^\star\sim \pi^{\otimes V}$, $\bpsi^\prime \sim p$, and $\bom^\prime\sim \Unif(V^k)$. Observe that since $\bsig^\star\sim \pi^{\otimes V}$, the total variation distance between the empirical distribution of $\bsig^\star$ and $\pi$ is at most $O(n^{-1/3})$ with probability tending to one. Thus, w.h.p., $\E_{p,u}\big[\bpsi(\bsig^\star_{\bom})\big]=\E_{p,\pi}\big[\bpsi(\bsig)\big]+O(n^{-1/3})\equiv \xi +O(n^{-1/3})$ holds. In addition, by our assumption that $\psi(\cdot)\in (0,\infty)$ and $\Psi$ is finite, there exists constants $c,C\in (0,\infty)$ such that $\psi(\cdot)\in [c,C]$ holds for any $\psi\in \Psi$. As a consequence, we have that
\beq\label{eq:mutual:info:tech:3}
\begin{split}
\frac{d}{k}\E \bigg[\frac{\bpsi^\prime(\bsig^\star_{\bom^\prime})}{\E_{p,u}[\bpsi(\bsig^\star_{\bom})]}\cdot
\log\frac{\bpsi^\prime(\bsig^\star_{\bom^\prime})}{\E_{p,u}[\bpsi(\bsig^\star_{\bom})]}\bigg]
&=\frac{d}{k}\E\bigg[\frac{\bpsi^\prime(\bsig^\star_{\bom^\prime})}{\xi}\cdot \log\frac{\bpsi^\prime(\bsig^\star_{\bom^\prime})}{\xi} \bigg]+o_n(1)\\
&=\frac{d}{k}\E_{p,\pi}\bigg[\frac{\bpsi(\bsig)}{\xi}\cdot \log\frac{\bpsi(\bsig)}{\xi}\bigg]+o_n(1)\,,
\end{split}
\eeq
where in the last equality, we again used the fact that the total variation distance between the empirical distribution of $\bsig^\star$ and $\pi$ is at most $O(n^{-1/3})$ with probability tending to one. Combining \eqref{eq:mutual:info:tech:1}, \eqref{eq:mutual:info:tech:2}, \eqref{eq:mutual:info:tech:3} with \eqref{eq:mutual:info:tech} concludes the proof of Lemma~\ref{lem:KL:mutual:info:free:energy}.
\subsection{Proof of Proposition~\ref{prop:derivative}}
\label{subsec:mutual:info:derivative}
 It is well-known that given a function $f:\N\to \R$ at most of exponential growth and $\alpha>0$, the map $d\to \E f(X_d)$ for $ X_d\sim \Poi(\alpha d)$ is differentiable w.r.t. $d$ with derivative $\frac{\partial}{\partial d}\E f(X_d)=\alpha\big(\E f(X_d+1)-\E f(X_d)\big)$. Thus, it follows that
\beq\label{eq:derivative:poisson}
\frac{1}{n}\frac{\partial}{\partial d}\E\log L(\bG^\star)=\frac{1}{k}\Big(\E \log L\big(\bG^\star(n,\bm+1)\big)-\E \log L\big(\bG^\star(n,\bm)\big)\Big)\,.
\eeq
We couple $\bG^\star(n,\bm)$ and $\bG^\star(n,\bm+1)$ to show that the RHS can be calculated as follows.
\begin{lemma}\label{lem:express:ratio:L}
    For any $d>0$ and $n,m\geq 1$, there exists a coupling between $\bG^\star(n,m)$ and $\bG^\star(n,m+1)$ such that
    \beqn
    \E\log\bigg(\frac{L\big(\bG^\star(n,m+1)\big)}{L\big(\bG^\star(n,m)\big)}\bigg)=\E_{\bG^\star(n,m),p,u}\Bigg[\bigg\langle \frac{\bpsi(\sig_{\bom})}{\E_{p,u}\big[\bpsi^\prime(\sig_{\bom^\prime})\big]}\bigg\rangle_{\bG^\star(n,m)}\log\bigg\langle \frac{\bpsi(\sig_{\bom})}{\E_{p,u}\big[\bpsi^\prime(\sig_{\bom^\prime})\big]}\bigg\rangle_{\bG^\star(n,m)}\Bigg]\,,
    \eeqn
    where the expectation $\E_{\bG^\star(n,m),p,u}$ is with respect to $\bG^\star(n,m), \bpsi \sim p, \bom \sim u:=\Unif(V^k)$, and the expectation $\E_{p,u}$ in the denominator is with respect to $\bpsi^\prime \sim p$ and $\bom^\prime \sim u$.
\end{lemma}
\begin{proof}
    Recalling the Definition~\ref{def:planted} of the planted model $\bG^\star(n,m)\equiv \bG^\star(n,m,\bsig^\star)$, we can couple $\bG^\star(n,m)$ and $\bG^\star(n,m+1)$ by first drawing $\bsig^\star\sim \pi^{\otimes n}$ and $\bG^\star(n,m, \bsig^\star)$, and then conditional on $\bsig^\star, \bG^\star(n,m, \bsig^\star)$, adding an independent clause $a_{m+1}$ to $\bG^\star(n,m)$ with neighborhood $\delta a_{m+1}$ and the weight function $\psi_{a_{m+1}}$ from the distribution
    \beq\label{eq:law:new:clause}
    \P\Big(\delta a_{m+1}= (v_1,\ldots, v_k)\,,\,\psi_{a_{m+1}}=\psi\Bgiven \bsig^\star , \bG^\star(n,m,\bsig^\star)\Big)=\frac{1}{n^k}\cdot \frac{p(\psi)\psi(\bsigma^\star_{v_1},\ldots, \bsigma^\star_{v_k})}{\E_{p,u}[\bpsi(\bsig^\star_{\bom}]}\,.
    \eeq
    Recalling the notation $\psi_G(\sig)$ in \eqref{def:psi:G}, we can express $L\big(\bG^\star(n,m+1)\big)$ under such coupling by
    \beq\label{eq:L:one:more}
    L\big(\bG^\star(n,m+1)\big)= \sum_{\sig\in [q]^V}\frac{\psi_{\bG^\star(n,m+1)}(\sig)}{\E\big[\psi_{\bG(n,m+1)}(\sig)\big]}\P\big(\bsig^\star=\sig\big)=\sum_{\sig\in [q]^V}\frac{\psi_{a_{m+1}}(\sig_{\delta a_{m+1}})}{\E_{p,u}\big[\bpsi(\sig_{\bom})\big]}\cdot\frac{\psi_{\bG^\star(n,m)}(\sig)}{\E\big[\psi_{\bG(n,m)}(\sig)\big]}\P\big(\bsig^\star=\sig\big)\,.
    \eeq
     Meanwhile, by Bayes rule, the posterior $\mu_{G}(\sig)$ in \eqref{def:posterior}, where $G$ have $n$ variables $m$ clauses, equals
    \beqn
\mu_G(\sig)=\frac{\P\big(\bG^\star(n,m)=G\bgiven\bsig^\star=\sig\big)\P(\bsig^\star=\sig)}{\P\big(\bG^\star(n,m)=G\big)}=\frac{\frac{\psi_{G}(\sig)}{\E\big[\psi_{\bG(n,m)}(\sig)\big]}\P\big(\bsig^\star=\sig\big)}{L(G)}\,.
    \eeqn
    By taking $G=G^\star(n,m)$ and combining with \eqref{eq:L:one:more}, it follows that
    \beq\label{eq:express:ratio:L}
     \E\log\bigg(\frac{L\big(\bG^\star(n,m+1)\big)}{L\big(\bG^\star(n,m)\big)}\bigg)= \E\log\bigg(\sum_{\sig\in [q]^V}\frac{\psi_{a_{m+1}}(\sig_{\delta a_{m+1}})}{\E_{p,u}\big[\bpsi(\sig_{\bom})\big]}\mu_{\bG^\star(n,m)}(\sig)\bigg)= \E \log\bigg\langle \frac{\psi_{a_{m+1}}(\sig_{\delta a_{m+1}})}{\E_{p,u}\big[\bpsi(\sig_{\bom})\big]}\bigg\rangle_{\bG^\star(n,m)}\,,
    \eeq
    where in the last expression, we abbreviated the superscript in $\sig^{1}\sim\mu_{\bG^{\star}(n,m)}$ and the outer expectation is w.r.t. $\bG^\star(n,m)$ and $(\delta a_{m+1}, \psi_{a_{m+1}})$. By the conditional law of $(\delta a_{m+1}, \psi_{a_{m+1}})$ in \eqref{eq:law:new:clause}, we have
    \beqn
    \begin{split}
    \E\bigg[\log\bigg\langle \frac{\psi_{a_{m+1}}(\sig_{\delta a_{m+1}})}{\E_{p,u}\big[\bpsi(\sig_{\bom})\big]}\bigg\rangle_{\bG^\star(n,m)}\bbgiven \bsig^\star, \bG^\star(n,m)\bigg]
    &=\sum_{\psi\in \Psi}\sum_{\delta a \in V^k}\log\bigg\langle \frac{\psi(\sig_{\delta a})}{\E_{p,u}\big[\bpsi(\sig_{\bom})\big]}\bigg\rangle_{\bG^\star(n,m)}\frac{p(\psi)\psi(\bsig^\star_{\delta a})}{n^k\cdot \E_{p,u}[\bpsi(\bsig^\star_{\bom}]}\\
    &=\E_{p,u}\bigg[\frac{\bpsi^\prime(\bsig^\star_{\bom^\prime})}{\E_{p,u}\big[\bpsi(\bsig^\star_{\bom})\big]}\log\bigg\langle \frac{\psi(\sig_{\delta a})}{\E_{p,u}\big[\bpsi(\sig_{\bom})\big]}\bigg\rangle_{\bG^\star(n,m)}\bigg]\,,
    \end{split}
    \eeqn
    where the outer expectation $\E_{p,u}$ is taken w.r.t. $\bpsi^\prime \sim p$ and $\bom^\prime \sim \Unif(V^k)$ independent of everything else. Thus, by tower property, we have
    \beqn
    \begin{split}
    &\E\bigg[\log\bigg\langle \frac{\psi_{a_{m+1}}(\sig_{\delta a_{m+1}})}{\E_{p,u}\big[\bpsi(\sig_{\bom})\big]}\bigg\rangle_{\bG^\star(n,m)}\bbgiven \bG^\star(n,m)\bigg]\\
&=\E\Bigg[\E_{p,u}\bigg[\frac{\bpsi^\prime(\bsig^\star_{\bom^\prime})}{\E_{p,u}\big[\bpsi(\bsig^\star_{\bom})\big]}\log\bigg\langle \frac{\psi(\sig_{\delta a})}{\E_{p,u}\big[\bpsi(\sig_{\bom})\big]}\bigg\rangle_{\bG^\star(n,m)}\bigg]\BBgiven \bG^\star(n,m) \Bigg]\\
&=\E_{p,u}\Bigg[\bigg\langle \frac{\bpsi(\sig_{\bom})}{\E_{p,u}\big[\bpsi^\prime(\sig_{\bom^\prime})\big]}\bigg\rangle_{\bG^\star(n,m)}\log\bigg\langle \frac{\bpsi(\sig_{\bom})}{\E_{p,u}\big[\bpsi^\prime(\sig_{\bom^\prime})\big]}\bigg\rangle_{\bG^\star(n,m)}\Bigg]\,, 
     \end{split}
    \eeqn
    where the last equality follows since $\E_{p,u}$ and the conditional expectation over $\bG^\star$ is exchangeable and the expectation w.r.t. to the measure $\P\big(\bsig^\star=\cdot\bgiven \bG^\star(n,m)\big)$ equals the expectation over $\sig\sim \mu_{\bG^\star(n,m)}$. Therefore, taking expectation in the equality above and combining with \eqref{eq:express:ratio:L} concludes the proof.
\end{proof}
In order to prove Proposition~\ref{prop:derivative}, it remains to approximate $\E_{p,u}\big[\bpsi^\prime(\sig_{\bom^\prime})\big]$ under $\sig\sim \mu_{\bG^\star(n,m)}$. Recalling that $R_{\sig}$ denotes the empirical distribution of $\sig\in [q]^V$ (cf. \eqref{eq:def:sig:empirical}), consider the event
\beqn
\AA_{\bal}(n,m):=\bigg\{\P\Big(\|R_{\bsig^\star}-\pi\|_{\infty}\geq n^{-1/3}\Bgiven \bG^\star(n,m)\Big)\leq \exp\big(-n^{1/4}\big)\bigg\}\,.
\eeqn
Then, we have the following lemma.
\begin{lemma}\label{lem:bal:approx}
  There exists a constant $c\equiv c_q$ which only depends on $q$ such that for any $n,m\geq 1$, we have $\P\big(\bG^\star(n,m)\in \AA_{\bal}(n,m)\big)\leq \exp(-cn^{1/3})$. Moreover, there exists a constant $C\equiv C_{k,q,\Psi}$ depending only on $k,q,\Psi$ such that on the w.h.p. event $\AA_{\bal}$, for any $\psi\in \Psi$ and $\delta a \in V^{k}$, we have
  \beq\label{eq:lem:bal:approx}
  \Bigg|\bigg\langle \frac{\psi(\sig_{\delta a})}{\E_{p,u}\big[\bpsi(\sig_{\bom})\big]}\bigg\rangle_{\bG^\star(n,m)}-\xi^{-1}\cdot \Big\langle \psi(\sig_{\delta a})\Big\rangle_{\bG^\star(n,m)}\Bigg|\leq C n^{-1/3}.
  \eeq
\end{lemma}
\begin{proof}
For the first statement, note that by Markov's inequality and tower property, we have
\beqn
\P\big(\bG^\star(n,m)\in \AA_{\bal}(n,m)\big)\leq \exp\big(n^{1/4}\big)\cdot \P\Big(\|R_{\bsig^\star}-\pi\|_{\infty}\geq n^{-1/3}\Big)\,.
\eeqn
Invoking \eqref{eq:bal:event}, $\P\big(\|R_{\bsig^\star}-\pi\|_{\infty}\geq n^{-1/3}\big)\leq q\exp(-2n^{1/3})$ holds, thus plugging this bound shows that $\P\big(\bG^\star(n,m)\in \AA_{\bal}(n,m)\big)\leq \exp(-cn^{1/3})$ holds for some $c\equiv c_q>0$.

We next prove \eqref{eq:lem:bal:approx}. To this end, assume that $\bG^\star(n,m)\in \AA_{\bal}(n,m)$ happens, and fix any $\psi\in \Psi$ and $\delta a \in V^k$. Throughout, we denote $C$ by a positive constant that depends only depends on $k,q,\Psi$ and write $f=O_{k,q,\Psi}(g)$ if there exists a constant $C$ such that $|f|\leq Cg$. Note that
\beq\label{eq:lem:bal:approx:neg}
\bigg\langle \frac{\psi(\sig_{\delta a})}{\E_{p,u}\big[\bpsi(\sig_{\bom})\big]}\one\Big\{\big\|R_{\sig}-\pi\big\|_{\infty}\geq n^{-1/3} \Big\}\bigg\rangle_{\bG^\star(n,m)}\leq C\bigg\langle\one\Big\{\big\|R_{\sig}-\pi\big\|_{\infty}\geq n^{-1/3} \Big\}\bigg\rangle_{\bG^\star(n,m)}\leq C\exp\big(-n^{1/4}\big)\,,
\eeq
where the last inequality holds since $\langle f(\sig)\rangle_G\equiv \E\big[f(\bsig^\star)\bgiven \bG^\star=G\big]$ holds by definition for $f:[q]^V\to\R$. Meanwhile, note that if $\utau\in [q]^V$ satisfies $\|R_{\utau}-\pi\|_{\infty}\leq n^{-1/3}$, then 
\beqn
\E_{p,u}\big[\bpsi(\utau)_{\bom}\big]=\E_{p}\Bigg[\sum_{\ui=(i_1,\ldots,i_k)\in [q]^k}\bpsi(\ui)\prod_{s=1}^{k}R_{\utau}(i_s)\Bigg]=\E_{p}\Bigg[\sum_{\ui=(i_1,\ldots,i_k)\in [q]^k}\bpsi(\ui)\prod_{s=1}^{k}\pi_{i_s}\Bigg]+O_{k,q,\Psi}(n^{-1/3})\,,
\eeqn
which equals $\xi+O_{k,q,\Psi}(n^{-1/3})$. Thus, it follows that
\beqn
\begin{split}
\bigg\langle \frac{\psi(\sig_{\delta a})}{\E_{p,u}\big[\bpsi(\sig_{\bom})\big]}\one\Big\{\big\|R_{\sig}-\pi\big\|_{\infty}\leq n^{-1/3} \Big\}\bigg\rangle_{\bG^\star(n,m)}
&=\bigg\langle \frac{\psi(\sig_{\delta a})}{\xi+O_{k,q,\Psi}(n^{-1/3})}\one\Big\{\big\|R_{\sig}-\pi\big\|_{\infty}\leq n^{-1/3} \Big\}\bigg\rangle_{\bG^\star(n,m)}\\
&=\xi^{-1}\cdot \Big\langle \psi(\sig_{\delta a})\Big\rangle_{\bG^\star(n,m)}+O_{k,q,\Psi}(n^{-1/3})\,,
\end{split}
\eeqn
where the last equality holds since $\big\langle\psi(\sig_{\delta a})\big\{\|R_{\sig}-\pi\|_{\infty}\geq n^{-1/3}\big\} \big\rangle_{\bG^\star(n,m)}\leq C\exp\big(-n^{1/4}\big)$ holds on the event $\AA_{\bal}(n,m)$. Therefore, combining with \eqref{eq:lem:bal:approx:neg} completes the proof of \eqref{eq:lem:bal:approx}.
\end{proof}
The proof of Proposition~\ref{prop:derivative} is straightforward from Lemma~\ref{lem:express:ratio:L} and Lemma~\ref{lem:bal:approx}.
\begin{proof}[Proof of Proposition~\ref{prop:derivative}]
    The first statement \eqref{eq:prop:derivative:first} follows immediately from \eqref{eq:derivative:poisson} and Lemma~\ref{lem:express:ratio:L}. To establish the second statement \eqref{eq:prop:derivative} from the first statement, note that Lemma~\ref{lem:bal:approx} implies the following. Given a continuous function $F:\R_{+}\to \R$, there exists a constant $C\equiv C_{F,k,q,\Psi}$, which only depends on $F,k,q,\Psi$, such that
    \beq\label{eq:approx:conseq}
    \Bigg|\E_{\bG^\star,p,u} F\Bigg(\bigg\langle \frac{\bpsi(\sig_{\bom})}{\E_{p,u}\big[\bpsi^\prime(\sig_{\bom^\prime})\big]}\bigg\rangle_{\bG^\star}\Bigg)-\E_{\bG^\star,p,u} F\bigg(\xi^{-1}\Big\langle \bpsi(\sig_{\bom})\Big\rangle_{\bG^\star}\bigg)\Bigg|\leq Cn^{-1/3}\,.
    \eeq
    Here, we note that the continuity of $F$ suffices (rather than Lipschitz continuity) since there exist constants $c_i\equiv c_{i,\Psi}>0$ such that for any $\psi\in \Psi$, $\psi(\cdot)\in [c_1,c_2]$ holds. Thus, taking $F(x)=x\log x$, the second statement \eqref{eq:prop:derivative} follows from \eqref{eq:prop:derivative:first}.
\end{proof}
\subsection{Proof of Proposition~\ref{prop:free:energy}}
\label{subsec:mutual:info:free:energy}
In order to prove Proposition~\ref{prop:free:energy} from Proposition~\ref{prop:derivative}, we establish the following lemmas. Recall that $\bG^\star\equiv \bG^\star(n,\bm)$, where $\bm\sim \Poi(dn/k)$.
\begin{lemma}\label{lem:L:deterministic}
    There exists a constant $C\equiv C_{\Psi}>0$ such that for any $m\geq 1$, $\big|\log L(G)\big|\leq Cm$ holds for any factor graph $G$ with $m$ clauses. In particular, we have $\E \big|\log L(\bG^\star)\big|\leq Cdn/k$.
\end{lemma}
\begin{proof}
We have for a factor graph $G=(V,F,E,(\psi_a)_{a\in F})$ that
\beqn
 L(G)\equiv\sum_{\sig\in [q]^V}\P(\bsig^\star=\sig)\prod_{a\in F}\frac{\psi_a(\sig_{\delta a})}{\E_{p,u}\big[\bpsi(\sig_{\bom})\big]}\,.
\eeqn
Note that since $\psi(\cdot), \psi\in \Psi$ is bounded away from $0$ and $\infty$, there exist constants $C_{i}\equiv C_{i,\Psi}>0,i=1,2$, such that for any $\sig\in [q]^V, \delta a\in V^k, \psi\in \Psi$, $\frac{\psi(\sig_{\delta a})}{\E_{p,u}[\bpsi(\sig_{\bom})]}\in [C_1,C_2]$ holds. Thus, if $G$ has $m$ clauses, $L(G)\in [C_1^{m},C_2^{m}]$ holds. Therefore, $\big|\log L(G)\big|\leq Cm$ for some constant $C\equiv C_{\Psi}>0$. In particular, this implies that $\E\big|\log L(\bG^\star)\big|\leq C\cdot\E\bm=Cdn/k$, which concludes the proof.
\end{proof}
\begin{lemma}\label{lem:psi:R}
For $d<d_{\ast}$, we have $\E_{\bG^\star,p,u}\big(\big\langle \bpsi(\sig_{\bom})\big\rangle_{\bG^\star}-\xi\big)^2\to 0$ as $n\to\infty$. Moreover, assuming {\sf (MIN)}, for any $\eps>0$, there exists $\delta\equiv \delta(\eps,q,\pi)>0$ and $n_0\equiv n_0(\eps,q,\pi)$ which only depends on $\eps,q,\pi$ such that the following holds: if $n\geq n_0$ and there exists an estimator $\hat{\sig}\equiv \hat{\sig}(\bG^\star)$ such that $\E\big[A(\bsig^\star,\hat{\sig})\big]\geq \frac{1}{q}+\eps$ holds, then $\E_{\bG^\star,p,u}\big(\big\langle \bpsi(\sig_{\bom})\big\rangle_{\bG^\star}-\xi\big)^2>\delta$ holds.
\end{lemma}
\begin{proof}
To start with, invoking \eqref{eq:approx:conseq} for $F(x)=x$, we have that
\beq\label{eq:approx:e:psi}
\E_{\bG^\star,p,u}\big\langle \bpsi(\sig_{\bom})\big\rangle_{\bG^\star}=\xi+O_{k,q,\Psi}(n^{-1/3})\,.
\eeq
Thus, recalling the function $\FF(R)\equiv \sum_{\sig,\utau\in [q]^k}\E_p\big[\bpsi(\sig)\bpsi(\utau)\big]\prod_{s=1}^{k}R(\sigma_s,\tau_s)$, it follows that
\beq\label{eq:express:psi:zeta:squared}
\E_{\bG^\star,p,u}\big(\big\langle \bpsi(\sig_{\bom})\big\rangle_{\bG^\star}-\xi\big)^2=\E\big\langle \FF(R_{\sig^1,\sig^2})\big\rangle_{\bG^\star}-\xi^2+O_{k,q,\Psi}(n^{-1/3})\,.
\eeq
Let us now assume $d<d_{\ast}$. Then, $\E\big\langle\| R_{\sig^1,\sig^2}-\pi\pi^{\sT}\|_1\big\rangle_{\bG^\star}\to 0$ as $n\to\infty$ by Theorem~\ref{thm:equiv}. Thus, 
\beqn
\E_{\bG^\star,p,u}\big(\big\langle \bpsi(\sig_{\bom})\big\rangle_{\bG^\star}-\xi\big)^2\to \E\FF(\pi\pi^{\sT})-\xi^2=0\,.
\eeqn

To prove the second claim, assume that the condition ${\sf (MIN)}$ and fix $\eps>0$. Since $\bG^\star$ satisfies ${\sf (EXG)}$ (cf. Lemma~\ref{lem:EXG:CONV}), Lemma~\ref{lem:nonasymptotic} shows that there exists $\eta\equiv \eta(\eps,q,\pi)$ and $n_0\equiv n_0(\eps,q,\pi)$ such that if $\E\big[A(\bsig^\star,\hat{\sig})\big]\geq \frac{1}{q}+\eps$ holds for some $\hat{\sig}\equiv \hat{\sig}(\bG^\star)$, then $\E\big\langle\| R_{\sig^1,\sig^2}-\pi\pi^{\sT}\|_1\big\rangle_{\bG^\star}\geq\eta$. Note that since $\|R_{\sig,\utau}-\pi\pi^{\sT}\|_1\leq q$ for $\sig,\utau\in [q]^V$, $\E\big\langle \big\|R_{\sig^1,\sig^2}-\pi\pi^{\sT}\big\|_1\big\rangle_{\bG^\star}\geq \eta$ implies that
   \beqn
   \E\Big\langle \one\Big\{\big\|R_{\sig^1,\sig^2}-\pi\pi^{\sT}\big\|_1>\frac{\eta}{2}\Big\}\Big\rangle_{\bG^\star}\geq \frac{\eta}{2q}\,.
   \eeqn
   Meanwhile, as shown in the proof of Lemma~\ref{lem:bal:approx}, there exists $C\equiv C_q>0$ such that for $\ell=1,2$,
   \beqn
    \E\Big\langle \one\Big\{\big\|R_{\sig^{\ell}}-\pi\big\|_1>n^{-1/3}\Big\}\Big\rangle_{\bG^\star}\leq C n^{-1/3}\,,
   \eeqn
   which can be seen by conditioning on $\bm=m$ and decomposing the LHS into the events where $\bG^\star(n,m)\in \AA(n,m)$ or $\bG^\star(n,m)\notin\AA(n,m)$. Thus, there exists a constant $c\equiv c_{q}>0$ such that for large enough $n\geq n_0(\eps,q,\pi)$,
   \beq\label{eq:lem:MIN:tech}
   \E\Big\langle \one\Big\{\big\|R_{\sig^1,\sig^2}-\pi\pi^{\sT}\big\|_1>\frac{\eta}{2}~~\textnormal{and}~~\big\|R_{\sig^{\ell}}-\pi\big\|_1\leq n^{-1/3}~~\textnormal{for}~~\ell=1,2\Big\}\Big\rangle_{\bG^\star}>c\eta\,.
   \eeq
   Now, observe that by the condition ${\sf (MIN)}$, there exists $\epsilon \equiv \epsilon(\eta)>0$ such that 
   \beqn
   \min\bigg\{\FF(R):R\in [0,1]^{q}\,,\;\;\;R\bone =R^{\sT}\bone=\pi\,,\;\;\;\big\|R-\pi\pi^{\sT}\big\|\geq\frac{\eta}{2}\bigg\}>\FF(\pi\pi^{\sT})+\epsilon\,.
   \eeqn
   Thus, by continuity of $R\to \FF(R)$, there exists small enough  $\epsilon^\prime\equiv \epsilon^\prime(\eta)>0$ such that
   \beqn
    \min\bigg\{\FF(R):R\in [0,1]^{q}\,,\;\;\;\big\|R\bone-\pi\big\|_1+\big\|R^{\sT}\bone-\pi\big\|_1<\epsilon^\prime \,,\;\;\;\big\|R-\pi\pi^{\sT}\big\|\geq\frac{\eta}{2}\bigg\}>\FF(\pi\pi^{\sT})+\epsilon^\prime\,.
   \eeqn
   As a consequence, for any $\sig^1,\sig^2\in [q]^V$, we have
   \beq\label{eq:lem:MIN:tech:2}
   \big\|R_{\sig^1,\sig^2}-\pi\pi^{\sT}\big\|_1>\frac{\eta}{2}~~\textnormal{and}~~\big\|R_{\sig^{\ell}}-\pi\big\|_1\leq n^{-1/3}~~\textnormal{for}~~\ell=1,2\Longrightarrow \FF(R_{\sig^1,\sig^2})>\xi^2+\epsilon^\prime\,,
   \eeq
   where we used the fact that $\FF(\pi\pi^{\sT})=\xi^2$. Finally, recalling the estimate \eqref{eq:express:psi:zeta:squared}, we have
   \beqn
   \begin{split}
\E_{\bG^\star,p,u}\Big(\big\langle \bpsi(\sig_{\bom})\big\rangle_{\bG^\star}-\xi\Big)^2+O(n^{-1/3})&=\E\big\langle \FF(R_{\sig^1,\sig^2})\big\rangle_{\bG^\star}-\xi^2\\
&\geq \epsilon^\prime  \E\Big\langle \one\Big\{\FF(R_{\sig^1,\sig^2})\geq \xi^2+\epsilon^\prime \Big\}\Big\rangle_{\bG^\star}\\
&>c\epsilon^\prime \eta\,,
   \end{split}
   \eeqn
   where we used {\sf (MIN)} in the second inequality and the last inequality is by \eqref{eq:lem:MIN:tech} and \eqref{eq:lem:MIN:tech:2}. Therefore, taking $\delta=c\epsilon^\prime \eta$ completes the proof. 
\end{proof}
\begin{proof}[Proof of Proposition~\ref{prop:free:energy}]
    The statements $\frac{\partial}{\partial d}\E \log L(\bG^\star)\geq 0$ and $\E\log L(\bG^\star)\geq 0$ follows from Jensen's inequality. Indeed, for ease of notations, let $F_0(x)=x\log x$. Since $F_0(\cdot)$ is convex, 
    \beqn
    \E\log L(\bG^\star)=\E L(\bG)\log L(\bG)\geq F_0\big(\E L(\bG)\big)=F_0(1)=0\,,
    \eeqn
    where the first equality follows from change of measure (see also \eqref{eq:proof:mutual:info:tech}). Similarly, by Proposition~\ref{prop:derivative}, 
    \beqn
    \frac{\partial}{\partial d}\E \log L(\bG^\star)=\frac{1}{k}\E_{\bG^\star,p,u} F_0\Bigg(\bigg\langle \frac{\bpsi(\sig_{\bom})}{\E_{p,u}\big[\bpsi^\prime(\sig_{\bom^\prime})\big]}\bigg\rangle_{\bG^\star}\Bigg)\geq \frac{1}{k} F_0\Bigg(\E_{\bG^\star,p,u}\bigg\langle \frac{\bpsi(\sig_{\bom})}{\E_{p,u}\big[\bpsi^\prime(\sig_{\bom^\prime})\big]}\bigg\rangle_{\bG^\star}\Bigg)=0\,.
    \eeqn
    To prove the second statement, consider $d_0<d_{\ast}$. Note that for any $0<d\leq d_0$, we have by Proposition~\ref{prop:derivative} that
    \beq\label{eq:derivative:conv:zero}
    \frac{1}{n}\frac{\partial}{\partial d}\E \log L(\bG^\star)=\frac{1}{k}\cdot \E_{\bG^\star,p,u}F_0\Big(\xi^{-1}\big\langle \bpsi(\sig_{\bom})\big\rangle_{\bG^\star}\Big)+o_n(1)\to 0\quad\textnormal{as}\quad n\to\infty\,,
    \eeq
    where the convergence holds since $\xi^{-1}\langle \bpsi(\sig_{\bom})\rangle_{\bG^\star}\pto 0$ by Lemma~\ref{lem:psi:R} and $\big(\xi^{-1}\langle \bpsi(\sig_{\bom})\rangle_{\bG^\star}\big)_{n\geq 1}$ is bounded away from $0$ and $\infty$. Moreover, note that $\lim_{d\to 0}\E\log L(\bG^\star)=0$ holds for fixed $n\geq 1$ by Lemma~\ref{lem:L:deterministic}. Thus, by fundamental theorem of calculus, we have
    \beqn
    \frac{1}{n}\E \log L(\bG^\star)\Big\rvert_{d=d_0}=\int_{0}^{d_0} \frac{1}{n}\frac{\partial}{\partial d}\E \log L(\bG^\star)\de d\to 0\quad\textnormal{as}\quad n\to\infty\,,
    \eeqn
    where the convergence holds by \eqref{eq:derivative:conv:zero} and dominated convergence theorem, since the equality \eqref{eq:prop:derivative:first} in Proposition~\ref{prop:derivative} shows that $\frac{1}{n}\frac{\partial}{\partial d}\E \log L(\bG^\star)$ is uniformly bounded for $n\geq 1$ and $d\leq d_0$.

    Next, we prove the third statement. Assume that ${\sf (MIN)}$ holds and that $\E\big[A(\bsig^\star, \hat{\sig})\big]\geq \frac{1}{q}+\eps$ holds for some estimator $\hat{\sig}\equiv \hat{\sig}_n(\bG^\star)$ and $\eps>0$. Then, Lemma~\ref{lem:psi:R} yields that under this assumption, $\E_{\bG^\star,p,u}\big(\big\langle \bpsi(\sig_{\bom})\big\rangle_{\bG^\star}-\xi\big)^2>\delta$ holds for $n\geq n_0$ for some $n_0\equiv n_0(\eps,q,\pi)$ and $\delta\equiv \delta(\eps,q,\pi)>0$. Moreover, since $F_0(x)=x\log x$ is strongly convex with $F_0(1)=0, F_0^\prime(1)=1$, there exists a universal constant $c>0$ such that $F_0(x)\geq x-1+C(x-1)^2$ holds for $x>0$. Thus, for $n\geq n_0$,
    \beqn
    \begin{split}
   \E_{\bG^\star,p,u}F_0\Big(\xi^{-1}\big\langle \bpsi(\sig_{\bom})\big\rangle_{\bG^\star}\Big)
   &\geq \E_{\bG^\star,p,u}\big[\xi^{-1}\big\langle \bpsi(\sig_{\bom})\big\rangle_{\bG^\star}-1\big]+C\xi^{-2}\cdot \E_{\bG^\star,p,u}\big(\big\langle \bpsi(\sig_{\bom})\big\rangle_{\bG^\star}-\xi\big)^2\\
   &=C\xi^{-2}\cdot \E_{\bG^\star,p,u}\big(\big\langle \bpsi(\sig_{\bom})\big\rangle_{\bG^\star}-\xi\big)^2+O_{k,q,\Psi}(n^{-1/3})\\
   &\geq C\xi^{-2}\delta + O_{k,q,\Psi}(n^{-1/3})\,,
    \end{split}
    \eeqn
    where the equality follows from \eqref{eq:approx:e:psi}. Combining with Proposition~\ref{prop:derivative}, it follows that
    \beqn
    \frac{1}{n}\frac{\partial}{\partial d}\E \log L(\bG^\star)= k^{-1}\E_{\bG^\star,p,u}F_0\Big(\xi^{-1}\big\langle \bpsi(\sig_{\bom})\big\rangle_{\bG^\star}\Big)+O_{k,q,\Psi}(n^{-1/3}) \geq C\xi^{-2}k^{-1}\delta+O_{k,q,\Psi}(n^{-1/3})\,,
    \eeqn
    thus letting $\eta \equiv C\xi^{-2}k^{-1}\delta/2$ concludes the proof of the third statement.

    Finally, we prove the fourth statement under ${\sf (MIN)}$ condition. To this end, consider $d>d_{\ast}$. Then, there exists $d_0\in (d_{\ast},d)$, and since weak recovery is possible at $d_0$, there exists a subsequence $(n_{\ell})_{\ell \geq 1}$ and $\eps>0$ such that for any $\ell\geq 1$, $\E\big[A(\bsig^\star, \hat{\sig})\big]\geq \frac{1}{q}+\eps$ holds for some estimator $\hat{\sig}\equiv \hat{\sig}_{n_{\ell}}(\bG^\star_0)$ where $\bG^\star_0\sim \GG_{\sf plant}(n_{\ell},d_0,p,\pi)$. Observe that this implies that for any fixed $\ell\geq 1$ and $\Breve{d}\in (d_0,d)$, 
    \[
    \E\big[A(\bsig^\star, \hat{\sig})\big]\geq \frac{1}{q}+\eps\quad\textnormal{for some}\quad \hat{\sig}\equiv \hat{\sig}_{n_{\ell}}(\Breve{\bG}^\star)\quad\textnormal{where}\quad \Breve{\bG}^\star\sim \GG_{\sf plant}(n_{\ell},\Breve{d},p,\pi)
    \]
    since subsampling the clauses independently with probability $\Breve{d}/d$ from $\Breve{\bG}^\star\sim \GG_{\sf plant}(n_{\ell},\Breve{d},p,\pi)$ gives a sample drawn from $\GG_{\sf plant}(n_{\ell},d_0,p,\pi)$ by Poisson thinning. Hence, by the third statement which we established in the previous paragraph, there exists $\eta\equiv \eta(\eps)$ and $n_0\equiv n_0(\eps)$ such that for all $n_{\ell}\geq n_0$ and $\Breve{d}\in (d_0,d)$, $\frac{1}{n_{\ell}}\frac{\partial}{\partial d}\E \log L(\Breve{\bG}^\star)\geq \eta$ holds. Since we proved in the first statement that $\frac{\partial}{\partial d}\E\log L(\bG^\star)\geq 0$ holds in general, it follows that for all $\ell$ large enough so that $n_{\ell}\geq n_0$,
    \[
    \frac{1}{n}\E \log L(\bG^\star)\geq \eta(d-d_0)\quad\textnormal{where}\quad \bG^\star\sim \GG_{\sf plant}(n_{\ell},d,p,\pi)\,,
    \]
    which concludes the proof.
\end{proof}
 \section*{Acknowledgements}
 E.M. is supported by Simons-NSF collaboration on deep learning NSF DMS-2031883, Vannevar Bush Faculty Fellowship award ONR-N00014-20-1-2826, ARO MURI W911NF1910217 and a Simons Investigator Award in Mathematics (622132). A.S. is supported by NSF grants DMS-1855527, DMS-1749103, a Simons Investigator grant, and a MacArthur Fellowship.
\bibliographystyle{amsalpha}
\bibliography{all,my,sbm}

\appendix

\section{Nontriviality of the weak recovery threshold}
\label{sec:nontriviality}
In this section, we prove Proposition~\ref{prop:nontrivial:threshold}.  We begin with the case where ${\sf (SYM)}$ does not hold.   Then we must have some $\psi, s$ and $\tau,\tau'$ such that 
\[
\E_{\pi}[\psi(\bsig)\given \bsigma_s=\tau]>\E_{p,\pi}[\bpsi(\bsig)]> \E_{\pi}[\psi(\bsig)\given \bsigma_s=\tau'].
\]
For each vertex $v$ let $d_v^\psi$ be the number of clauses of type $\psi$ containing $v$. Define
\[
x_i:=\frac{dp(\psi) \sum_{\tau_1,\ldots,\tau_{k-1}}\psi(\tau_1,\ldots,\tau_{k-1},i)\prod_{j=1}^{j-1}\pi_{\tau_j}}{\E_{p,\pi}[\bpsi(\bsig)]}= \frac{dp(\psi) \E_{\pi}[\psi(\bsig)\given \bsigma_s=i]}{\E_{p,\pi}[\bpsi(\bsig)]},
\]
When $\bsigma_v=i$ we have that $d_v^\psi$ is Poisson with mean $(1+o(1))x_i$.  We take the very simple estimator
\[
\hat{\sigma}_v = \begin{cases}
\tau & d_v^\psi \geq dp(\psi)\\
\tau' & d_v^\psi \geq dp(\psi).
\end{cases}
\]
Then
\begin{align*}\label{eq:overlap:nontrivial}
\limsup_{n \to \infty} \E\left[A(\bsig^\star, \hat{\sig})\right] 
&= \frac{1}{q\pi_\tau}\P[\bsigma_v=\tau,d_v^\psi \geq dp(\psi)] + \frac{1}{q\pi_{\tau'}}\P[\bsigma_v=\tau',d_v^\psi < dp(\psi)]\\
&= \frac{1}{q}\big(\P[\hbox{Pois}(x_\tau) \geq dp(\psi)] - \P[\hbox{Pois}(x_{\tau'}) < dp(\psi)]\big)\\
&= \frac{1}{q} + \frac{1}{q}\big(\P[\hbox{Pois}(x_\tau) \geq dp(\psi)] - \P[\hbox{Pois}(x_{\tau'}) \geq dp(\psi)]\big) \geq \frac1{q}+\epsilon,
\end{align*}
since $x_\tau>x_{\tau'}$ which establishes weak recovery whenever $d>0$.

Now suppose that ${\sf (SYM)}$ holds and $d<\frac1{k-1}$.  In this case the graph is subcritical and almost all the variables are in trees of size $O(1)$.  We will consider the local weak limit of $(G,\bsig)$ under the planted model.  A vertex of type $i$ is connected to a Poisson with mean
\[
\sum_\psi \frac{dp(\psi) \E_{\pi}[\psi(\bsig)\given \bsigma_s=i]}{\E_{p,\pi}[\bpsi(\bsig)]} = d p(\psi)
\]
neighbouring clauses.  In particular, the distribution does not depend on the type $i$ and so the law of the local neighbourhood is a branching process independent of the state of the root.  Among all the trees in the graph with a fixed topology, a $\pi_i+o(1)$ will have state $i$ at the root.  Since the estimator cannot distinguish between trees with the same topology, it will be independent of $\bsig$ and so  $\E\left[A(\bsig^\star, \hat{\sig})\right] =\frac1q+o(1)$.  Hence weak recovery is impossible when $d<\frac1{k-1}$.

Finally, we show that for $d$ large that weak recovery is possible.  We define 
\[
\Psi(\bsig)=\sum_a \psi_a (\bsig_{\delta a})
\]
and take our estimator to be the maximizer
\[
\hat{\bsig} = \argmax \Psi(\hat{\bsig}).
\]
among $\hat{\sigma}$ with empirical distribution $\pi+o(1)$.  We will first show that when $d$ is large enough that the joint distribution of $(\bsig^\star,\hat{\bsig})$ is not close to the product measure.

Recall that we defined the map $\FF(R'):=\sum_{\sig, \utau\in [q]^k}\E_{p}[\bpsi(\sig)\bpsi(\utau)]\prod_{s=1}^{k}R'(\sigma_s,\tau_s)$  in \eqref{eq:def:FF}. Let $R=\pi\pi^{\sT}$ and $R_\star=\hbox{diag}(\pi_1,\ldots,\pi_q)$.  Then
\[
\FF(R_\star) - \FF(R)= \sum_{\sig, \utau\in [q]^k}\E_{p}[\bpsi(\sig)\bpsi(\utau)]\prod_{s=1}^{k}\pi(\sigma_s)\pi(\tau_s) -\sum_{\sig, \utau\in [q]^k}\E_{p}[\bpsi(\sig)\bpsi(\utau)]\prod_{s=1}^{k}\pi(\sigma_s)\pi(\tau_s)
\]
This can be expressed as $\sum_{\psi} p(\psi) \psi(X_1,\ldots,X_k)$ where $X_i$ are IID according to $\pi$ and since at least one of the $\psi$ is not constant we have that this variance is strictly positive and hence $\FF(R_\star) > \FF(R)$. 
Let $\bsig^\star$ be fixed and let $\utau\in [q]^V$ be a fixed configuration.  Define their joint empirical distribution as  
\[
W_{ij}=W_{ij}(\bsig^\star,\utau)=\frac1{n}|\{v\in V: \sigma^\star_v=i,\tau_v=j\}|.
\]
The probability that a clause of type $\psi$ connects variables $v_1\ldots,v_k$ is 
\[
\frac{d}{Z{n \choose k}}\psi(\sigma^\star_{v_1},\ldots, \sigma^\star_{v_k})
\]
for some normalizing constant $Z$ and so
\[
\E[\Psi(\utau)]= \frac{d}{Z{n \choose k}}\sum_{v_1,\ldots,v_k}\psi(\sigma^\star_{v_1},\ldots, \sigma^\star_{v_k}) \psi(\tau_{v_1},\ldots, \tau_{v_k}) = \frac{dn}{Z} \FF(W) + o(n).
\]
and
\[
\E[\Psi(\bsig^\star)]=  \frac{dn}{Z} \FF(R_\star) + o(n) .
\]
If we reveal each clause one by one, each step affects the value of $\Psi(\utau)$ by $O(1)$ so by the Azuma-Hoeffding Inequality,
\[
\P[|\Psi(\utau) - \E[\Psi(\utau) ]|\geq d\epsilon n ] \leq 2\exp(-C\epsilon^2 dn)
\]
By the continuity of $\FF$, for some $\delta>0$ sufficiently small we have that
\[
\inf_{R':\|R'-R\|\leq \delta} \FF(R_\star) - \FF(R') \geq \delta.
\]
Then
\begin{align*}
\P[\|W(\bsig^\star,\hat{\bsig})-R\|<\delta] &\leq \P\Big[|\Psi(\bsig) - \E[\Psi(\bsig) ]|\geq \frac{d\delta n}{3Z} \Big]+ \sum_{\utau:\|W(\bsig^\star,\utau)-R\|\leq \delta} \P\Big[|\Psi(\utau) - \E[\Psi(\utau) ]|\geq \frac{d\delta n}{3Z} \Big]\\
&\leq 2q^n \exp(-C(\delta/3Z)^2 dn) = o(1),
\end{align*}
for large enough $d$.  Hence, we have that with high probability $W(\bsig^\star,\hat{\bsig})$ is at least $\delta$ away from $\pi\pi^T$.  The perfomance of our estimator can be written as
\[
\limsup_{n \to \infty} \E\left[A(\bsig^\star, \hat{\sig})\right] = \limsup_{n \to \infty} \E \max_{\Gamma \in S_q}\Bigg\{
\frac{1}{q}\sum_{i=1}^{q} \frac{W_{i,\Gamma(i)}(\bsig^\star,\hat{\bsig})}{\pi_i }
\Bigg\}.
\]
Let us write
\[
\Upsilon(R')= \max_{\Gamma \in S_q}\Bigg\{
\frac{1}{q}\sum_{i=1}^{q} \frac{R'_{i,\Gamma(i)}}{\pi_i }
\Bigg\}
\]
Let
\[
\RR_\pi^\delta=\{R':\|R'-R\|\geq \delta, \sum_i R_{ij}=\pi_j, \sum_j R_{ij}=\pi_i\}
\]
Since w.h.p. we have that $W(\bsig^\star,\hat{\bsig})$ is $o(1)$ distance from $\RR_\pi^\delta$, to prove weak recovery it is enough to show that
\begin{equation}\label{e:Upsilon.bound}
\inf_{R'\in \mathcal{R}} \Upsilon(R') \geq \frac1{q}+\epsilon.
\end{equation}
Since $\mathcal{R}$ is compact it is enough to show that $\Upsilon(R') > \frac1{q}$ for all $R'\in\RR_\pi^\delta$.  First we have that for any $R'\in\RR_\pi$ that
\[
\Upsilon(R')\geq \frac1{q!}\sum_{\Gamma \in S_q} 
\frac{1}{q}\sum_{i=1}^{q} \frac{R'_{i,\Gamma(i)}}{\pi_i }
= \frac1{q^2} \sum_{i=1}^{q} \frac{\sum_{i'=1}^q R'_{i,i'}}{\pi_i } = \frac1{q}
\]
since the max is always at least the average and since every $\Gamma(i)$ is equal to $i'$ for a $\frac1{q}$ fraction of $\Gamma$.  So if $\Upsilon(R') = \frac1{q}$ then we must have
\[
\forall \ \Gamma: \ \ 
\frac{1}{q}\sum_{i=1}^{q} \frac{R'_{i,\Gamma(i)}}{\pi_i }
= \frac1{q}.
\]
In particular, switching two entries of $\Gamma$ cannot change the above expression so for any $i,i',j,j'$
\[
\frac{R_{ij}}{\pi_i} - \frac{R_{ij'}}{\pi_i} = \frac{R_{i'j}}{\pi_{i'}} - \frac{R_{i'j'}}{\pi_{i'}}.
\]
But this can only occur if all the rows of $R_{ij}/\pi_i$ are the same in which case $R'=\pi\pi^T$.  Hence,~\eqref{e:Upsilon.bound} holds and for large enough $d$ we have weak recovery.

\section{Hypothesis testing below the weak recovery threshold}
\label{sec:proof:hypothesis:estimation}
In this section, we prove Lemmas~\ref{lem:dist:L:infty}, \ref{lem:T:conv}, and \ref{lem:resampling}.

\subsection{Proof of Lemma~\ref{lem:dist:L:infty}}
We show that
\[
\boldsymbol{L}_{\infty}\equiv \prod_{\ell \geq 1}\prod_{\zeta\in S_{\ell}}\Big\{(1+\delta_{\zeta})^{X_{\zeta \infty}}e^{-\lambda_{\zeta}\delta_{\zeta}}\Big\}\,.
\]
is discrete if and only if  $d\in(0,\frac1{k-1})$.  The case where $d<\frac1{k-1}$ is straightforward as in this case
\[
\sum_{\ell \geq 1}\sum_{\zeta\in S_{\ell}} \P[X_{\zeta \infty}\geq 1] \leq \sum_{\ell \geq 1}\sum_{\zeta\in S_{\ell}} \E[X_{\zeta \infty}] < \infty
\]
and so by Borel-Cantelli, only finitely many $X_{\zeta \infty}$ are non-zero and so $\boldsymbol{L}_{\infty}$ takes values in the set given in equation~\eqref{eq:support.set} and hence is discrete.  When $d\geq \frac1{k-1}$ we have that $\boldsymbol{L}_{\infty}$ is the exponential of a weighted sum of Poisson random variables the sum of whose means is infinite  and so we can write
\[
\log (\boldsymbol{L}_{\infty}) = \lim_{T\to\infty}\int_0^T f(t) dN_t - g(T)
\]
where $f(t)$ and $g(t)$ are deterministic functions, $f(t)$ is strictly positive and $N_t$ is a standard Poisson process.  We will show that any such random variable must be continuous. Define
\[
R(K,t)=\max_{x_1<x_2<\ldots<x_k \in \R}\sum_{i=1}^K \P[\int_0^T f(t) dN_t = x_i],
\]
that is the maximum probability $\int_0^T f(t) dN_t$ puts on $k$ distinct points.  We claim that,
\begin{equation}\label{eq:atom.derivative}
\frac{d}{dt} R(K,t) \bigg|_{t=T} \leq - (R(K,T)- R(K-1,T)) + (R(K+1,T) - R(K,T)).
\end{equation}
To see this, suppose that $\mathcal{X}=\{x_1,x_2,\ldots,x_K\}$ is some maximizing set of $k$ atoms at time $T$.  Note that at time $t$ the size of the $j$-th largest atom is $R(j,t)- R(j-1,t)$.  Between time $t$ and $t+\delta t$ there is probability $\delta t$ of a new Poisson point which increases the integral by $f(t)$.  Thus we have that
\begin{align*}
\frac{d}{dt}\sum_{x\in \mathcal{X}}  \P[\int_0^T f(t) dN_t = x] \bigg|_{t=T} &= -\sum_{x\in \mathcal{X}}  (\P[\int_0^T f(t) dN_t = x] - \P[\int_0^T f(t) dN_t = x-f(t)])\\
&= -\sum_{x\in \mathcal{X}}  \P[\int_0^T f(t) dN_t = x] 
+ \sum_{y:y+f(T)\in \mathcal{X}} \P[\int_0^T f(t) dN_t = y]\\
&= -\sum_{x\in \mathcal{X}:x-f(T)\not\in \mathcal{X}}  \P[\int_0^T f(t) dN_t = x] 
+ \sum_{y\not\in \mathcal{X}:y+f(T)\in \mathcal{X}} \P[\int_0^T f(t) dN_t = y]\\
&\leq  -|\{x\in \mathcal{X}:x-f(T)\not\in \mathcal{X}\}|(R(K,T)- R(K-1,T)) \\
&\qquad + |\{y\not\in \mathcal{X}:y+f(T)\in \mathcal{X}\}|(R(K+1,T)- R(K,T))
\end{align*}
where the third equality holds by canceling out terms that appear in the first and second sums while the inequality holds by noting that every atom in $\mathcal{X}$ is one of the largest $K$ atoms while every while every atoms in $\mathcal{X}^c$ is at most the $(K+1)$-th biggest.  Note that 
\[
|\{x\in \mathcal{X}:x-f(T)\not\in \mathcal{X}\}|=|\{y\not\in \mathcal{X}:y+f(T)\in \mathcal{X}\}|
\]
and both are at least of size one, since $x_1$ is not in the former.  Hence we have equation~\eqref{eq:atom.derivative}.  Also since $R(j,t)- R(j-1,t)$ is the size of the $j$-th atoms it is decreasing in $j$ so we have in particular that
\[
\frac{d}{dt} R(K,t) \bigg|_{t=T} \leq 0.
\]
Suppose that $R(1,t)\to q>0$ as $t\to\infty$.  Then set $K=\lceil\frac1{q}\rceil$.  Then
\begin{align*}
\frac{d}{dt} \sum_{j=1}^K R(j,t) &\leq \sum_{j=1}^K - (R(j,t)- R(j-1,t)) + (R(j+1,t) - R(j,t))\\
&= - R(1,t) + (R(K+1,t) - R(K,t))\\
&\leq -q + \frac{1}{K+1}< 0.
\end{align*}
But $\sum_{j=1}^K R(j,0)=K$ we have that $\sum_{j=1}^K R(j,t)\leq K-t\big(q - (K+1)^{-1}\big)$  for all $t$ which contradicts the fact that it is non-negative.  Hence $R(1,t)\to 0$ as $t\to\infty$.  Now if $X$ and $Y$ are independent random variables,
\[
\max_{x}\P[X+Y=x] \leq \max_x \P[X=x]
\]
and so
\[
\max_x \P[\log (\boldsymbol{L}_{\infty})=x] \leq \inf_t R(1,t)=0.
\]
Hence $\boldsymbol{L}_{\infty}$ has a continuous distribution.
\subsection{Proof of Lemma~\ref{lem:T:conv}}
Fix $d<d_{\ks}$ and assume {\sf (SYM)}. We only prove the weak convergence
\beq\label{eq:goal:proof:lemma:T:conv}
T_n(\bG^\star)\dto \bL^\star_{\infty}\equiv \prod_{\ell\geq 1} \prod_{\zeta\in S_{\ell}}\Big\{(1+\delta_{\zeta})^{X^\star_{\zeta,\infty}}e^{-\la_{\zeta}\delta_{\zeta}}\Big\}\,,
\eeq
since the weak convergence for $T_n(\bG)$ can be established using the same argument. To start with, for $K\geq 1$, let us denote
\beqn
T_{n,K}(G):=\prod_{1\leq \ell \leq K}\prod_{\zeta\in S_{\ell}}\Big\{(1+\delta_{\zeta})^{X_{\zeta}(G)}e^{-\lambda_{\zeta}\delta_{\zeta}}\Big\}\,.
\eeqn
Thus, $T_n(\bG^\star)\equiv T_{n,K_n}(\bG^\star)$ holds by definition. Observe that by considering a fixed $K\geq 1$, we have by Fact~\ref{fact:cycle}-(2) that as $n\to\infty$,
\beqn
T_{n,K}(\bG^\star)\dto \bL^\star_{K}:=\prod_{\ell=1}^{K} \prod_{\zeta\in S_{\ell}}\Big\{(1+\delta_{\zeta})^{X^\star_{\zeta,\infty}}e^{-\la_{\zeta}\delta_{\zeta}}\Big\}\,.
\eeqn
Observe that since $\bL^\star_K$ converges a.s. to $\bL^\star_{\infty}$ as $K\to\infty$, we have $\bL^\star_K\dto \bL^\star_{\infty}$ holds. Thus, combining with the above weak convergence, a diagonal argument shows that there exists an arbitrarily slowly growing sequence $(K_n')_{n\geq 1}$ such that $1\ll K_n'\leq K_n$ holds and as $n\to\infty$,
\beqn
T_{n,K_n'}(\bG^\star)\dto \bL^\star_{\infty}\,.
\eeqn
Hence, in order to achieve our goal in Eq.~\eqref{eq:goal:proof:lemma:T:conv}, it suffices to prove that as $n\to\infty$,
\beq\label{eq:goal:second}
X_n:=\sum_{\ell=K_n'+1}^{K_n}\sum_{\zeta\in S_{\ell}}\Big\{\log(1+\delta_{\zeta}) X_{\zeta}(\bG^\star)-\la_{\zeta}\delta_{\zeta}\Big\}\pto 0\,,
\eeq
since $T_{n,K_n}(\bG^\star)/T_{n,K'_n}(\bG^\star)=e^{X_n}$ holds by definition. For the rest of the proof, we prove Eq.~\eqref{eq:goal:second} by showing that $\E X_n^2\to 0$ as $n\to\infty$.

To this end, we calculate the moments of $(X_{\zeta}(\bG^\star))_{\zeta\in \cup_{\ell\leq K_n}S_{\ell}}$. Since the length of $\zeta$ is $2\ell \leq K_n\ll \log n$, it is standard to approximate the the first and second moments of $X_{\zeta}(\bG^\star)$ (see e.g. \cite[Eq.~(8.7)]{CEJKK:18}) as follows. For two distinct signatures $\zeta,\zeta'\in\cup_{\ell\leq K_n}S_{\ell}$, we have
\beqn
\begin{split}
    &\E X_{\zeta}(\bG^\star)=\la^\star_{\zeta}+O_{k,q,\Psi}\bigg(\frac{\log n}{\sqrt{n}}\bigg)\,,\;\;\;\;\;\;\;\;\Var \big(X_{\zeta}(\bG^\star)\big)=\la^\star_{\zeta}+O_{k,q,\Psi}\bigg(\frac{\log n}{\sqrt{n}}\bigg)\,,\\
    &\E X_{\zeta}(\bG^\star) X_{\zeta'}(\bG^\star)=\la^\star_{\zeta}\la^\star_{\zeta'}+O_{k,q,\Psi}\bigg(\frac{\log n}{\sqrt{n}}\bigg)\,,
\end{split}
\eeqn
where the dominant error $\frac{\log n}{\sqrt{n}}$ comes from approximating the empirical distribution of $\bsig^\star$ by $\pi$. Moreover, note that the number of signatures $\zeta \in \cup_{\ell\leq K_n} S_{\ell}$ is at most $e^{O(\log\log n)}\ll n^{1/10}$ since $K_n=O(\log\log n)$. Thus, $\E X_n^2$ can be bounded by
\beq\label{eq:bound:sec:mo:X}
\begin{split}
\E X_n^2
&\leq \sum_{\ell=K_n'+1}^{K_n}\sum_{\zeta\in S_{\ell}}\Big\{\log^2(1+\delta_{\zeta})\la^\star_{\zeta}+\big( \log(1+\delta_{\zeta})\la^\star_{\zeta}-\la_{\zeta}\delta_{\zeta}\big)^2\Big\}\\
&~~~~~~~+\sum_{\ell,\ell'=K_n'+1}^{K_n}\sum_{\zeta\in S_{\ell}, \zeta'\in S_{\ell'}}\Big( \log(1+\delta_{\zeta})\la^\star_{\zeta}-\la_{\zeta}\delta_{\zeta}\Big)\Big( \log(1+\delta_{\zeta'})\la^\star_{\zeta'}-\la_{\zeta'}\delta_{\zeta'}\Big)+O_{k,q,\Psi}(n^{-1/3})\\
&\leq \sum_{\ell=K_n'+1}^{\infty}\sum_{\zeta\in S_{\ell}}\Big\{ (1+\delta_{\zeta})\log^2(1+\delta_{\zeta})\la_{\zeta}+\big((1+\delta_{\zeta})\log(1+\delta_{\zeta})-\delta_{\zeta}\big)^2\la_{\zeta}^2\Big\}\\
&~~~~~~~+\bigg(\sum_{\ell=K_n'+1}^{\infty}\sum_{\zeta\in S_{\ell}}\big((1+\delta_{\zeta})\log(1+\delta_{\zeta})-\delta_{\zeta}\big)\la_{\zeta}\bigg)^2+O_{k,q,\Psi}(n^{-1/3})\,,
\end{split}
\eeq
where the final inequality holds since $\la^\star_{\zeta}=\la_{\zeta}(1+\delta_{\zeta})$. We upper bound the RHS by using a taylor approximation w.r.t. $\delta_{\zeta}$. To do so, we first argue that $\sup_{\zeta\in S_{\ell}}|\delta_{\zeta}|\to 0$ as $\ell\to\infty$, which can be argued using the following lemma.
\begin{lemma}
\label{lem:MC:fundamental}
    Consider a triangular array of stochastic matrices $(P_{n,m})_{1\leq m\leq a_n, n\geq 1
    }\in \R^{q\times q}$, where $a_n\to \infty$ as $n\to\infty$, and it satisfies the following. For every $1\leq m\leq a_n$, $\pi$ is the stationary distribution of $P_{n,m}$, and there exists a constant $c>0$ such that the minimial element of $P_{n,m}$ is bounded below by $c$, i.e. $P_{n,m}(i,j)\geq c$ holds for $i,j\in [q]$. Then, the product $\prod_{m=1}^{a_n} P_{n,m}$ converges to $\bone\pi^{\sT}$ as $n\to\infty$.
\end{lemma}
\begin{proof}
    For $n\geq 1$ and $1\leq t \leq a_n$, consider the following distances:
    \beqn
\Delta_n(t):=\sup_{i \in [q]}\Big\|\Big(\prod_{m=1}^{t}P_{n,m}\Big)(i,\cdot)-\pi\Big\|_{\tv}\,,\;\;\;\;\;\;\;\wb{\Delta}_n(t):=\sup_{i,j \in [q]}\Big\|\Big(\prod_{m=1}^{t}P_{n,m}\Big)(i,\cdot)-\Big(\prod_{m=1}^{t}P_{n,m}\Big)(j,\cdot)\Big\|_{\tv}\,.
    \eeqn
    Then, since $\pi$ is the stationary distribution for $\prod_{m=1}^{t}P_{n,m}$, $\Delta_n(t)\leq \wb{\Delta}_n(t)$ holds (see e.g. proof of \cite[Lemma 4.10]{levin2008markov}). Moreover, a standard coupling argument (see e.g. \cite[Lemma 4.11]{levin2008markov}) shows that for any $n\geq 1$ and $1\leq s\leq s+t\leq a_n$, we have
    \beqn
    \begin{split}
    &\wb{\Delta}_n(t+s)\leq \wb{\Delta}_n(t)\wb{\Delta}_n(s;t)\,,\;\;\;\;\textnormal{where}\,,\\
    &\wb{\Delta}_n(s;t):=\sup_{i,j \in [q]}\Big\|\Big(\prod_{m=t+1}^{t+s}P_{n,m}\Big)(i,\cdot)-\Big(\prod_{m=t+1}^{t+s}P_{n,m}\Big)(j,\cdot)\Big\|_{\tv}\,.
    \end{split}
    \eeqn
    Finally, note that $\wb{\Delta}_n(1;t)\leq 1-c$ holds for any $0\leq t\leq a_n-1$ since every element of $P_{n,t+1}$ is bounded below by $c>0$. Thus, by the above submultiplicative property, we have that 
    \beqn
    \Delta_n(t)\leq \wb{\Delta}_n(t) \leq (1-c)^t\,.
    \eeqn
    By letting $t=a_n$, the inequality above establishes that $\prod_{m=1}^{a_n}P_{n,m}$ converges to $\bone\pi^{\sT}$. 
\end{proof}
To use Lemma~\ref{lem:MC:fundamental}, observe that under {\sf (SYM)}, the matrix $\Phi_{\zeta,s,t}\in \R^{q\times q}$ defined in \eqref{eq:Phi:s:t} is a stochastic matrix with stationary distribution $\pi$. Moreover, since we assumed that $\Phi$ is finite whose elements are positive weight functions, the elements of $\Phi_{\psi,s,t}$ are lower bounded by a constant $c_{\Psi}>0$. Recalling that $\Phi_{\zeta}:=\prod_{i=1}^{\ell}\Phi_{\psi_i,s_i,t_i}$ for a signature $\zeta=(\psi_1,s_1,t_1,\ldots, \psi_{\ell},s_{\ell},t_{\ell})$, it follows from Lemma~\ref{lem:MC:fundamental} that
\beqn
\lim_{\ell\to\infty}\sup_{\zeta\in S_{\ell}}\big|\delta_{\zeta}\big|\equiv \lim_{\ell\to\infty}\sup_{\zeta\in S_{\ell}}\big|\tr(\Phi_{\zeta})-1\big|=0\,.
\eeqn
Hence, the constants $(\delta_{\zeta})_{\zeta\in \cup_{\ell\geq 1}S_{\ell}}$ are bounded from $-1$ and $\infty$. Therefore, by using a taylor approximation for the term $\log(1+\delta_{\zeta})$ in the RHS of \eqref{eq:bound:sec:mo:X}, we can further bound
\beqn
\E X_n^2\leq C\Bigg(\sum_{\ell=K_n'+1}^{\infty}\sum_{\zeta\in S_{\ell}}\Big\{\la_{\zeta}\delta_{\zeta}^2+\la_{\zeta}^2\delta_{\zeta}^4\Big\}+\bigg(\sum_{\ell=K_n'+1}^{\infty}\sum_{\zeta\in S_{\ell}}\la_{\zeta}\delta_{\zeta}^2\bigg)^2+n^{-1/3}\Bigg)\,,
\eeqn
where $C>0$ only depends on $k,q,\Psi$. Since for $d<d_{\ks}$, $\sum_{\ell\geq 1}\sum_{\zeta\in S_{\ell}}\la_{\zeta}\delta_{\zeta}^2<\infty$ holds by Lemma~\ref{lem:lambda:equality}, the RHS tends to $0$ as $n\to\infty$. Therefore, $X_n\pto 0$, which concludes the proof of~\eqref{eq:goal:second}. 

\subsection{Proof of Lemma~\ref{lem:resampling}}
Let $\bG^\star_{\ga}(n,m,\bsig^\star)$ be the factor graph obtained from $\ga$ resampling procedure starting from $\bG^\star(n,m,\bsig^\star)$. By definitions of the planted model and the resampling procedures, the factor graph  $\bG^\star_{\ga}(n,m,\bsig^\star)$ is distributed as follows. Conditional on $\bsig^\star=\sig$, independently draw for each clause $a\in F$ the neighborhood $\delta a$ and the weight function $\psi_a$ from the distribution
\beqn
\P\big(\delta a= (v_1,\ldots, v_k)\,,\,\psi_a=\psi\big)=\frac{p(\psi)}{n^k}\bigg( \frac{(1-\ga)\psi(\sigma_{v_1},\ldots, \sigma_{v_k})}{\E_{p,u}\big[\bpsi(\sig_{\bom})\big]}+\frac{\ga}{\xi}\bigg)\,.
\eeqn
On the other hand, let $(\bG^\star)'(n,m,\bsig^\star)\equiv (\bG^\star)'(n,m,\bsig^\star, p_{\ga},\pi)$ be the planted model with $\ga$-modified weight functions defined in \eqref{eq:modified:weights}. Then, conditional on $\bsig^\star=\sig$, the neighborhood $\delta a$ and the weight function $\psi_a$ in $(\bG^\star)'(n,m,\sig)$ are independently drawn from a slightly different distribution 
\beqn
\P\big(\delta a= (v_1,\ldots, v_k)\,,\,\psi_a=\psi\big)=\frac{p(\psi)}{n^k}\cdot \frac{(1-\ga)\psi(\sigma_{v_1},\ldots, \sigma_{v_k})+\ga \xi}{\xi}\,.
\eeqn
Hence, given $\bsig^\star=\sig$, the likelihood ratio of $\bG^\star_{\ga}\equiv \bG^\star_{\ga}(n,\bm,\bsig^\star)$ and $(\bG^\star)'\equiv (\bG^\star)'(n,\bm,\bsig^\star)$ evaluated at a factor graph $G=(V,F,E,(\psi_a)_{a\in F})$ is given by
\beq\label{eq:LR:resampled}
\frac{\P\big(\bG^\star_{\ga}=G\bgiven\bsig^\star=\sig\big)}{\P\big((\bG^\star)'=G\bgiven \bsig^\star=\sig\big)}=\prod_{a\in F}\bigg(1+\frac{\ga\big(\E_{p,u}[\bpsi(\sig_{\bom})]-\xi\big)}{\ga \xi+(1-\ga)\psi_a(\sig_{\delta a})}\bigg)\,.
\eeq
Observe that under {\sf (SYM)}, Lemma~\ref{lem:main:approx} shows that $\E_{p,u}[\bpsi(\sig_{\bom})]-\xi=O_p(n^{-1})$. More precisely, for $\sig\in [q]^V$, define the vector $R_{\sig}\equiv \big(R_{\sig}(i)\big)_{i\leq q}\in \R^q$ by the empirical distribution of $\sig$:
\beq\label{eq:def:sig:empirical}
R_{\sig}(i):=\frac{1}{n}\sum_{v\in V}\one\{\sig_v=i\}\,.
\eeq
Then, by Lemma~\ref{lem:main:approx}, for any $\sig\in [q]^V$, we have the bound
\beqn
\Big|\E_{p,u}\big[\bpsi(\sig_{\bom})\big]-\xi\Big|\leq \frac{C}{n}\big\|\sqrt{n}(R_{\sig}-\pi)\big\|_1^2\,,
\eeqn
where $C\equiv C_{k,q,\Psi}>0$. Thus, on the event $\big\{\big\|\sqrt{n}(R_{\sig}-\pi)\big\|_1\leq K\big\}$, which happens with probability $1-o_K(1)$ under $\bsig^\star\sim \pi^{\otimes n}$ by the central limit theorem, we have that $\big|\E_{p,u}\big[\bpsi(\sig_{\bom})\big]-\xi\big|=O_{k,q,\Psi}\big(\frac{K^2}{n}\big)$ holds. Therefore, using this bound in the RHS of \eqref{eq:LR:resampled} shows that on the event $\big\{\big\|\sqrt{n}(R_{\sig}-\pi)\big\|_1\leq K\big\}$, we have
\beqn
\frac{\P\big(\bG^\star_{\ga}=G\bgiven\bsig^\star=\sig\big)}{\P\big((\bG^\star)'=G\bgiven \bsig^\star=\sig\big)}= \exp\Bigg(O_{k,q,\Psi}\bigg(\frac{K^2|F|}{n}\bigg)\Bigg)\,.
\eeqn
Since the number of clauses $\bm$ is Poisson with mean $dn/k$ for both $\bG^\star_{\ga}$ and $(\bG^\star)'$, we have $|F|\leq Cn$ holds with probability tending to $1$ under both models. Therefore, we conclude that $(\bG^\star_{\ga},\bsig^\star)$ and $((\bG^\star)',\bsig^\star)$ are contiguous (see e.g. \cite[Proposition 9.47]{Janson95random}).

\section{Applications to hypergraph stochastic block models}
\label{sec:appendix:HSBM}
In this section, we prove Theorems~\ref{thm:HSBM:contiguity}, \ref{thm:optimal:power:HSBM}, and \ref{thm:optimal:cycle:test:HSBM} by applying our results for the planted factor models. Throughout, we fix $k,q\geq 2$ and $\pi\in \PPP([q])$. Also, we recall that $\GGG_n$ denotes the event consisting of factor graphs that satisfy \ref{item:GGG:a} and \ref{item:GGG:b}. We denote by $(\wbG^{\star},\bsig^\star)$ the planted model $(\bG^\star,\bsig^\star)$ conditioned on the event $\GGG_n$. Then, $\wbG^\star$ is a $k$-uniform hypergraph with no multiple edges by viewing each clause as an hyperedge.

\subsection{Proof of Theorem~\ref{thm:optimal:power:HSBM} and Theorem~\ref{thm:optimal:cycle:test:HSBM}}
\label{subsec:HSBM:power}
First, we prove that $\bGH\sim \GG^{\sf H}(n,M,\pi)$ is mutually contiguous with respect to $\wbG^{\star}$ with a particular choice of weight function defined as follows. By viewing the symmetric tensor $M$ of order $k$ as a weight function $M:[q]^k\to\R_{+}$, let $p_M$ be the prior on the weight functions which puts all of its mass on $M$
\begin{equation}\label{eq:prior:weight:HSBM}
p_M=\delta_{M}\,.
\end{equation}
That is, there is a single weight function $\Psi\equiv \{M\}$. Then, we consider the planted model
\begin{equation*}
    \bG^\star_M\equiv \bG^\star(n,\bm, \bsig^\star,p_M,\pi)\,,\;\;\;\;\textnormal{where}\;\;\;\;\;\bm\sim \Poi(dn/k)\,,
\end{equation*}
where $d$ is the average degree of the HSBM defined in \eqref{eq:degree:HSBM}. We let $\wbG^{\star}_M$ be the $k$-uniform hypergraph obtained from $\bG^\star_M$ by conditioning on the event $\GGG_n$.

\begin{lemma}\label{lem:HSBM:planted:contiguity}
  Let $M$ be a symmetric tensor of order $k$ with positive entries such that the degree condition~\eqref{eq:degree:condition} holds for $M_0= M/d$. Then, the total variation distance between $(\bG^\star_{\HSBM},\bsig^\star)$ and $(\wbG^{\star}_M,\bsig^\star)$ tends to $0$ as $n\to\infty$:
\beqn
\lim_{n\to\infty}
\tv\Big(\big(\bG^\star_{\HSBM},\bsig^\star\big), \big(\wbG^{\star}_M,\bsig^\star\big)\Big)=0\,.
\eeqn
\end{lemma}
\begin{proof}
Let $Q(G,\sig):=\frac{\P(\wbG^\star_M=G,\bsig^\star=\sig)}{\P(\bG^\star_{\HSBM}=G,\bsig^\star=\sig)}$ be the likelihood ratio between $(\bG^\star_{\HSBM},\bsig^\star)$ and $(\wbG^{\star}_M,\bsig^\star)$. Then, it is standard to see that
\beqn
\tv\Big(\big(\bG^\star_{\HSBM},\bsig^\star\big), \big(\wbG^{\star}_M,\bsig^\star\big)\Big)=\E\Big[\big(1-Q(\bG^\star_{\HSBM},\bsig^\star)\big)\one\big\{Q(\bG^\star_{\HSBM},\bsig^\star)\leq 1\big\}\Big]\,.
\eeqn
Since the term inside the expectation in the right hand side is bounded, if we establish that
  \beq\label{eq:Q:conv}
  Q(\bG^\star_{\HSBM},\bsig^\star)\pto 1\,,
  \eeq
  then this implies our goal. Thus, we aim to show \eqref{eq:Q:conv} for the rest of the proof.

Given a factor graph $G=\big(V,F,E,(\psi_a)_{a\in F}\big)$ and a $k$-tuple of variables $\omega \equiv (v_1,\ldots, v_k)\in V^k$, let $F_{\omega}(G)$ be the set of $a\in F$ such that $\delta a=(v_1,\ldots v_k)$ and denote its size by $f_{\omega}(G)\equiv |F_{\omega}(G)|$. For $\bG^\star_M\equiv \bG^\star(n,\bm,\bsig^\star,p_M,\pi)$, we have by Poisson thinning that conditional on $\bsig^\star$, $\big(f_{\omega}(\bG^\star_M)\big)_{\omega\in V^k}$ are independent Poisson random variables with mean
\begin{equation*}
\E \Big[f_{\omega}(\bG^\star_M)\Bgiven \bsig^\star\Big]=\frac{dn}{k}\frac{M(\bsig^\star_{\omega})}{\sum_{\omega^\prime\in V^k}M(\bsig^\star_{\omega^\prime})}\,.
\end{equation*}
Observe that $G\in \GGG_n$ if and only if $f_{\omega}(G)=0$ for any $\omega=(v_1,\ldots v_k)$ such that $v_1,\ldots v_k$ are not distinct and $f_{\omega}(G)\leq 1$ holds for any $\omega\in V^k$. Thus, conditional on $\bsig^\star=\sig$, each possible hyperedge $(v_1,\ldots v_k)$ in $\wbG^{\star}_M$, where $v_1,\ldots,v_k$ are distinct, is included independently with probability $\frac{\phi(\sigma_{v_1},\ldots,\sigma_{v_k})}{1+\phi(\sigma_{v_1},\ldots,\sigma_{v_k})}$. Here, for $\ui=(i_1,\ldots,i_k)\in [q]^k$, the function $\phi(\ui)\equiv\phi_{\sig}(\ui)$ is defined by 
\begin{equation}\label{eq:def:phi}
    \phi(\ui):=\frac{d(k-1)!  M(\ui)}{\sum_{\omega^\prime\in V^k}M(\sig_{\omega^\prime})}\cdot n=\frac{d(k-1)!M(\ui) }{\sum_{i_1^\prime,\ldots,i_k^\prime}M(i_1^\prime,\ldots,i_k^\prime)\prod_{s=1}^{k}R_{\sig}(i^\prime_s)}n^{-(k-1)}\,,
\end{equation}
where in the last equality, $R_{\sig}\equiv \big(R_{\sig}(i)\big)_{1\leq i \leq  q}$ denotes the empirical distribution of $\sig$ (cf. \eqref{eq:def:sig:empirical}). Thus, if we let $E_{\ui}\equiv E_{\ui}(G)$ denote the number of hyperedges in $G$ such that the communities of the end points of the hyperedges are given by $\ui=(i_1,\ldots,i_k)$, then we have that
\beqn
\P\big(\wbG^{\star}_M=G\bgiven \bsig^\star=\sig\big)
    =\prod_{\ui=(i_1,\ldots, i_k)\in [q]^k}\bigg(\frac{\phi(\ui)}{1+\phi(\ui)}\bigg)^{E_{\ui}}\cdot \bigg(\frac{1}{1+\phi(\ui)}\bigg)^{n^k \prod_{s=1}^{k}R_{\sig}(i_s)-E_{\ui}}\,.
\eeqn
Hence, $Q(G,\sig)$ can be expressed by
\beq\label{eq:express:Q}
\begin{split}
Q(G,\sig)
&=\frac{\P\big(\wbG^{\star}_M=G\bgiven \bsig^\star=\sig\big)}{\P\big(\bGH=G \bgiven \bsig^\star=\sig\big)}\\
&= \prod_{\ui\in [q]^k}\bigg(\frac{\phi(\ui)\cdot \binom{n}{k-1}}{\big(1+\phi(\ui)\big)\cdot 
M(\ui)}\bigg)^{E_{\ui}}\cdot \Bigg(\Big(1+\phi(\ui)\Big)\bigg(1-\frac{M(\ui)}{\binom{n}{k-1}}\bigg)\Bigg)^{-n^k \prod_{s=1}^{k}R_{\sig}(i_s)+E_{\ui}}
\end{split}
\eeq
To this end, we estimate the RHS when $G=\bG^\star_{\HSBM}$ and $\sig=\bsig^\star$. Observe that $\phi(i_1,\ldots i_k)$ can be estimated as follows. If we let $Z_{\sig}:=\sqrt{n}(R_{\sig}-\pi)$, then $Z_{\bsig^\star}$ is asymptotically normal by the central limit theorem. In particular, $Z_{\bsig^\star}$ has $O(1)$ fluctuations. Moreover, since $M_0=M/d$ satisfies the condition~\eqref{eq:degree:condition}, we have
\beqn
\sum_{\ui=(i_1,\ldots i_k)\in [q]^k}M(i_1,\ldots,i_k)\prod_{s=1}^{k}\Big(\pi_{i_s}+n^{-1/2}Z_{\sig}(i_s)\Big)=d+O\bigg(\frac{\big\|Z_{\sig}\big\|_{\infty}^2\cdot\big(\frac{Z_{\sig}}{\sqrt{n}}\vee 1\big)^{k-2}}{n}\bigg)\,.
\eeqn
Thus, by plugging in the estimate above to \eqref{eq:def:phi}, we can approximate
\begin{equation}\label{eq:connectivity:prob}
\phi_{\sig}(\ui)=\bigg(1+\frac{h_{\sig}(\ui)}{n}\bigg)\cdot\frac{M(\ui)}{\binom{n}{k-1}}\,,\;\;\;\;\textnormal{where}\;\;\;\;\;h_{\bsig^\star}(i_1,\ldots,i_k)\leq \log n\;\;\;\;\textnormal{holds w.h.p.}
\end{equation}
Moreover, note that
\beq\label{eq:number:edges}
    E_{\ui}(\bG^\star_{\HSBM})=\frac{n^k M(\ui)\prod_{s=1}^{k}\pi_{i_s}}{\binom{n}{k-1}}+o_{\P}(n^{2/3})=n\cdot (k-1)! M(\ui)\prod_{s=1}^{k}\pi_{i_s}+o_{\P}(n^{2/3})\,,
\eeq
where $Y=o_{\P}(n^{2/3})$ denotes a term that $n^{-2/3}Y\pto 0$ holds under $(\bG^\star_{\HSBM}, \bsig^\star)$. Hence, by using the estimates \eqref{eq:connectivity:prob} and \eqref{eq:number:edges}, we have the approximation
\beqn
\bigg(\frac{\phi_{\bsig^\star}(\ui)\cdot \binom{n}{k-1}}{\big(1+\phi_{\bsig^\star}(\ui)\big)\cdot 
M(\ui)}\bigg)^{E_{\ui}(\bG^\star_{\HSBM})}=\exp\Bigg((k-1)!\Big(h_{\bsig^\star}(\ui)-\one\{k=2\}M(\ui)\Big)M(\ui)\prod_{s=1}^{k}\pi_{i_s}+o_{\P}(1)\Bigg)\,.
\eeqn
Similarly, the second term of the product in \eqref{eq:express:Q} can be approximated by
\beqn
\begin{split}
&\Bigg(\Big(1+\phi_{\bsig^\star}(\ui)\Big)\bigg(1-\frac{M(\ui)}{\binom{n}{k-1}}\bigg)\Bigg)^{-n^k \prod_{s=1}^{k}R_{\bsig^\star}(i_s)+E_{\ui}(\bG^\star_{\HSBM})}\\
&=\exp\Bigg(-(k-1)!\Big(h_{\bsig^\star}(\ui)-\one\{k=2\}M(\ui)\Big)M(\ui)\prod_{s=1}^{k}R_{\bsig^\star}(i_s)+o_{\P}(1)\Bigg)\\
&=\exp\Bigg(-(k-1)!\Big(h_{\bsig^\star}(\ui)-\one\{k=2\}M(\ui)\Big)M(\ui)\prod_{s=1}^{k}\pi_{i_s}+o_{\P}(1)\Bigg)\,,
\end{split}
\eeqn
where the last approximation holds since $R_{\bsig^\star}=\pi+o_{\P}(n^{-1/3})$. Therefore, the RHS in the two displays above exactly cancels out, and combing with \eqref{eq:express:Q} yields that $Q(\bG^\star_{\HSBM},\bsig^\star)\pto 1$.

\end{proof}
Let $\wt{d}_{\ast}(p_{M_0},\pi)$ denote the weak recovery threshold w.r.t. $\wbG^\star_{d\cdot M_0}$. That is, recalling the general definition of weak recovery in Definition~\ref{def:weak:recovery:general}, let
\beqn
\wt{d}_{\ast}(p_{M_0},\pi):=\inf\Big\{d>0: \textnormal{weak recovery is possible at $d$ for $\wbG^\star_{d\cdot M_0}$}\Big\}\,.
\eeqn
As a consequence of Lemma~\ref{lem:HSBM:planted:contiguity}, we have
\begin{equation}\label{eq:threshold:same}
d_{\ast}^{\sf H}(M_0,\pi)=\wt{d}_{\ast}(p_{M_0},\pi)\,.
\end{equation}
To transfer Theorem~\ref{thm:likelihood:conv} for planted factor models to HSBM, we need to further prove that the weak recovery threshold is unchanged after conditioning on $\GGG_n$.
\begin{lemma}\label{lem:identical:threshold:conditioning}
  Let $M_0$ be a symmetric tensor of order $k$ with positive entries such that the degree condition~\eqref{eq:degree:condition} is satisfied. Then, we have $\wt{d}_{\ast}(p_{M_0},\pi)=d_{\ast}(p_{M_0},\pi)$. and $d_{\ks}(p_{M_0},\pi)=d_{\ks}^{\sf H}(M_0,\pi)$.
\end{lemma}
\begin{proof}
The final assertion that $d_{\ks}(p_{M_0},\pi)=d_{\ks}^{\sf H}(M_0,\pi)$ is immediate from the definition of KS thresholds for HSBM and planted factor models stated in \eqref{eq:def:KS:HSBM} and Definition~\ref{def:KS:factor} respectively. For the rest of the proof, we aim to prove that $\wt{d}_{\ast}(p_{M_0},\pi)=d_{\ast}(p_{M_0},\pi)$ holds by showing that
\[
\limsup_n  \E\left[ \sup_{\hat{\sig}(\wbG^\star_{M})} A(\bsig^\star, \hat{\sig}) -  \frac{1}{q} \right] = 0 \quad \Leftrightarrow \quad \limsup_n  \E\left[ \sup_{\hat{\sig}(\bG^\star_{M})} A(\bsig^\star, \hat{\sig}) -  \frac{1}{q} \right] = 0.
\]
Note that $A(\bsig^\star, \hat{\sig}) -  \frac{1}{q}\in [0,1]$ and so the above equivalence would follow if $\wbG^\star_{M}$ and $\bG^\star_{M}$ were mutually contiguous since the events $\{A(\bsig^\star, \hat{\sig}) -  \frac{1}{q} \geq \epsilon\}$ would both either tend to 0 or not.  Unfortunately, while $\wbG^\star_{M}$ is contiguous with respect to $\bG^\star_{M}$ (which gives one direction of the equivalence) $\bG^\star_{M}$ is not contiguous with respect to $\wbG^\star_{M}$ because there is a constant probability of clauses that violate \ref{item:GGG:a} or \ref{item:GGG:b}.  Instead, we will add clauses to $\wbG^\star_{M}$, using no information about $\bsig^\star$ so that the resulting graph $\Breve{\bG}^\star_{M}$ is mutually contiguous with respect to $\bG^\star_{M}$.  Let
\[
\CC = \{\uj\in \Z^k : 1\leq j_1\leq j_2\leq \ldots\leq j_k \leq n\}, 
\]
be the set of increasing integer sequences of length $k$ between 1 and $n$.  For $\uj\in \CC$ let $N_{\uj}$ be the number of clauses in $\bG^\star_{M}$ with vertex set given by a permutation of $v_{j_1},v_{j_2},\ldots,v_{j_k}$  (counting multiplicity if there are repeated $j_i$).  By construction of $\bG^\star_{M}$, conditional on $\bsig^\star$, the $N_{\uj}$ are independent Poisson random variables with means $p_{\uj}$ where
\[
C_1^{-1} n^{-(k-1)} \leq p_{\uj} \leq C_1 n^{-(k-1)}
\]
for some $C_1>1$ which may depend on the model but not on $n$ or $\uj$.  We will let $\DD$ denote the set of $\uj$ with all $j_i$ distinct.  The $p_{\uj}$ satisfy
\begin{equation}\label{eq:puj.sum}
\sum_{\uj\in \CC} p_{\uj} \leq CN, \quad \sum_{\uj\in \CC\setminus \DD} p_{\uj} \leq C.
\end{equation}
Letting $\wN_{\uj}$ denote the number of clauses of $\wbG^\star_{M}$ with vertices $\uj$, the $\wN_{\uj}$ are conditionally independent with distribution Poisson with mean $p_{\uj}$ conditioned to be at most 1 if $\uj\in \DD$ and are equal to 0 if $\uj\in \CC\setminus\DD$.

Next let $N'_{\uj}$ be independent Poisson with mean $n^{-(k-1)}$ and define the transformed graph $\Breve{\bG}^\star_{M}$ by adding clauses such that
\[
\bN_{\uj} := \begin{cases}
    N'_{\uj} &\hbox{if } \wN_{\uj} \in \DD, N'_{\uj} \geq 2\\
    N'_{\uj} &\hbox{if } \wN_{\uj} \in \CC\setminus\DD\\
    \wN_{\uj} &\hbox{otherwise.}
\end{cases}
\]
Note that we do not use any information about $\bsig^\star$ to construct $\Breve{\bG}^\star_{M}$ from $\wbG^\star_{M}$.  For $\uj\in\DD$, the distribution of $\bN_{\uj}$ satisfies
\[
\frac{\P[\bN_{\uj} = \ell]}{\P[N_{\uj} = \ell]} = \begin{cases}
    \frac{\P[N'_{\uj} \leq 2]}{\P[N_{\uj} \leq 2]} &\hbox{if } \ell \leq 1\\
    \frac{n^{-(k-1)\ell e^{-n^{-(k-1)}}}}{p_{\uj}^\ell e^{-p_{\uj}}} &\hbox{if } \ell \geq 1\\
\end{cases}
\qquad  = \begin{cases}
    1+O(n^{-2(k-1)}) &\hbox{if } \ell \leq 1\\
    e^{O(\ell)} &\hbox{if } \ell \geq 2\\
\end{cases}
\]
Hence, if we write $\nu_{\uj}$ and $\bnu_{\uj}$ for the law of $N_{\uj}$ and $\bN_{\uj}$ respectively then
\[
\|\nu_{\uj} - \bnu_{\uj}\|^2_{L^2(\nu_{\uj})} = \sum_{\ell=0}^\infty \Big|\frac{\P[\bN_{\uj} = \ell]}{\P[N_{\uj} = \ell]} - 1 \Big|^2\P[N_{\uj} = \ell]\leq O(n^{-4(k-1)}) + \sum_{\ell=2}^\infty C^{2\ell} \frac{p_{\uj}^\ell e^{-p_{\uj}}}{\ell!}  = O(n^{-(k-1)}p_{\uj}),
\]
and similarly
\[
\|\nu_{\uj} - \bnu_{\uj}\|^2_{L^2(\bnu_{\uj})}  = O(n^{-(k-1)}p_{\uj}).
\]
For $\uj\in\CC\setminus\DD$,
\begin{align*}
\|\nu_{\uj} - \bnu_{\uj}\|^2_{L^2(\nu_{\uj})} 
&= \sum_{\ell=0}^\infty \Big|\frac{n^{-(k-1)\ell}e^{-n^{-(k-1)}}}{p_{\uj}^\ell e^{-p_{\uj}^\ell}} - 1 \Big|^2 \frac{p_{\uj}^\ell e^{-p_{\uj}}}{\ell!}\\
&\leq O(n^{-2(k-1)}) + \sum_{\ell=1}^\infty C^{2\ell} \frac{p_{\uj}^\ell e^{-p_{\uj}}}{\ell!}  \\
&= O(n^{-(k-1)}) = O(p_{\uj}),
\end{align*}
and similarly
\[
\|\nu_{\uj} - \bnu_{\uj}\|^2_{L^2(\bnu_{\uj})} =O(p_{\uj}).
\]
Setting $\nu$ and $\bnu$ for the law of the vectors $\{N_{\uj}\}_{\uj\in\CC}$ and $\{\bN_{\uj}\}_{\uj\in\CC}$ respectively, since both are product measures,
\begin{align*}
1+\|\nu - \bnu\|^2_{L^2(\bnu)} 
&= \prod_{\uj\in\CC}\big(1+ \|\nu_{\uj} - \bnu_{\uj}\|^2_{L^2(\bnu_{\uj})}\big)\\
&\leq \exp\bigg(Cn^{-(k-1)}\sum_{\uj\in\DD} p_{\uj} + C \sum_{\uj\in\CC\setminus \DD} p_{\uj} \bigg)\\
&\leq O(1),
\end{align*}
where the last inequality used equation~\eqref{eq:puj.sum}.  Similarly
\[
\|\wnu - \bnu\|^2_{L^2(\wnu)} = O(1).
\]
It follows that $\{N_{\uj}\}_{\uj\in\CC}$ and $\{\bN_{\uj}\}_{\uj\in\CC}$ are mutually contiguous and hence $\Breve{\bG}^\star_{M}$ is mutually contiguous with respect to $\bG^\star_{M}$. Hence we have that $\wt{d}_{\ast}(p_{M_0},\pi)=d_{\ast}(p_{M_0},\pi)$.

\end{proof}

Let $\wbG$ denote the null model $\bG\sim \GG_{\sf null}(n,d,p_{M})$ conditioned on the event $\GGG_n$. Here, we drop the subscript $M$ since the factors $(\psi_a)_{a\in F}\equiv M$ in $\wbG$ (and $\bG$) do not play any role. For a factor graph $G$, let 
\begin{equation}\label{eq:def:L:tilde}
\wt{L}(G):= \frac{\P(\wbG^\star_M=G)}{\P(\wbG=G)}=\frac{L(G)\one\{G\in \GGG_n\}}{\E[L(\bG)\given \bG\in \GGG_n]}\,,
\end{equation}
where the last equality holds by definition of $L(G)$ (see Eq.~\eqref{eq:proof:mutual:info:tech}).
Then, the following proposition plays an important role in the proof of Theorem~\ref{thm:optimal:power:HSBM}. Recall the constants $(\alpha_{\ell})_{\ell\geq 2}$ defined in Eq.~\eqref{eq:def:alpha}, which are functions of $M_0$ and $\pi$.
\begin{prop}\label{prop:L:tilde:conv}
    Let $M_0$ be a symmetric tensor of order $k$ with positive entries. Below the weak revery and KS thresholds $d<d_{\ast}(p_{M_0},\pi)\wedge d_{\ks}(p_{M_0},\pi)$, we have that as $n\to\infty$,
    \begin{equation}\label{eq:prop:L:tilde:conv}
    \wt{L}(\wbG)\dto \boldsymbol{\LL}_{\infty}
    :=\prod_{\ell=2+\one\{k=2\}}^{\infty}\frac{(1+\alpha_{\ell})^{\bX_{\ell}}}{\E(1+\alpha_{\ell})^{\bX_{\ell}}}\,,
\end{equation}
where $(\bX_{\ell})_{\ell\geq 2}$ are independent Poisson random variables with mean $\E \bX_{\ell}=\frac{1}{2\ell}\big((k-1)d\big)^{\ell}$.
\end{prop}
\begin{proof}
The proof follows from a combination of Fact~\ref{fact:cycle}-(3) Theorem~\ref{thm:likelihood:conv}. For a signature $\zeta \in S_{\ell}$, if the prior on the weight functions $p$ puts all of its mass on $M$, then the constants $\lambda_{\zeta}$ and $\la^\star_{\zeta}$ in \eqref{def:lambda:zeta} equal
\[
\lambda_{\zeta}=\frac{1}{2\ell}\left(\frac{d}{k}\right)^{\ell}\,,\quad\quad\la^\star_{\zeta}=\frac{1}{2\ell}\tr(B^{\ell})\,,
\]
where $B$ is the matrix defined in \eqref{eq:def:B}. Thus, $\la_{\zeta},\la^\star_{\zeta}$ depends on $\zeta$ on through its order $\ell$ and
\[
\delta_{\zeta}\equiv \la^\star_{\zeta}/\la_{\zeta}-1=\alpha_{\ell}\,.
\]
Consequently, Theorem~\ref{thm:likelihood:conv} yields that for $d<d_{\ast}(p_{M_0},\pi)\wedge d_{\ks}(p_{M_0},\pi)$,
\[
L(\bG)\dto \bL_{\infty}=\prod_{\ell\geq 1}\frac{(1+\alpha_{\ell})^{\bX_{\ell}}}{\E(1+\alpha_{\ell})^{\bX_{\ell}}}\,,
\]
where this convergence holds jointly with the weak convergence of the number of cycles of given length in Fact~\ref{fact:cycle}-(1).
Meanwhile by Fact~\ref{fact:cycle}-(3), $\P(\bG\in \GGG_n)=\P(C(\bG)=0)+O(1/n)$ and $\P(\bG^\star\in \GGG_n)=\P(C(\bG^\star)=0)+O(1/n)$ hold, where $C(G)$ is the number of self-loops if $k\geq 3$ and the sum of the number of self-loops and the number of cycles of length $2$ if $k=2$. Hence, it follows that $L(\wbG)$ weakly converges to the distribution of $\bL_{\infty}$ conditional on the event $\big\{\bX_1+\one\{k=2\}\bX_2=0\big\}$. Moreover,  $\wt{L}(G)$ and $L(G)$ just differs by a multiplicative factor for $G\in \GGG_n$ by~\eqref{eq:def:L:tilde} where $\wt{L}(\wt{\bG})$ satisfies the normalization $\E \wt{L}(\wt{\bG})=1$. Therefore, the desired claim~\eqref{eq:prop:L:tilde:conv} follows.
\end{proof}
\begin{proof}[Proof of Theorem~\ref{thm:optimal:power:HSBM}]
Consider $d<d_{\ast}^{\sf H}(M_0,\pi)\wedge d_{\ks}^{\sf H}(M_0,\pi)$. Recalling \eqref{eq:threshold:same}, we have $d_{\ast}^{\sf H}(M_0,\pi)=\wt{d}_{\ast}(p_{M_0},\pi)$. Moreover, $\wt{d}_{\ast}(p_{M_0},\pi)=d_{\ast}(p_{M_0},\pi)$ and $d_{\ks}^{\sf H}(M_0,\pi)=d_{\ks}^{\sf H}(p_{M_0},\pi)$ hold by Lemma~\ref{lem:identical:threshold:conditioning}. Hence, Proposition~\ref{prop:L:tilde:conv} yields $\wt{L}(\wbG)\dto \boldsymbol{\LL}_{\infty}$. Note that $\wt{L}(G)$ and $\LL_n(G)$ are related by the chain rule
\begin{equation}\label{eq:chain:rule}
\LL_n(G)=\wt{L}(G)\cdot \frac{\P(\wbG=G)}{\P(\bG_{\ER}=G)}\cdot \frac{\P(\bG^\star_{\HSBM}=G)}{\P(\wbG^\star_M=G)}\equiv \wt{L}(G)\frac{f_1(G)}{f_2(G)}\,,
\end{equation}
where
\[
f_1(G):= \frac{\P(\wbG=G)}{\P(\bG_{\ER}=G)}\,,\quad\quad f_2(G):=\frac{\P(\bG^\star_{\HSBM}=G)}{\P(\wbG^\star_M=G)}\,.
\]
Note that by considering $\pr=\bone_{k,q}$, the all-$1$-tensor in Lemma~\ref{lem:HSBM:planted:contiguity}, we have $f_1(\wbG)\pto 1$. Clearly, $f_2(\wbG^\star_M)\pto 1$ follows from Lemma~\ref{lem:HSBM:planted:contiguity}. Note that $\wbG^\star_M$ is mutually contiguous with $\wbG$ by Proposition~\ref{prop:L:tilde:conv}, thus $f_2(\wbG)\pto 1$ also holds. Therefore, by combining the established convergences $\wt{L}(\wbG)\dto \boldsymbol{\LL}_{\infty}$, $f_i(\wbG)\pto 1$ for $i=1,2$, we have $\LL_n(\wbG)\dto \boldsymbol{\LL}_{\infty}$ by \eqref{eq:chain:rule}. Finally, the stated properties of $\boldsymbol{\LL}_{\infty}$ are a special case of Lemma~\ref{lem:dist:L:infty} and Corollary~\ref{cor:power:factor}. 
\end{proof}
\begin{proof}[Proof of Theorem~\ref{thm:optimal:cycle:test:HSBM}]
    This is a special case of Lemma~\ref{lem:T:conv} and Corollary~\ref{cor:cycle:test:factor}.
\end{proof}

\subsection{Proof of Theorem~\ref{thm:HSBM:contiguity}}
The following lemma relates the mutual information and the free energy for HSBMs, which is the analog of Lemma~\ref{lem:KL:mutual:info:free:energy} for planted factor models.
\begin{lemma}\label{lem:HSBM:mutual:info:relation}
    Let $M_0$ be a symmetric tensor of order $k$ with positive entries such that the normalization~\eqref{eq:normalization} is satisfied. Then, the normalized mutual information can be approximated by
    \beqn
    \frac{1}{n}I(\bG^\star_{\HSBM},\bsig^\star)+\frac{1}{n}\E\log \LL_n(\bG^\star_{\HSBM})=\frac{d}{k}\sum_{i_1,\ldots i_k=1}^{q}\pr(i_1,\ldots, i_k)\log\big(\pr(i_1,\ldots,i_k)\big) \prod_{s=1}^{k}\pi_{i_s}+o_n(1)\,.
    \eeqn
\end{lemma}
\begin{proof}
    By definition of mutual information and the likelihood ratio $\LL_n(G)$, we have the identity
    \beq\label{eq:mutual:info:identity}
    \begin{split}
    &\frac{1}{n}I(\bG^\star_{\HSBM},\bsig^\star)\\
    &=-\frac{1}{n}\E\log \LL_n(\bG^\star_{\HSBM})+\frac{1}{n}\sum_{G,\sig}\P(\bG^\star_{\HSBM}=G,\bsig^\star=\sig)\log\frac{\P(\bG^\star_{\HSBM}=G\given \bsig^\star=\sig)}{\P(\bG_{\ER}=G)}\,.
    \end{split}
    \eeq
    Note that for any $G$ and $\sig$, we can compute 
    \beqn
    \frac{\P(\bG^\star_{\HSBM}=G\given \bsig^\star=\sig)}{\P(\bG_{\ER}=G)}=\prod_{(v_1,\ldots, v_k)\in E(G)}\frac{
\frac{M(\sigma_{v_1},\ldots, \sigma_{v_k})}{\binom{n}{k-1}}}{\frac{d}{\binom{n}{k-1}}}\prod_{(v_1,\ldots, v_k)\notin E(G)}\frac{1-
\frac{M(\sigma_{v_1},\ldots, \sigma_{v_k})}{\binom{n}{k-1}}}{1-\frac{d}{\binom{n}{k-1}}}\,,
    \eeqn
    where $E(G)$ denote the set of hyperedges in $G$. Recall that $M=dM_0$, so taking logarithm and conditional expectation w.r.t. $\bG^\star_{\HSBM}$ conditional on $\bsig^\star=\sig$ in the RHS gives
    \beqn
    \begin{split}
&\sum_{G}\P(\bG^\star_{\HSBM}=G\given\bsig^\star=\sig)\log\frac{\P(\bG^\star_{\HSBM}=G\given \bsig^\star=\sig)}{\P(\bG_{\ER}=G)}\\
&=\sum_{(v_1,\ldots v_k)}\frac{dM_0(\sigma_{v_1},\ldots, \sigma_{v_k})}{\binom{n}{k-1}}\log M_0(\sigma_{v_1},\ldots, \sigma_{v_k})+\bigg(1-
\frac{dM_0(\sigma_{v_1},\ldots, \sigma_{v_k})}{\binom{n}{k-1}}\bigg)\log\Bigg(\frac{1-
\frac{dM_0(\sigma_{v_1},\ldots, \sigma_{v_k})}{\binom{n}{k-1}}}{1-\frac{d}{\binom{n}{k-1}}}\Bigg)\,,
\end{split}
    \eeqn
    where the sum is over distinct $k$ vertices $(v_1,\ldots v_k)$. Taking expectation over $\bsig^\star$ yields
    \beqn
    \begin{split}
&\sum_{G,\sig}\P(\bG^\star_{\HSBM}=G,\bsig^\star=\sig)\log\frac{\P(\bG^\star_{\HSBM}=G\given \bsig^\star=\sig)}{\P(\bG_{\ER}=G)}\\
&=\sum_{\ui=(i_1,\ldots i_k)\in[q]^{k}}\binom{n}{k}\cdot \Bigg(\frac{dM_0(\ui)}{\binom{n}{k-1}}\log M_0(\ui)+\bigg(1-
\frac{dM_0(\ui)}{\binom{n}{k-1}}\bigg)\log\Bigg(\frac{1-
\frac{dM_0(\ui)}{\binom{n}{k-1}}}{1-\frac{d}{\binom{n}{k-1}}}\Bigg)\prod_{s=1}^{k}\pi_{i_s}\\
    &=\frac{dn}{k}\Bigg(\sum_{\ui\in[q]^{k}}M_0(\ui)\log M_0(\ui)\prod_{s=1}^{k}\pi_{i_s}+\sum_{\ui\in [q]^k}\big(1-M_0(\ui)\big)\prod_{s=1}^{k}\pi_{i_s}+o_n(1)\Bigg)\\
    &=\frac{dn}{k}\sum_{\ui\in[q]^{k}}M_0(\ui)\log M_0(\ui)\prod_{s=1}^{k}\pi_{i_s}+o_n(1)\,,
    \end{split}
    \eeqn
    where we used $\sum_{\ui\in [q]^k}\big(1-M_0(\ui)\big)\prod_{s=1}^{k}\pi_{i_s}=0$ (cf.~\eqref{eq:normalization}) in the last equality. Combining this with \eqref{eq:mutual:info:identity} concludes the proof.
\end{proof}
We next prove that the free energy is bounded away from $0$ along a subsequence by transferring Proposition~\ref{prop:free:energy} for the planted factor models to HSBM.
\begin{prop}\label{prop:HSBM:free:energy}
For any $n\geq 1$ and $d>0$, we have $\E \log \LL_n(\bG^\star_{\HSBM})\geq 0$. Moreover, if $d>d_{\ast}(p_{M_0},\pi)$ and the condition~${\sf (MIN)}$ holds, then there exists $\eta>0$ such that
\beq\label{eq:prop:HSBM:free:energy}
\limsup_{n\to\infty} \frac{1}{n}\E\log \LL_n(\bG^\star_{\HSBM})\geq \eta\,.
\eeq
\begin{proof}
Letting $F_0(x)=x\log x$, we have by a change of measure that
\beqn
\E \log \LL_n(\bG^\star_{\HSBM})=\E\LL_n(\bG_{\ER})\log \LL_n(\bG_{\ER})\geq F_0(\E \LL_n(\bG_{\ER}))=0\,,
\eeqn
where the inequality follows from Jensen's inequality and the last equality holds by $\E \LL_n(\bG_{\ER})=1$.

Next, suppose that $d>d_{\ast}(p_{M_0},\pi)$ and the condition {\sf (MIN)} holds. Assume by contraction that \eqref{eq:prop:HSBM:free:energy} does not hold. Then, since $\E \log \LL_n(\bG^\star_{\HSBM})\geq 0$, we must have that
\beqn
\lim_{n\to\infty} \frac{1}{n}\E \log \LL_n(\bG^\star_{\HSBM})=0\,.
\eeqn
Now, recall the chain rule $\LL_n(G)=\wt{L}(G)f_1(G)/f_2(G)$ in \eqref{eq:chain:rule}, where we defined $f_1(G)\equiv \frac{\P(\wbG=G)}{\P(\bG_{\ER}=G)}$ and $f_2(G)\equiv \frac{\P(\bG^\star_{\HSBM}=G)}{\P(\wbG^\star_M=G)}$. Then, note that for $i=1,2$,
\[
\lim_{n\to\infty}\frac{1}{n}\E \log f_i(\bG^\star_{\HSBM})=0\,.
\]
Indeed, $i=2$ case holds since $\bG^\star_{\HSBM}$ and $\wbG^\star_M$ are mutually contiguous by Lemma~\ref{lem:HSBM:planted:contiguity} and $n^{-1}\log f_2(\cdot)$ is bounded. For $i=1$, $n^{-1}\log f_1(G)$ can be calculated directly since $\wbG$ and $\bG_{\ER}$ doesn't have planted structure, from which it can be seen that for any sparse random graph $\Breve{\bG}$ with at most linear number of edges in expectation, $n^{-1}\E\log f_1(\Breve{\bG})=o_n(1)$. Therefore, it follows that 
\[
\lim_{n\to\infty}\frac{1}{n}\E \log \wt{L}(\bG^\star_{\HSBM})=0\,.
\]
Since $\tv(\bG^\star_{\HSBM}, \wbG^\star_M)=o_n(1)$ holds by Lemma~\ref{lem:HSBM:planted:contiguity} and $n^{-1}\log \wt{L}(G)=O(1)$ holds for factor graph $G$ with at most $O(n)$ clauses, it follows that 
\[
\lim_{n\to\infty}\frac{1}{n}\E \log \wt{L}(\wbG^\star_M)=0\,.
\]
Note that by Azuma Hoeffding's inequality, $n^{-1}\log \wt{L}(\wbG^\star_M)$ and $n^{-1}\log L(\bG^\star_M)$ concentrate tightly around their expectation with $O(n^{-1/2})$ fluctuation, and that $\bG^\star_M\in \GGG_n$ holds with uniformly positive probability. Thus, $n^{-1}\E \log \wt{L}(\wbG^\star_M)=n^{-1}\E \log L(\bG^\star_M)+O(n^{-1/2})$ holds. Consequently, we have that
\[
\lim_{n\to\infty}\frac{1}{n}\E \log L(\bG^\star_M)=0\,,
\]
which contradicts Proposition~\ref{prop:free:energy}-(3).
\end{proof}
\end{prop}
\begin{proof}[Proof of Theorem~\ref{thm:HSBM:contiguity}]
    The first statement regarding the mutual contiguity follows immediately from Theorem~\ref{thm:optimal:power:HSBM} and Le Cam's first lemma (see e.g.~\cite[Lemma 6.4]{Vaart98} or \cite[Proposition 9.49]{Janson00random}) since $\boldsymbol{\LL}_{\infty}>0$ a.s. and $\E \boldsymbol{\LL}_{\infty}=1$. The mutual contiguity between $\bG^\star_{\HSBM}$ and $\bG_{\ER}$ implies that
    \[
    \lim_{n\to\infty} \frac{1}{n}\E \log \LL_n(\bG^\star_{\HSBM})=0\,.
    \]
    Combining with Lemma~\ref{lem:HSBM:mutual:info:relation} finishes the proof of the second statement. For the third statement, recall that $d_{\ast}^{\sf H}(M_0,\pi)=\wt{d}_{\ast}(p_{M_0},\pi)$ holds (cf. \eqref{eq:threshold:same}) and that $\wt{d}_{\ast}(p_{M_0},\pi)=d_{\ast}(p_{M_0},\pi)$ holds by Lemma~\ref{lem:identical:threshold:conditioning}. Thus, for $d>d_{\ast}^{\sf H}(M_0,\pi)$, Proposition~\ref{prop:HSBM:free:energy} yields that there exists a constant $\eta>0$ such that $n^{-1}\E \log \LL_n(\bG^\star_{\HSBM})\geq \eta>0$ holds along a subsequence $(n_{\ell})_{\ell\geq 1}$. Then, consider the event
    \[
    \AAA^\ast_n=\{G: \LL_n(G)\geq e^{\eta n/2}\}\,.
    \]
    Then, the same argument as in the proof of Theorem~\ref{thm:mutual:info} shows that $\P(\bG_{\ER}\in \AAA^\ast_{n_{\ell}})\leq e^{-\eta^\prime n}$ and $\P(\bG^\star_{\HSBM}\in \AAA^\ast_{n_{\ell}})\geq 1-e^{-\eta^\prime n}$ for some constant $\eta^\prime>0$. Finally, the fourth statement is immediate from Lemma~\ref{lem:HSBM:mutual:info:relation} and Proposition~\ref{prop:HSBM:free:energy}, which concludes the proof.
\end{proof}

\end{document}